\documentclass[12pt, a4paper]{amsart}
\usepackage{amscd, amssymb, pb-diagram}
\usepackage{mathtools}
\usepackage{mathrsfs}
\usepackage{enumerate}
\usepackage[shortlabels]{enumitem}
\usepackage{amsmath}
\usepackage[margin=2.9cm]{geometry}
\usepackage{amscd, amssymb, pb-diagram}
\usepackage{tikz-cd}
\usepackage{amsthm} 
\usepackage[nospace,noadjust]{cite}
\usepackage{lmodern}
\usepackage{graphicx}
\graphicspath{{images/}{../images/}}
\usepackage[toc,page]{appendix}
\usepackage[normalem]{ulem}
\usepackage{emptypage}
\usepackage{hyperref}
\usepackage{subfiles}
\usepackage{datetime}

\usepackage{tensor}
\usepackage{wrapfig}
\usepackage{tikz}
\usepackage{multicol}
\usepackage[autostyle]{csquotes}

\allowdisplaybreaks

\def\diag{\operatorname{diag}}

\def\sp{\operatorname{span}}
\def\supp{\operatorname{supp}}
\def\dim{\operatorname{dim}}
\def\id{\operatorname{id}}

\def\sup{\operatorname{sup}}
\def\Eta{\operatorname{H}}
\def\max{\operatorname{max}}

\def\ent{\operatorname{h}}

\def\cl{\operatorname{cl}}

\def\supp{\operatorname{supp}}
\def\Lip{\operatorname{Lip}}

\def\Kt{\operatorname{K}}
\def\KKt{\operatorname{KK}}
\def\Ext{\operatorname{Ext}}

\def\rank{\operatorname{rank}}
\def\st{\operatorname{st}}
\newcommand{\ep}{\varepsilon}

\newtheorem{thm}{Theorem}[section]
\newtheorem{cor}[thm]{Corollary}

\newtheorem{lemma}[thm]{Lemma}
\newtheorem{prop}[thm]{Proposition}

\newtheorem*{thm*}{Main Theorem}

\theoremstyle{definition}
\newtheorem{definition}[thm]{Definition}

\theoremstyle{remark}
\newtheorem{remark}[thm]{Remark}

\newtheorem{example}[thm]{Example}

\newtheorem*{Acknowledgements}{Acknowledgements}

\numberwithin{equation}{section}

\newcounter{nameOfYourChoice}
\newcommand{\vertiii}[1]{{\left\vert\kern-0.25ex\left\vert\kern-0.25ex\left\vert #1 
\right\vert\kern-0.25ex\right\vert\kern-0.25ex\right\vert}}

\tikzstyle{vertex}=[circle]
\tikzstyle{goto}=[->,shorten >=1pt,>=stealth,semithick]

\makeatletter
\@namedef{subjclassname@2020}{%
  \textup{2020} Mathematics Subject Classification}
\makeatother

\begin{document}
\setcounter{section}{0}
%\date{}
\title{On finitely summable Fredholm modules from Smale spaces}
\author[Dimitris Michail Gerontogiannis]{Dimitris Michail Gerontogiannis}
\address{\footnotesize Dimitris Michail Gerontogiannis, Mathematical Institute, Leiden University, Niels Bohrweg 1, 2333 CA Leiden, The Netherlands}
\email{d.gerontogiannis@hotmail.com}

\begin{abstract}
We prove that all $\Kt$-homology classes of the stable (and unstable) Ruelle algebra of a Smale space have explicit Fredholm module representatives that are finitely summable on the same smooth subalgebra and with the same degree of summability. The smooth subalgebra is induced by a metric on the underlying Smale space groupoid and fine transversality relations between stable and unstable sets. The degree of summability is related to the fractal dimension of the Smale space. Further, the Fredholm modules are obtained by taking Kasparov products with a fundamental class of the Spanier-Whitehead $\Kt$-duality between the Ruelle algebras. Finally, we obtain general results on stability under holomorphic functional calculus and construct Lipschitz algebras on {\'e}tale groupoids. 
\end{abstract}

\maketitle

\vspace{-0.68cm}

\section{Introduction}

The purpose of this paper is to initiate a systematic study of index theory on $C^*$-algebras associated to Smale spaces. In this light, we prove a strong finiteness property for the $\Kt$-homology of the stable and unstable Ruelle algebras of irreducible Smale spaces. This is a decisive step in obtaining an index theorem for Ruelle algebras through Chern-Connes characters in cyclic cohomology.

Smale spaces are Ruelle's topological counterpart \cite{Ruelle} of the, often fractal-like, hyperbolic nonwandering parts of Smale's Axiom A systems \cite{Smale}. Ruelle algebras are Kirchberg algebras constructed by Putnam-Spielberg \cite{Putnam_algebras, PS} from {\'e}tale groupoids encoding stable and unstable relations on irreducible Smale spaces, and are higher dimensional analogues of stabilised Cuntz-Krieger algebras \cite{CK}.

Given a Ruelle algebra, we provide for the first time an exhaustive description of its $\Kt$-homology classes in terms of explicit Fredholm modules, which are derived by Kasparov slant products with the Spanier-Whitehead $\Kt$-duality fundamental class of Kaminker-Putnam-Whittaker \cite{KPW}. The Fredholm modules are finitely summable on the same holomorphically stable dense $*$-subalgebra, which is constructed by a metric on the underlying groupoid and transversality conditions related to the topological entropy of the Smale space. Notably, the degree of summability is uniform over this smooth subalgebra and is related to the fractal dimension of the Smale space. 

One important application is that we can now use the Chern-Connes characters of these Fredholm modules in cyclic cohomology \cite{ConnesDifGeom} to derive concrete formulas, given by noncommutative integration, for index maps on $\Kt$-theory and thus obtain an index theorem for Ruelle algebras. We will pursue this in a subsequent paper.

Relevant $\Kt$-homological finiteness results were shown to hold for crossed products of certain hyperbolic group boundary actions \cite{EN} and for Cuntz-Krieger algebras \cite{GM}. Although Ruelle algebras share a common ground (see \cite{CK, LacaSp}) due to hyperbolicity, their setting is different and studying their fine analytic structure requires a novel approach. We will elaborate on this shortly.

Before presenting our main result, for the convenience of the reader, let us mention that Fredholm modules are the building blocks of Kasparov's $\Kt$-homology \cite{Kasparov1,HR} and are abstractions of order zero elliptic operators, according to Atiyah \cite{Atiyah}. Roughly speaking, an even (odd) Fredholm module over a unital $C^*$-algebra $A$ consists of a separable Hilbert space $H$, a representation $\rho:A\to \mathcal{B}(H)$ and an operator $G\in \mathcal{B}(H)$ that, modulo compact operators, is a unitary (projection) and commutes with $\rho(A)$. For the formal definition that involves $\mathbb Z_2$-gradings we refer to Subsection \ref{sec:KK-theory_Ext-groups}. 

Further, finite summability is a regularity property that a Fredholm module might have, in which case, relatively to a dense $*$-subalgebra, the conditions holding modulo the ideal of compact operators actually hold modulo a Schatten $p$-ideal $\mathcal{L}^p$ ($p>0$). For more details see Subsection \ref{sec:Smooth_ext_sec}. A Fredholm module with this property is called $p$-summable and admits a noncommutative Chern character in cyclic cohomology, Connes' noncommutative generalisation of de Rham homology. The character is then used to produce a simple trace formula for the pairing of the associated $\Kt$-homology class with $\Kt$-theory. Consequently, one obtains noncommutative analogues of the Atiyah-Singer Index Theorem (see \cite{AS,Atiyah} and \cite[Chapter IV]{Connes}) and the Toeplitz Index Theorems (e.g see \cite{Goffeng2, HR}). 

The $\Kt$-homological finiteness of Ruelle algebras is a noncommutative manifestation of the classical fact that, for an $n$-dimensional closed smooth manifold $M$ and every $p>n$, each $\Kt$-homology class of $C(M)$ has a Fredholm module representative that is $p$-summable on $C^{\infty}(M)$. This follows from the $\Kt$-duality of $M$ with its cotangent bundle and the spectral theory of order zero elliptic operators, see \cite{Rave} for details. To place the notion of $\Kt$-homological finiteness into a formal context, we follow \cite{EN}. Specifically, we say that the $\Kt$-homology of a $C^*$-algebra $A$ is \textit{uniformly} $\mathcal{L}^p$\textit{-summable} if there is a dense $*$-subalgebra $\mathscr{A}\subset A$ such that, every (even and odd) $\Kt$-homology class has a Fredholm module representative that is $p$-summable on $\mathscr{A}$. For details and more examples see Subsection \ref{sec:Smooth_ext_sec}.

We now briefly present our main Theorems \ref{cor:KPW_smoothness} and \ref{cor:KPW_smoothness_SFT} in a combined statement. Let $\mathcal{R}$ denote the stable or unstable Ruelle algebra of an irreducible Smale space. Then, the following holds.

\begin{thm*}
There is a topological constant $\delta > 0$ such that for every $p>\delta$, there is a holomorphically stable dense $*$-subalgebra $\mathscr{R}\subset \mathcal{R}$ on which $\mathcal{R}$ has uniformly $\mathcal{L}^p$-summable $\Kt$-homology. The constant $\delta$ depends on the topological entropy and the supremal contraction constant associated to self-similar metrics of the Smale space. If the Smale space is zero-dimensional, the statement also holds with $0$ in place of $\delta$.
\end{thm*}

In addition to being useful for studying index theory on Ruelle algebras, this result indicates interesting connections with the dimension theory of Smale spaces. Namely, the degree of summability on each smooth subalgebra is related to the Hausdorff, box-counting and Assouad dimensions of a compatible self-similar metric on the Smale space. We computed these dimensions in \cite{Gero}. In this light, it would be interesting to pursue a noncommutative fractal dimension theory for Smale spaces. Moreover, we obtain that all classes in the Brown-Douglas-Fillmore group $\Ext(\mathcal{R})$ (see \cite{BDF}) are represented by smooth Toeplitz extensions in a uniform way, see Subsection \ref{sec:Smooth_ext_sec} and the author's PhD thesis \cite{Gero_Thesis}. This is related to Douglas' open question \cite[p.68]{Douglas} about invertibility of smooth extensions of $C^*$-algebras, see \cite{DV, Goffeng}. Finally, it is worth mentioning that this finiteness is remarkable as it cannot occur in the unbounded picture of $\Kt$-homology, see \cite{Connestraces} and Remark \ref{rem:obstruction}. Before outlining the theorem's proof and our other results, we present sufficient background on Smale spaces.

\subsection{A primer on Smale spaces}\label{sec:Primer_Smale_spaces}
For details we refer to Section \ref{sec:Smale_spaces}. Roughly, a Smale space $(X,d,\varphi)$ is a dynamical system consisting of a homeomorphism $\varphi$ acting on a compact metric space $(X,d)$ in a hyperbolic way; every $x\in X$ has a small neighbourhood homeomorphic to the product of two sets on which $\varphi$ expands and contracts distances at least by a factor $\lambda_d>1$ and $\lambda_d^{-1}<1$, respectively. The study of non-wandering Smale spaces can be reduced to the study of their irreducible or mixing parts due to Smale's Decomposition Theorem \ref{thm: Smale decomposition}.

Prototype examples are the non-wandering parts of Axiom A systems, for instance hyperbolic toral automorphisms \cite{BS} and horseshoes derived from iterated function systems \cite{PolW,Smale}. Smale spaces with totally disconnected stable (or unstable) sets are known as Wieler solenoids \cite{Wieler} and some examples are Williams solenoids \cite{Williams}, solenoids built from the limit sets of self-similar groups \cite{Nekr_ssg}, aperiodic substitution tiling spaces \cite{APutnam} and subshifts of finite type (SFT) \cite{BS}. The SFT are the zero-dimensional Smale spaces \cite{Putnam_Book} and special examples are the topological Markov chains (TMC), see Example \ref{ex:TMC}. We note that any SFT is topologically conjugate to a TMC. The latter provide a combinatorial model of arbitrary precision for Smale spaces, following the seminal work of Bowen \cite{Bowen2,Bowen3,Bowen4} on Markov partitions. 

Although a Smale space is a metric space, it is often only the underlying topology (rather than a specific metric) that is relevant when studying its $C^*$-algebras. With this in mind, for the $\Kt$-homological finiteness of Ruelle algebras it suffices to study Smale spaces up to topological conjugacy. This flexibility allows us to work with tractable metrics. Specifically, following the work of Artigue \cite{Artigue} on self-similar metrics for expansive dynamical systems, every Smale space $(X,d,\varphi)$ admits a compatible (self-similar) metric $d'$ such that $(X,d',\varphi)$ is a topologically conjugate (self-similar) Smale space with $\varphi$ and $\varphi^{-1}$ being $\lambda_{d'}$-Lipschitz. The conjugacy is given by the identity homeomorphism $(X,d)\to (X,d')$. Prototype examples of self-similar Smale spaces are the SFT. In \cite[Theorem 7.6]{Gero} we showed that for mixing self-similar Smale spaces $(X,d,\varphi)$ with topological entropy $\ent(\varphi)$, Bowen's measure of maximal entropy is Ahlfors $s$-regular; it is of order $r^s$ on every closed ball of radius $r$, where 
\begin{equation}\label{eq:fractal_dimension}
s=\frac{2\ent(\varphi)}{\log(\lambda_d)}.
\end{equation}
Consequently, the Hausdorff, box-counting and Assouad dimensions coincide and are equal to $s$. A straightforward application of Smale's Decomposition Theorem extends this result to irreducible Smale spaces, on which our focus henceforth lies.

In the 1980's, Ruelle \cite{Ruelle_algebras} constructed $C^*$-algebras from homoclinic equivalence relations on Smale spaces. Putnam \cite{Putnam_algebras} extended Ruelle's work by constructing the stable, unstable (groupoid) $C^*$-algebras $\mathcal{S}, \mathcal{U}$ and the stable, unstable Ruelle algebras $\mathcal{R}^s=\mathcal{S}\rtimes \mathbb Z, \mathcal{R}^u=\mathcal{U}\rtimes \mathbb Z$; the $\mathbb Z$-action is induced by the Smale space homeomorphism. Putnam-Spielberg \cite{PS} constructed $\mathcal{S},\mathcal{U}$ (up to strong Morita equivalence) from {\'e}tale stable, unstable groupoids $G^s,G^u$ and showed that Ruelle algebras are separable, simple, nuclear, $C^*$-stable, purely infinite and in the UCT class, hence are classified up to isomorphism by $\Kt$-theory \cite{KirP}. Here we focus only on the {\'e}tale groupoid picture of the Smale space $C^*$-algebras. 

Ruelle algebras of SFT are isomorphic to stabilised Cuntz-Krieger algebras \cite{CK}. Also, the Ruelle algebras of the dyadic solenoid \cite{BS} (an one-dimensional Smale space) are isomorphic to the stabilised crossed product of the modular group acting on its boundary \cite{LacaSp}. The latter is a remarkable connection between Ruelle algebras and the $C^*$-algebras studied by Emerson-Nica \cite{EN}.

Kaminker-Putnam-Whittaker \cite{KPW} proved that for every Smale space, the stable and unstable Ruelle algebras $\mathcal{R}^s$ and $\mathcal{R}^u$ have an odd $\Kt$-duality. This duality is given by a $\Kt$-homology class $\Delta\in \KKt_1(\mathcal{R}^s\otimes \mathcal{R}^u,\mathbb C)$ and a $\Kt$-theory class $\widehat{\Delta}\in \KKt_1(\mathbb C,\mathcal{R}^s\otimes \mathcal{R}^u)$ with Kasparov products
$$\widehat{\Delta}\otimes_{\mathcal{R}^u} \Delta=1_{\mathcal{R}^s}\enspace \text{and}\enspace \widehat{\Delta}\otimes_{\mathcal{R}^s} \Delta=-1_{\mathcal{R}^u}.$$
One of the various isomorphisms that we obtain via Kasparov slant product is
\begin{equation}\label{eq:slant_prod}
-\otimes_{\mathcal{R}^s} \Delta: \KKt_{*}(\mathbb C, \mathcal{R}^s)\to \KKt_{*+1}(\mathcal{R}^u,\mathbb C).
\end{equation}
The class $\Delta$ is given in terms of an extension $\tau_{\Delta}$ of $\mathcal{R}^s\otimes \mathcal{R}^u$ by the compact operators (that we call the KPW-extension), which is formed by the product of two essentially (modulo compact operators) commuting representations of $\mathcal{R}^s,\mathcal{R}^u$.

\subsection{Outline of proof} We describe the $\Kt$-homological finiteness of the unstable Ruelle algebra, the stable case being similar. Let $(X,d,\varphi)$ be an irreducible Smale space and $G^s,G^u$ be the stable, unstable ({\'e}tale) groupoids. By choosing a compatible self-similar metric $d'$ we focus on the topologically conjugate self-similar Smale space $(X,d',\varphi)$ that has the same groupoids, see \cite{Putnam_Func}. Then, using the Alexandroff-Urysohn-Frink Metrisation Theorem \ref{thm:Metrisation_thm} we build metrics $D_{s,d'},D_{u,d'}$ on the groupoids $G^s,G^u$. Each metric generates the topology of the groupoid and yields a Lipschitz structure on it, making the Smale space homeomorphism (at the groupoid level) bi-Lipschitz. These results constitute Theorem \ref{thm:GroupoidMetrisation}.

In Proposition \ref{prop:Lip_alg_Smale_groupoids}, we show that these metrics produce dense $*$-subalgebras of compactly supported Lipschitz functions $\Lip_{c}(G^s,D_{s,d'})\subset \mathcal{S},\, \Lip_{c}(G^u,D_{u,d'})\subset \mathcal{U}$, which are invariant under the $\mathbb Z$-action and hence give rise to the dense $*$-subalgebras \begin{align*}
\Lambda_{s,d'}&=\Lip_{c}(G^s,D_{s,d'})\rtimes_{\text{alg}} \mathbb Z \subset \mathcal{R}^s,\\
\Lambda_{u,d'}&=\Lip_{c}(G^u,D_{u,d'})\rtimes_{\text{alg}} \mathbb Z \subset \mathcal{R}^u.
\end{align*}
Also, $\mathcal{S},\mathcal{U}$ are faithfully represented on a Hilbert space $\mathscr{H}$ and $\mathcal{R}^s,\mathcal{R}^u$ admit faithful representations on $\mathscr{H}\otimes \ell^2(\mathbb Z)$ that commute modulo compacts, where their product in the Calkin algebra gives the KPW-extension $\tau_{\Delta}$, see Subsections \ref{sec:K-duality_Ruelle} and \ref{sec:SmalespaceCalg}. 

In Proposition \ref{prop:com_Lip_alg_general}, we prove that the images of $\Lambda_{s,d'}, \Lambda_{u,d'}$ on $\mathscr{H}\otimes \ell^2(\mathbb Z)$ commute modulo $\mathcal{L}^p(\mathscr{H}\otimes \ell^2(\mathbb Z))$ for every $p>p(d')$, where $p(d')>0$ is related to the Hausdorff dimension (\ref{eq:fractal_dimension}) of $(X,d',\varphi)$, see Subsection \ref{sec:Optimisation} and Remark \ref{rem:KPW_smooth_ext2}. Using holomorphic functional calculus (see Lemma \ref{lem:extend_almost_commuting_algebras}) and the aforementioned commutation relation, we enlarge $\Lambda_{s,d'},\Lambda_{u,d'}$ to dense $*$-subalgebras $\text{H}_{s,u,d'} \subset \mathcal{R}^s$, $\text{H}_{u,s,d'}\subset \mathcal{R}^u$ that are holomorphically stable and commute modulo $\mathcal{L}^p(\mathscr{H}\otimes \ell^2(\mathbb Z))$, for every $p>p(d')$. This step forces us to work in the setting of quasi-Banach spaces since we may have to consider Schatten $p$-ideals for $p\in (0,1)$. To this end, we extend a lemma of Connes about Riemann integration in Banach spaces to include the case of integrating in these quasi-normed Schatten ideals, see Proposition \ref{prop:Schatten_p_zero_one} and Lemma \ref{lem:Banach_alg_hol_cal_zero_one}.

These are the smooth subalgebras on which the $\Kt$-homological finiteness occurs. Specifically, we first use Proposition \ref{prop:slantprodKtheory} to compute the slant product (\ref{eq:slant_prod}) with the class of the KPW-extension and obtain explicit Fredholm module representatives for the $\Kt$-homology of $\mathcal{R}^u$. The main difficulty here is the lack of unit in $\mathcal{R}^s$ and $\mathcal{R}^u$. We circumvent this by using approximate identities $(u_n)_{n\in \mathbb N}$ satisfying $u_{n+1}u_n=u_n$ and the fact that the corners of simple, purely infinite $C^*$-algebras have nice $\Kt$-theory \cite{Cuntz}. Then, from Proposition \ref{prop:Ext_smoothness_prop} the representatives can be chosen to be $p$-summable on $\text{H}_{u,s,d'}$, for every $p>p(d')$.  

Then, Theorem \ref{cor:KPW_smoothness} is about the infimum degree of summability $\delta=\inf_{d'}p(d')>0$. This infimum depends on the topological entropy and the $\lambda$-number $\lambda(X,\varphi)\in (1,\infty]$; the supremum of the contraction constants for compatible self-similar metrics on the Smale space. Proposition \ref{prop:lambda_number_top.invar.} states that the $\lambda$-numbers are topological invariants and Proposition \ref{prop:lambda_number_dimension} relates them to the topological dimension of the Smale space.

If $X$ is zero-dimensional, the different metrisation approach of Proposition \ref{prop:SFT_groupoids_Lipschitz} yields refined groupoid ultrametrics that lead to sharp summability estimates. Using the same notation, we focus on compatible self-similar metrics $d'$ for which $\mathcal{R}^u$ ends up having uniformly $\mathcal{L}^p$-summable $\Kt$-homology on the smooth subalgebra $\text{H}_{u,s,d'}$ for every $p>p(d')$, where $p(d')$ is exactly half the Hausdorff dimension (\ref{eq:fractal_dimension}) of $(X,d',\varphi)$. This also happens to be the Hausdorff dimension of all local stable and unstable sets, see \cite[Remark 7.7]{Gero}. Further, the associated infimum degree of summability is zero. For details see Theorem \ref{cor:KPW_smoothness_SFT}. We should mention that such groupoid ultrametrics have not been constructed before and our method differs from the one used in \cite{GM} for Cuntz-Krieger algebras, as it also works for the even $\Kt$-homology. 

\begin{remark}\label{rem:totalinf}
This $\delta>0$ should be distinguished from the typically smaller infimum degree of summability over arbitrary smooth subalgebras. For example, for the dyadic solenoid the latter is zero, since $\mathcal{R}^u$ is isomorphic, using $\Kt$-theory \cite{LacaSp}, to the unstable Ruelle algebra of a TMC.
\end{remark}

The machinery we have described so far also leads to Theorems \ref{thm:KPW_smoothness} and \ref{thm:KPW_smoothness_SFT} which are about the smoothness of the KPW-extension in the sense of Douglas \cite{Douglas}, see Subsection \ref{sec:Smooth_ext_sec}. These results will be used in a subsequent paper for providing a geometric picture of the $\Kt$-duality between $\mathcal{R}^s$ and $\mathcal{R}^u$.

Before concluding, we note that the $\Kt$-homological finiteness results of Emerson-Nica \cite{EN} and Goffeng-Mesland \cite{GM} are analogous to ours in the use of $\Kt$-duality. Specifically, Emerson-Nica proved the $\Kt$-homological finiteness of $C(\partial \Gamma)\rtimes \Gamma$, where $\Gamma$ is a regular torsion-free hyperbolic group acting on its boundary $\partial \Gamma$. They used visual metrics and measures on $\partial \Gamma$, Coornaert's Hausdorff dimensions \cite{Coornaert} and Emerson's odd self-duality of $C(\partial \Gamma)\rtimes \Gamma$ \cite{EmersonDuality}. Moreover, Goffeng-Mesland proved that, the odd $\Kt$-homology of a Cuntz-Krieger algebra is uniformly $\mathcal{L}^p$-summable for every $p>0$. They used an explicit description of the $\Kt$-theory of Cuntz-Krieger algebras \cite{Cuntz_K_theory}, the odd $\Kt$-duality of Kaminker-Putnam \cite{KP} and extension lifting results from \cite{Goffeng}.

Nevertheless, the treatment of each case (including ours) requires specific tools and techniques. For instance, the aforementioned authors do not use groupoid metrics. Also, Ruelle algebras being non-unital  complicates the computation of the Kasparov slant products. Further, we consider quasi-normed Schatten ideals, while in \cite{GM} it is not necessary and in \cite{EN} they focus only on Schatten $p$-ideals for $p>2$. Finally, our method emerges from focusing on the KPW-extension, which is more tractable than the ones in \cite{EmersonDuality,KP}. In particular, it is the product of two commuting extensions which immediately lift to essentially commuting $*$-representations in the bounded operators on a Hilbert space, while the other two do not lift in an obvious way.
\enlargethispage{\baselineskip}
\subsection{Organisation} Section \ref{sec:K-theoretic_prelim} contains the $\Kt$-theoretic background and Section \ref{sec:Smale_spaces} is about $C^*$-algebras from Smale spaces and the $\Kt$-duality of Ruelle algebras. Then, in Section \ref{sec:Slant_prod_Unif_Smooth} we compute Kasparov slant products for simple, purely infinite $C^*$-algebras and develop holomorphic functional calculus tools to study uniform summability. In Section \ref{sec:Lip_etale_alg} we construct Lipschitz algebras on {\'e}tale groupoids and in Section \ref{sec:Metrise_Groupoids} we build dynamical metrics on Smale space groupoids. Finally, in Section \ref{sec:SmoothRuelle} we obtain the desired smooth subalgebras of Ruelle algebras and prove the main results.
\newpage
\section{$\Kt$-theoretic preliminaries}\label{sec:K-theoretic_prelim}

\subsection{KK-theory}\label{sec:KK-theory_Ext-groups} For the theory of Hilbert $C^*$-modules we refer to \cite{Lance}. The classical references for $\KKt$-theory are \cite{Blackadar, JT}. In \cite{Kasparov5}, 
Kasparov associated to a pair $(A,B)$ of separable $\mathbb Z_2$-graded $C^*$-algebras an abelian group $\KKt_0(A,B)$ whose elements are classes of generalised $*$-homomorphisms from $A$ to $B$. 
\begin{definition}[{\cite[Def. 17.1.1]{Blackadar}}]\label{def:Kasparovbimodules}
Let $A,B$ be separable $\mathbb Z_2$-graded $C^*$-algebras. A \textit{Kasparov} $(A,B)$-\textit{bimodule} is a triple $(E,\rho,F)$ such that
\begin{enumerate}[(1)]
\item $E$ is a countably generated $\mathbb Z_2$-graded Hilbert $B$-module;
\item $\rho:A\to \mathcal B_B (E)$ is a graded $*$-homomorphism to the adjointable operators;
\item the operator $F\in \mathcal B_B(E)$ anti-commutes with the grading of $E$ and satisfies $$\rho(a)(F^*-F)\in \mathcal{K}_B(E),\,\, \rho(a)(F^2-1)\in \mathcal{K}_B(E),\,\, [\rho(a),F]\in \mathcal{K}_B(E),$$ for all $a\in A$, where $\mathcal{K}_B(E)$ is the two-sided ideal of compact operators.
\end{enumerate}
\end{definition}
By considering unitary equivalence and operator homotopy of such bimodules we obtain the group $\KKt_0(A,B)$. Also, the group $\KKt_0(A\otimes \mathbb C_1,B)$, where $\mathbb C_1$ is the Clifford algebra with one generator, is denoted by $\KKt_1(A,B)$. For trivially graded $A,B,$ if $A=\mathbb C$ then $\KKt_i(\mathbb C, B)\cong \Kt_i(B)$, and if $B=\mathbb C$ then $\KKt_i(A, \mathbb C)=\Kt^i(A)$, with the $\Kt$-homology defined as in \cite{HR}. The cycles of $\Kt^0(A)$ are the \textit{even Fredholm modules over} $A$, which are exactly the Kasparov $(A,\mathbb C)$-bimodules. The cycles of $\Kt^1(A)$ are the \textit{odd Fredholm modules over} $A$ and are the Kasparov $(A,\mathbb C)$-bimodules without any mention on gradings.

Let $f:B\to D$ be a graded $*$-homomorphism and $[\mathcal{E}]=[E,\rho,F]\in \KKt_0(A,B)$. The \textit{push-forward} operation gives  
\begin{equation}\label{eq:push-forwardtriple}
f_*[\mathcal{E}]=[E\otimes_{f}D,\rho\otimes_{f}\id,F\otimes_{f}\id]\in \KKt_0(A,D),
\end{equation}
where $E\otimes_{f}D$ is the internal tensor product. If $g:D\to A$ is a graded $*$-homomorphism, the \textit{pull-back} operation yields the triple 
\begin{equation}\label{eq:pullback}
g^*[\mathcal{E}]=[E,\rho \circ g, F]\in \KKt_0(D,B).
\end{equation}
Further, the external tensor product gives 
\begin{equation}\label{eq:tensortriple}
\tau_D[\mathcal{E}]=[E\otimes D, \rho \otimes \id , F\otimes \id]\in \KKt_0(A\otimes D, B\otimes D),
\end{equation}
where the graded tensor products of $C^*$-algebras have the spatial norm. Tensoring with $D$ from the left gives the class $\tau^D[\mathcal{E}]$. Moreover, the $\KKt$-bifunctor is $C^*$-stable, homotopy-invariant in both variables and satisfies \textit{formal Bott periodicity}; that is, $\KKt_1(A\otimes \mathbb C_1,B)\cong \KKt_0(A,B)$ and $\KKt_0(A,B\otimes \mathbb C_1)\cong \KKt_1(A,B)$.

At the heart of $\KKt$-theory lies the \textit{Kasparov product}
\begin{equation}\label{eq:Kasparovprod1}
\otimes_D:\KKt_0(A,D)\times \KKt_0(D,B)\to \KKt_0(A,B)
\end{equation}
which generalises the cup-cap product from topological $\Kt$-theory and composition of graded $*$-homomorphisms. It is associative and functorial in all possible ways, thus creating the $\KKt$-category \cite{Blackadar}. This turns $\KKt_0(A,A)$ into a ring with identity $1_A$. The most general form of the product is given as
\begin{equation}\label{eq:Kasparovprod2}
\otimes_D:\KKt_0(A_1,B_1\otimes D)\times \KKt_0(D\otimes A_2,B_2)\to \KKt_0(A_1\otimes A_2,B_1\otimes B_2)
\end{equation}
with a slight abuse of notation; for $x\in \KKt_0(A_1,B_1\otimes D)$ and  $y\in \KKt_0(D\otimes A_2,B_2)$ we define $x\otimes_D y= \tau_{A_2}(x)\otimes_{B_1\otimes D\otimes A_2}\tau^{B_1}(y).$ The product formulates \textit{Bott periodicity}; one of the isomorphisms is $\beta \otimes_{C_0(\mathbb R)}:\KKt_1(C_0(\mathbb R)\otimes A,B)\to \KKt_0(A,B)$ (using formal Bott periodicity), where $\beta$ is the generator of $\KKt_1(\mathbb C, C_0(\mathbb R))\cong \mathbb Z$.

If $A,B$ are trivially graded then $\KKt_1(A,B)$ is naturally isomorphic to the group $\Ext^{-1}(A,B)\subset \Ext(A,B)$ of classes of invertible extensions of $A$ by $B\otimes \mathcal{K}(H)$. Here $\mathcal{K}(H)$ is the ideal of compact operators on a separable Hilbert space $H$. Also, $\Ext(A,\mathbb C)$ will be denoted by $\Ext(A)$. If $A$ is nuclear then $\Ext^{-1}(A,B)=\Ext(A,B)$. Further, extensions $\mathcal{E}$ of $A$ by $\mathcal{K}(H)$ can be given in terms of their Busby invariants $\tau_{\mathcal{E}}:A\to \mathcal{Q}(H)$, where $\mathcal{Q}(H)$ is the Calkin algebra. Finally, the isomorphism $\Kt^1(A)\cong \Ext^{-1}(A)$ maps the $\Kt$-homology class $[H,\rho,F]\in \Kt^1(A)$ (assuming the normalisation $F^2=1,F^*=F$) to the class of the Toeplitz extension $\tau:A\to \mathcal{Q}(H)$ given by $\tau(a)=P\rho(a)P+\mathcal{K}(H)$, where $P=(F+1)/2$ is a projection.

\subsection{Spanier-Whitehead K-duality}\label{sec:SW-duality} Here all $C^*$-algebras have trivial grading. Spanier-Whitehead duality \cite{BG} relates the homology of a finite complex with the cohomology of a dual complex. The following noncommutative analogue is based on the definitions given in \cite{EmersonDuality,KP,KPW}.

\begin{definition}\label{def:SW-Kduality}
Let $A$ and $B$ be two separable $C^*$-algebras. We say that $A$ and $B$ are \textit{Spanier-Whitehead} $\Kt$\textit{-dual} (or just $\Kt$\textit{-dual}) if there is a $\Kt$-homology class $\Delta\in \KKt_i(A\otimes B,\mathbb C)$ and a $\Kt$-theory class $\widehat{\Delta}\in \KKt_i(\mathbb C, A\otimes B)$ such that 
\begin{align*}
\widehat{\Delta}\otimes_B \Delta=1_A \,\,\,\, \text{and} \,\,\,\, \widehat{\Delta}\otimes_A \Delta=(-1)^i1_B.
\end{align*}
In particular, if $B$ is the opposite algebra $A^{\textit{op}}$, $A$ is called a \textit{Poincar{\'e} duality algebra}.
\end{definition}

\begin{remark}
By $\widehat{\Delta}\otimes_B \Delta$ we mean the product $\widehat{\Delta}\otimes_B (\sigma_{12})^*(\Delta)$ and by $\widehat{\Delta}\otimes_A \Delta$ we mean $(\sigma_{12})_*(\widehat{\Delta})\otimes_A \Delta$, where $\sigma_{12}:A\otimes B\to B\otimes A$ is the flip map.
\end{remark}

The following was first proved by Connes in the case $i=0$. For the general case see \cite[Lemma 9]{EmersonDuality}. Given a duality pair $(\widehat{\Delta},\Delta)$ between $A$ and $B$, we obtain the isomorphisms
\begin{equation}\label{eq:slantprod1}
\widehat{\Delta}\otimes_B:\KKt_j(B, \mathbb C)\to \KKt_{j+i}(\mathbb C, A),
\end{equation}
\begin{equation}\label{eq:slantprod2}
\otimes_A \Delta : \KKt_j(\mathbb C, A)\to \KKt_{j+i}(B,\mathbb C).
\end{equation}
It is worth mentioning that $\widehat{\Delta}$ is unique for $\Delta$ and vice versa. The question whether a separable $C^*$-algebra has a Spanier-Whitehead $\Kt$-dual is addressed in \cite{KS}.

\begin{remark} $\Kt$-duality is traced back to Kasparov's isomorphism \cite{Kasparov2} between the $\Kt$-theory and $\Kt$-homology of a compact Riemannian manifold and its cotangent bundle. Kasparov referred to this duality as $\KKt$-theoretic Poincar{\'e} duality. Later, Connes \cite[Chapter VI]{Connes} showed self-duality for the irrational rotation algebras. Then, Kaminker\texttt{-}Putnam \cite{KP} studied the case of Cuntz-Krieger algebras. In their work, they argued that Spanier-Whitehead $\Kt$-duality is a more appropriate terminology than Poincar{\'e} duality, since classical Poincar{\'e} duality relates the homology and cohomology of the same manifold. A basic difference between classical Spanier-Whitehead duality and classical Poincar{\'e} duality is that, for the first duality one does not have to assume orientability of the space, while for the latter, orientability is necessary. For the relation between the two dualities we refer the reader  to \cite[Section 2]{KP}.
\end{remark}
Popescu\texttt{-}Zacharias \cite{PZ} showed Poincar{\'e} duality for higher rank graph algebras. In \cite{EmersonDuality} Emerson showed Poincar{\'e} duality for crossed products of a large class of hyperbolic group boundary actions. Echterhoff\texttt{-}Emerson\texttt{-}Kim \cite{EEK2} showed $\Kt$-duality for certain orbifold $C^*$-algebras. Nishikawa-Proietti \cite{NP} studied $\Kt$-duality in the setting of discrete groups. Kaminker\texttt{-}Putnam\texttt{-}Whittaker \cite{KPW} proved the $\Kt$-duality between the stable and unstable Ruelle algebras, see Subsection \ref{sec:K-duality_Ruelle}.

\subsection{K-homological summability}\label{sec:Smooth_ext_sec}
Let $H$ be a separable Hilbert space. For a compact operator $T\in \mathcal{K}(H)$, let $(s_n(T))_{n\in \mathbb N}$ be the sequence of its singular values in decreasing order, counting their multiplicities. 
\begin{definition}[\cite{GK}]
A \textit{symmetrically normed ideal} is a proper non-zero two-sided ideal $\mathcal{I}$ of $\mathcal{B}(H)$ with norm $\|\cdot \|_{\mathcal{I}}$ such that
\begin{enumerate}[(1)]
\item $\|SRT\|_{\mathcal{I}}\leq \|S\| \|R\|_{\mathcal{I}} \|T\|$, for all $S,T \in \mathcal{B}(H)$ and $R\in \mathcal{I}$;
\item $\mathcal{I}$ is a Banach space with the norm $\|\cdot \|_{\mathcal{I}}.$
\end{enumerate}
\end{definition}
The \textit{Schatten} $p$-\textit{ideal} on $H$, where $p>0$, is defined as 
\begin{equation}\label{eq:Schattenideal}
\mathcal{L}^p(H)=\{T\in \mathcal{K}(H): (s_n(T))_{n\in \mathbb N}\in \ell^p(\mathbb N)\},
\end{equation}
and is equipped with the \textit{Schatten} $p$\textit{-norm} 
\begin{equation}\label{eq:Schattennorm}
\|T\|_p=\|(s_n(T))_{n\in \mathbb N}\|_{\ell^p(\mathbb N)}.
\end{equation}
For $p\geq 1$, the ideal $\mathcal{L}^p(H)$ is symmetrically normed. For $p\in (0,1)$ it is only quasi-normed, hence a quasi-Banach space, see Subsection \ref{sec:Hol_Fun_Cal}. However, all the other basic properties of symmetrically normed ideals are satisfied \cite{GK}. Symmetrically normed ideals are used in the work of Douglas \cite{Douglas} on smooth extensions. 

Let $(\mathcal{I}, \|\cdot \|_{\mathcal{I}})$ denote a symmetrically normed ideal or a general Schatten-ideal and $A$ be a separable $C^*$-algebra. 

\begin{definition}[\cite{Douglas}]\label{def:summableextension} Let $\mathcal{A}\subset A$ be a dense $*$-subalgebra. An extension $\tau:A\to \mathcal{Q}(H)$ will be called $\mathcal{I}$\textit{-smooth on} $\mathcal{A}$ if there is a linear map $\eta: \mathcal{A}\to \mathcal{B}(H)$ such that $$\eta(ab)-\eta(a)\eta(b)\in \mathcal{I},\, \eta(a^*)-\eta(a)^*\in \mathcal{I}$$ and $\tau(a)=\eta(a)+\mathcal{K}(H)$, for all $a,b \in \mathcal{A}$. If $\mathcal{I}=\mathcal{L}^p(H)$, the extension will be called $p$\textit{-smooth} (or just \textit{finitely smooth}).
\end{definition}

Connes' work \cite{ConnesDifGeom} on summable Fredholm modules produces a plethora of smooth extensions. Consider the two-sided ideal $\mathcal{I}^{1/2}=\text{span}\{T\in \mathcal{K}(H):T^*T,\, TT^*\in \mathcal{I}\}$ and for consistency assume it is symmetrically normed or a general Schatten-ideal. For instance, if $\mathcal{I}=\mathcal{L}^{p/2}(H)$ then $\mathcal{I}^{1/2}=\mathcal{L}^p(H).$

\begin{definition}[{\cite{ConnesDifGeom, GM}}]
Let $\mathcal{A}\subset A$ be a dense $*$-subalgebra. An even or odd Fredholm module $(H,\rho,F)$ over $A$ is $\mathcal{I}^{1/2}$\textit{-summable on} $\mathcal{A}$ if, for every $a\in \mathcal{A}$, it holds that 
$$\rho(a)(F^*-F)\in \mathcal{I},\, \rho(a)(F^2-1)\in \mathcal{I},\, [\rho(a),F]\in \mathcal{I}^{1/2}.$$ If $\mathcal{I}^{1/2}=\mathcal{L}^p(H)$, the Fredholm module will be called $p$\textit{-summable} (or just \textit{finitely summable}).
\end{definition}

The natural isomorphism $\Kt^1(A)\to \Ext^{-1}(A)$ sends the class of an odd $\mathcal{I}^{1/2}$-summable Fredholm module to the class of an extension that is $\mathcal{I}$-smooth on the same dense $*$-subalgebra. However, in general, realising (lifting) classes of smooth extensions in this way is a delicate process, as we see in \cite{Goffeng}.

Summability plays a crucial role in index theory and quantized calculus (see \cite{Connes}). Consequently, $C^*$-algebras with the following strong $\Kt$-homological condition are of particular interest. 

\begin{definition}[{\cite{EN}}]\label{def:uniformsummmability}
The $\Kt$-homology of $A$ is \textit{uniformly} $\mathcal{I}$\textit{-summable} if there is a dense $*$-subalgebra $\mathcal{A}\subset A$ such that, every $x\in \Kt^*(A)$ can be represented by a Fredholm module which is $\mathcal{I}$-summable on $\mathcal{A}$.
\end{definition}

\begin{remark}
Analogously, one can define uniform summability for $\Ext$-groups, but we do not know if, for general $C^*$-algebras, the two notions would be equivalent. 
\end{remark}

The $\Kt$-homology of $C(M)$, where $M$ is an $n$-dimensional closed smooth manifold, is uniformly $\mathcal{L}^p$-summable on $C^{\infty}(M)$ for every $p>n$ \cite{Rave}. Further, the odd $\Kt$-homology of a Cuntz-Krieger algebra is uniformly $\mathcal{L}^p$-summable on the dense $*$-subalgebra generated by the $C^*$-generators, for every $p>0$ \cite{GM}. Moreover, for certain hyperbolic groups $\Gamma$, the $\Kt$-homology of $C(\partial \Gamma)\rtimes \Gamma$ is uniformly $\mathcal{L}^p$-summable on $L\rtimes_{\text{alg}} \Gamma$, where $L$ is a dense $*$-subalgebra of Lipschitz (with respect to a visual metric on $\partial \Gamma$) functions and $p$ depends on the Hausdorff dimension of the metric. Also, the $\Kt$-homology of the reduced group $C^*$-algebra $C^*_r(\Gamma)$ of a finitely generated free group $\Gamma$ is uniformly $\mathcal{L}^p$-summable on the group ring $\mathbb C \Gamma$, whenever $p>2$. For these see \cite{EN}. To obtain uniform results in the aforementioned examples the authors use $\Kt$-duality. A different approach was used to show that every AF-algebra has uniformly $\mathcal{L}^1$-summable $\Kt$-homology \cite[Chapter 4]{Rave}.

\begin{remark}\label{rem:obstruction}
So far in the literature there is no known obstruction that prevents the $\Kt$-homology of $C^*$-algebras to be uniformly finitely summable. But, it is also not true that this finiteness condition is universal, see \cite[Lemma 6]{GM} and \cite{Push} for some special counterexamples. In the unbounded picture of $\Kt$-homology the situation is different. Pure infiniteness is a $C^*$-algebraic condition that prevents the existence of finitely summable unbounded Fredholm modules. This is because purely infinite $C^*$-algebras are traceless and any such module would yield a tracial state on the $C^*$-algebra \cite{Connestraces}. What makes the study of such an obstruction in $\Kt$-homology even more difficult, but certainly more interesting, is that classes of unbounded Fredholm modules which are at best $\theta$-summable \cite{Connes} (weaker than finite summability) can be associated to classes of finitely summable Fredholm modules. For example, this happens for the purely infinite $C^*$-algebras studied in \cite{EN,GM} and of this paper.
\end{remark}

\section{Preliminaries on Smale spaces}\label{sec:Smale_spaces}

\subsection{Smale spaces}
For more details about Smale spaces we refer to \cite{Putnam_Lec, Ruelle}. We start by introducing some basic recurrence notions for topological dynamical systems. Throughout, we will consider (infinite) compact metric spaces $(Z,\rho)$  equipped with a homeomorphism $\psi:Z\to Z$. The corresponding dynamical systems will be denoted by $(Z,\rho,\psi)$ or simply by $(Z,\psi)$, if there is no risk of confusion. 

\begin{definition}\label{def:Recurrence}
Let ($Z,\psi$) be a dynamical system. 
\begin{enumerate}[(1)]
\item A point $z\in Z$ is called \textit{non-wandering} if for every open neighbourhood $U$ of $z$ there is some $n\in \mathbb N$ such that $\psi^n(U)\cap U\neq \varnothing$. Moreover, we say that $(Z,\psi)$ is \textit{non-wandering} if every $z\in Z$ is non-wandering.
\item ($Z,\psi$) is called \textit{irreducible} if for every ordered pair of non-empty open sets $U,V\subset Z$, there is some $n\in \mathbb N$ such that $\psi^n(U)\cap V\neq \varnothing$.
\item ($Z,\psi$) is called \textit{mixing} if for every ordered pair of non-empty open sets $U,V\subset Z$, there is some $N\in \mathbb N$ such that $\psi^n(U)\cap V\neq \varnothing$, for every $n\geq N$.
\end{enumerate}
\end{definition}

\begin{definition}[{\cite{Ruelle}}]
Let $(X,d)$ be a compact metric space and $\varphi:X\to X$ be a homeomorphism. The dynamical system ($X,d,\varphi$) is a \textit{Smale space} if there are constants $\varepsilon_X>0$, $\lambda_X >1$ (which depend on $d$) and a locally defined bi-continuous map, called the \textit{bracket map},
$$[\cdot,\cdot]:\{ (x,y)\in X\times X: d(x,y)\leq \varepsilon_X\} \to X$$
that satisfies the axioms:
\begin{align*}
\tag{B1} [x,x]&=x,\\
\tag{B2} [x,[y,z]]&=[x,z],\\
\tag{B3} [[x,y],z]&=[x,z],\\
\tag{B4} \varphi([x,y])&=[\varphi(x),\varphi(y)];
\end{align*}
for any $x,y,z \in X$, whenever both sides are defined. For $x\in X$ and $0<\varepsilon \leq \varepsilon_X$ let 
\begin{align}
X^s(x,\varepsilon)&=\{y\in X: d(x,y)<\varepsilon, [x,y]=y\},\\
X^u(x,\varepsilon)&=\{y\in X: d(x,y)<\varepsilon, [y,x]=y\}
\end{align}
be the \textit{local stable} and \textit{unstable sets}. On these sets we have the \textit{contraction axioms}: 
\begin{align*}
\tag{C1} d(\varphi(y),\varphi(z))&\leq \lambda_X^{-1}d(y,z), \enspace \text{for any} \enspace y,z \in X^s(x,\varepsilon),\\
\tag{C2} d(\varphi^{-1}(y),\varphi^{-1}(z))&\leq \lambda_X^{-1}d(y,z),\enspace \text{for any} \enspace y,z \in X^u(x,\varepsilon).
\end{align*}
\end{definition}

We say that the bracket map defines a \textit{local product structure} on $X$ because, for any $x\in X$ and $0<\varepsilon \leq \varepsilon_X/2$, the bracket map 
\begin{equation}\label{eq:bracketmap}
[\cdot, \cdot]:X^u(x,\varepsilon)\times X^s(x,\varepsilon)\to X
\end{equation}
is a homeomorphism onto its image \cite[Prop. 2.1.8]{Putnam_Book}. Also, due to the uniform continuity, there is a constant $0<\varepsilon_X'\leq \varepsilon_X/2$ such that, if $d(x,y)\leq \varepsilon_X'$, then both $d(x,[x,y]),d(y,[x,y])< \varepsilon_X/2$ and hence 
\begin{equation}\label{eq:uniquebracket}
X^s(x,\ep_X/2)\cap X^u(y,\ep_X/2)=[x,y].
\end{equation}
Equation (\ref{eq:uniquebracket}) together with a bracket independent description of the local stable and unstable sets (see \cite[Subsection~4.1]{Putnam_Lec}) imply that the bracket map is unique on $X$ (but of course depends on $\ep_X$ and $\lambda_X$). 

\begin{example}[TMC]\label{ex:TMC}
Equip $\{1,\ldots , N\}$ with the discrete topology and $\{1,\ldots , N\}^{\mathbb Z}$ with the product topology making it a compact Hausdorff space. Let $M$ be an $N\times N$ matrix, with $0$ and $1$ entries, and consider the closed subspace of allowable sequences 
\begin{equation}\label{eq:SFT}
\Sigma_M=\{x=(x_i)_{i\in \mathbb Z}\in \{1,\ldots , N\}^{\mathbb Z}: M_{x_i,x_{i+1}}=1\}
\end{equation}
whose topology is generated by the ultrametric 
\begin{equation}\label{eq:SFTmetric}
d(x,y)=\text{inf}\{2^{-n}:n\geq 0, \enspace x_i=y_i \enspace \text{for} \enspace |i|<n\}.
\end{equation}
The shift homeomorphism $\sigma_M:\Sigma_M\to \Sigma_M$ given by $\sigma_M(x)_{i}=x_{i+1}$, for any $i\in \mathbb Z$, defines the \textit{topological Markov chain} ($\Sigma_M,\sigma_M$) which becomes a Smale space with bracket map given by
\begin{equation}\label{eq:SFTbracket}
([x,y])_n=
\begin{cases}
y_n, & \text{for} \enspace n\leq 0\\
x_n, & \text{for} \enspace n\geq 1,
\end{cases}
\end{equation}
for $x,y \in \Sigma_M$ with $d(x,y)\leq 2^{-1}$. We have $\varepsilon_{\Sigma_M}=2^{-1}, \lambda_{\Sigma_M}=2$ and $\sigma_M, \sigma_M^{-1}$ are $2$-Lipschitz, halving distances on local stable and unstable sets, respectively.
\end{example}

Smale spaces are expansive \cite[Prop. 2.1.9]{Putnam_Book} and thus have finite covering dimension \cite{Mane} and topological entropy $\ent(\varphi)$ \cite[Theorem~3.2]{Walters}. In fact, without any assumption on the metric, every Smale space has finite Hausdorff dimension \cite[Chapter 7]{Ruelle}. Further, the study of Smale spaces from a topological perspective is facilitated from the fact that there is flexibility in the choice of metric that we can work with. Before elaborating on this we give the following definition. 

\begin{definition}[\cite{Gero}]\label{def:selfsimilarSmalespace}
A Smale space $(X,d,\varphi)$ will be called \textit{self-similar} if both $\varphi,\,\varphi^{-1}$ are $\lambda_X$-Lipschitz with respect to the metric $d$.
\end{definition}

Prototype examples are the TMC, or more generally the SFT. Self-similar Smale spaces are particularly nice since the homeomorphism $\varphi$ acts as the $\lambda_X^{-1}$-multiple of an isometry on local stable sets and $\varphi^{-1}$ acts similarly on local unstable sets. Following the work of Artigue \cite{Artigue} on self-similar metrics for expansive dynamical systems we can deduce that self-similar Smale spaces do exist in abundance (see also \cite[Subsection 4.5]{Gero} for a criterion that allows to construct self-similar metrics on Smale spaces with a bi-Lipschitz homeomorphism). 

\begin{lemma}[\cite{Artigue}]\label{lem:Artiguelemma} Every Smale space $(X,d,\varphi)$ admits a compatible (self-similar) metric $d'$ so that $(X,d',\varphi)$ is a self-similar Smale space.
\end{lemma}

According to Smale's program \cite{Smale}, the interesting dynamics of a Smale space lie in the non-wandering set which is the closure of its periodic points and which can be studied through its irreducible and mixing components \cite[Section~7.4]{Ruelle}.

\begin{thm}[Smale's Decomposition Theorem] \label{thm: Smale decomposition}
Assume that the Smale space $(X,\varphi)$ is non-wandering. Then $X$ can be decomposed into a finite disjoint union of clopen, $\varphi$-invariant, irreducible sets $X_0,\ldots , X_{N-1}$. Each of these sets can then be decomposed into a finite disjoint union of clopen sets $X_{i0},\ldots ,X_{iN_i}$ that are cyclically permuted by $\varphi$, and where $\varphi^{N_i+1}|_{X_{ij}}$ is mixing, for every $0\leq j\leq N_i$.
\end{thm}

\enlargethispage{\baselineskip}
\subsection{Stable and Unstable Groupoids}
For the theory of {\'e}tale groupoids see \cite{Renault_Book, Sims} and Section \ref{sec:Lip_etale_alg}. In the sequel our focus lies on an irreducible Smale space $(X,d,\varphi)$. The \textit{global stable} and \textit{unstable sets} at $x\in X$ are
\begin{equation}
\begin{split}
X^s(x)&=\{y\in X: \lim_{n\to +\infty}d(\varphi^n(x),\varphi^n(y))=0\},\\
X^u(x)&=\{y\in X: \lim_{n\to -\infty}d(\varphi^n(x),\varphi^n(y))=0\}.
\end{split}
\end{equation}
If $(X,\varphi)$ is mixing, $X^s(x)\cap X^u(y)$ is countable and dense in $X$, for all $x,y\in X$ \cite{Ruelle}.

From \cite{Putnam_Lec}, if $0<\varepsilon \leq \varepsilon_X$, then $X^s(x,\varepsilon)\subset X^s(x), \, X^u(x,\varepsilon)\subset X^u(x)$,
\begin{align}\label{eq:increasingnetleaves}
X^s(x)=\bigcup_{n\geq 0} \varphi^{-n}(X^s(\varphi^n(x),\varepsilon))\,\,\,\, \text{and} \,\,\,\,
X^u(x)=\bigcup_{n\geq 0} \varphi^{n}(X^u(\varphi^{-n}(x),\varepsilon)).
\end{align}
The \textit{stable} and \textit{unstable equivalence relations} are $G^s=\{(x,y):y\in X^s(x)\}$ and  $G^u=\{(x,y):y\in X^u(x)\}$. Putnam \cite{Putnam_algebras} showed that $G^s,G^u$ admit the structure of topological groupoids. Later, Putnam\texttt{-}Spielberg \cite{PS} constructed {\'e}tale versions of these groupoids giving rise to $C^*$-algebras that are strongly Morita equivalent to the ones of $G^s$ and $G^u$, following the philosophy of the abstract transversals in \cite{MRW}. Specifically, for periodic orbits $P,Q$ in the Smale space, define 
\begin{equation}
X^s(P)=\bigcup_{p\in P}X^s(p)\,\,\,\,\text{and}\,\,\,\, X^u(Q)=\bigcup_{q\in Q}X^u(q),
\end{equation}
and consider the \textit{stable} and \textit{unstable groupoids} 
\begin{equation}
\begin{split}
G^s(Q)&=\{(v,w)\in G^s: v,w \in X^u(Q)\},\\
G^u(P)&=\{(v,w)\in G^u: v,w \in X^s(P)\}.
\end{split}
\end{equation}
Due to irreducibility, Theorem \ref{thm: Smale decomposition} yields that $G^s(Q)$ and $G^u(P)$ meet every stable and unstable equivalence class, respectively, at countably many points. Using (\ref{eq:increasingnetleaves}), the unit spaces $X^u(Q)$ and $X^s(P)$ can be equipped with inductive limit topologies. The topologies on the groupoids are given by families of local homeomorphisms.

Recall the definition of the constant $0<\varepsilon_X'\leq \varepsilon_X/2$ from (\ref{eq:uniquebracket}), which is small enough so that, for all $x,y \in X$ with $d(x,y)\leq \varepsilon_X'$, we have that both $d(x,[x,y])$ and $d(y,[x,y])$ are less than $\varepsilon_X/2$. Let $(v,w)\in G^s(Q)$ and for some $N\in \mathbb N$ one has $\varphi^N(w)\in X^s(\varphi^N(v),\varepsilon_X'/2)$. Due to the continuity of $\varphi^{N}$ we can find some $\eta >0$ such that $\varphi^N(X^u(w,\eta))\subset X^u(\varphi^N(w),\varepsilon_X'/2)$. The map $h^s:X^u(w,\eta)\to X^u(v,\varepsilon_X/2)$ defined by 
\begin{equation}
h^s(z)=\varphi^{-N}[\varphi^N(z),\varphi^N(v)],
\end{equation}
is called a \textit{stable holonomy map} and is a homeomorphism onto its image. Also, it holds that $h^s(w)=v$. Similarly, we have \textit{unstable holonomy maps} $h^u$. More precisely, for every $(v,w)\in G^u(P)$ there are $N\in \mathbb N$ and $\eta>0$ so that, the map $h^u:X^s(w,\eta)\to X^s(v,\varepsilon_X/2)$ given by 
\begin{equation}
h^u(z)=\varphi^N[\varphi^{-N}(v),\varphi^{-N}(z)],
\end{equation}
is a homeomorphism onto its image and $h^u(w)=v$. 

\begin{thm}[{\cite[Theorem 8.3.5]{Putnam_Lec}}]\label{thm: stable bisections}
The collections of sets 
\begin{align*}
V^s(v,w,h^s,\eta,N)&=\{(h^s(z),z):z\in X^u(w,\eta)\},\\
V^u(v,w,h^u,\eta,N)&=\{(h^u(z),z):z\in X^s(w,\eta)\}
\end{align*}
generate second countable, locally compact and Hausdorff topologies on $G^s(Q)$ and $G^u(P)$, respectively, for which $G^s(Q)$ and $G^u(P)$ are {\'e}tale groupoids.
\end{thm}

\subsection{C*\texttt{-}algebras from Smale spaces}\label{sec:SmalespaceCalg}
For more details we refer to \cite{KPW,Putnam_Lec,PS}. We now briefly review the construction of the $C^*$-algebras of the stable groupoid $G^s(Q)$. The case of the unstable groupoid $G^u(P)$ is similar. 

Let $C_c(G^s(Q))$ denote the complex vector space of compactly supported continuous functions on $G^s(Q)$. We define a convolution and an involution on $C_c(G^s(Q))$ by 
\begin{equation}
\begin{split}
(f\cdot g) (v,w) &= \sum_{(v,z)\in  G^s(Q)} f(v,z)g(z,w),\\
f^*(v,w)&= \overline{f(w,v)}.
\end{split}
\end{equation}
%Note that every function in $C_c(G^s(Q))$ can be written as a finite sum of functions with support in an open set of the form $V^s(v,w,h^s,\eta,N)$ described in Theorem \ref{thm: stable bisections}. 

Recall that $P$ is a periodic orbit in $(X,\varphi)$ and denote the countable dense subset $X^s(P)\cap X^u(Q)$ of $X$ by $X^h(P,Q)$.
The \textit{fundamental representation} $\rho_s$ of $C_c(G^s(Q))$ on $\mathscr{H}=\ell^2(X^h(P,Q))$ is the faithful representation given, for $f\in C_c(G^s(Q)), \, \xi \in \mathscr{H}$, by 
\begin{equation}\label{eq:regular representation_fund}
\rho_s(f)\xi(v)= \sum_{(v,z)\in  G^s(Q)}f(v,z) \xi (z).
\end{equation}
The \textit{stable algebra} $\mathcal{S}(Q)$ is defined as the completion of $\rho_s(C_c(G^s(Q)))$ in $\mathcal{B}(\mathscr{H})$ and is isomorphic to the reduced $C^*$-algebra of the (amenable) groupoid $G^s(Q)$ \cite{PS}. 

In many instances we can consider $P=Q$. However, later on it will be important to choose $P\neq Q$, so that $X^h(P,Q)$ contains no periodic points. We will usually suppress the notation of $\rho_s$ and instead of writing $\rho_s(a)\xi$, for $a\in C_c(G^s(Q)), \, \xi \in \mathscr{H}$, we will simply write $a\xi$.

\begin{lemma}[{\cite[Lemma 3.3]{KPW}}]
Suppose that $V^s(v,w,h^s,\eta,N)$ is an open set as in Theorem \ref{thm: stable bisections} and $a\in C_c(G^s(Q))$ with $\text{supp}(a) \subset V^s(v,w,h^s,\eta,N)$. Then for every $x\in X^h(P,Q)$ we have $$a\delta_x= a(h^s(x),x)\delta_{h^s(x)},$$ if $x\in X^u(w,\eta)$, and $a\delta_x=0$, otherwise.
\end{lemma}

The homeomorphism $\varphi$ induces the automorphism $\Phi=\varphi \times \varphi $ on $G^s(Q)$, hence the automorphism $\alpha_s$ of $C_c(G^s(Q))$ given by 
\begin{equation}\label{eq:groupoid_autom_0}
\alpha_s(f)=f\circ \Phi^{-1},
\end{equation}
which extends on $\mathcal{S}(Q)$. For $a\in C_c(G^s(Q))$ with $\supp(a)\subset V^s(v,w,h^s,\eta,N)$ and $x\in X^h(P,Q)$ such that $h^s(\varphi^{-1}(x))$ is defined, we have 
\begin{equation}\label{eq:groupoid_autom}
\alpha_s(a)\delta_x=a(h^s\circ \varphi^{-1}(x), \varphi^{-1}(x))\delta_{\varphi \circ h^s \circ \varphi^{-1}(x)}.
\end{equation}
Moreover, the unitary $u$ on $\mathscr{H}$ given by $u\delta_x=\delta_{\varphi (x)}$ satisfies $\alpha_s(a)=uau^*$ for all $a\in \mathcal{S}(Q)$ \cite[Lemma 3.3.1]{Whittaker_PhD}, thus we can form the crossed product $\mathcal{S}(Q)\rtimes_{\alpha_s} \mathbb Z$ called the \textit{stable Ruelle algebra}.

\begin{definition}\label{def:SmaleCalg}
The \textit{stable} and \textit{unstable algebras} are denoted by $\mathcal{S}(Q)$ and $\mathcal{U}(P)$. The \textit{stable} and \textit{unstable Ruelle algebras} are the crossed products $\mathcal{R}^s(Q)=\mathcal{S}(Q)\rtimes_{\alpha_s} \mathbb Z$ and $\mathcal{R}^u(P)=\mathcal{U}(P)\rtimes_{\alpha_u} \mathbb Z$.
\end{definition}

\begin{example}
The stable, unstable algebras of a TMC $(\Sigma_M,\sigma_M)$ are approximately finite dimensional \cite[Section 8.5]{Putnam_Lec}. The stable, unstable Ruelle algebras are strongly Morita equivalent to the Cuntz-Krieger algebras $\mathcal{O}_M, \mathcal{O}_{M^t}$ \cite[Theorem 3.8]{CK}. 
\end{example}

\begin{remark}
Following \cite{PS}, the stable and unstable (Ruelle) algebras are separable, nuclear and satisfy the UCT. Irreducibility of $(X,\varphi)$ guarantees that $\mathcal{R}^s(Q),\,\mathcal{R}^u(P)$ are simple and purely infinite. If $(X,\varphi)$ is mixing, then $\mathcal{S}(Q),\,\mathcal{U}(P)$ are also simple. Also, they are $C^*$-stable (see \cite[Theorem A.2]{DY} and \cite[Corollary 4.5]{HRor}) and hence, for any periodic orbit $Q'$, we have $\mathcal{S}(Q)\cong \mathcal{S}(Q')$ and $\mathcal{R}^s(Q)\cong \mathcal{R}^s(Q')$, the unstable case being similar. Other fine structure properties are studied in \cite{DGY,DS2,DS}.
\end{remark}

\subsection{K-duality for Ruelle algebras}\label{sec:K-duality_Ruelle}
We review the construction of the $\Kt$-homology duality class of Kaminker\texttt{-}Putnam\texttt{-}Whittaker in \cite{KPW}, that we call the KPW-extension class since it is given on the level of extensions. Note that Ruelle algebras are nuclear. The description of the $\Kt$-theory duality class is not required for this paper. 

Assume now that the periodic orbits $P,Q$ satisfy $P\cap Q=\varnothing$. Note that $X^h(P,Q)$ has no periodic points. Also, recall the Hilbert space $\mathscr{H}=\ell^2(X^h(P,Q))$ and the open subsets of $G^s(Q), G^u(P)$ of Theorem \ref{thm: stable bisections}. We state the main theorem in \cite{KPW} and the lemmas leading to the construction of the KPW-extension class.

\begin{thm}[{\cite{KPW}}]\label{thm:Ruelleduality}
The Ruelle algebras $\mathcal{R}^s(Q)$ and $\mathcal{R}^u(P)$ are Spanier-Whitehead $\Kt$-dual. Moreover, if the $\Kt$-theory groups of $\mathcal{S}(Q)$ or $\mathcal{U}(P)$ have finite rank, then $\mathcal{R}^s(Q)\cong \mathcal{R}^u(P)$ and consequently both Ruelle algebras are Poincar{\'e} duality algebras.
\end{thm}

Due to transversality (existence of bracket map) one has the following result. 

\begin{lemma}[{\cite[Lemma 6.1]{KPW}}]\label{lem:stableunstablerankone}
For every $a\in C_c(G^s(Q)), \, b\in C_c(G^u(P))$ with $\supp(a)\subset V^s(v,w,h^s,\eta,N),\, \supp(b)\subset V^u(v',w',h^u,\eta',N')$, it holds that $\rank(ab)$ and $\rank(ba)$ are at most one. Consequently, for every $a\in \mathcal{S}(Q),\, b\in \mathcal{U}(P)$ we have that the products $ab,\, ba\in \mathcal{K}(\mathscr{H})$.
\end{lemma}

Using the expanding and contracting nature of the dynamics, together with the fact that $P\cap Q=\varnothing$, we have the following. 

\begin{lemma}[{\cite[Lemma 6.2]{KPW}}]\label{lem:minusinftylimit}
For every $a\in C_c(G^s(Q)), \, b\in C_c(G^u(P))$ with $\supp(a)\subset V^s(v,w,h^s,\eta,N), \,\supp(b)\subset V^u(v',w',h^u,\eta',N')$, there is $M\in \mathbb N$ such that $\alpha_s^{-n}(a)b=b\alpha_s^{-n}(a)=0$, for $n\geq M$. As a result, for every $a\in \mathcal{S}(Q), \, b\in \mathcal{U}(P)$ we have that $$\lim_{n\to +\infty} \alpha_s^{-n}(a)b=0 \,\,\,\, \text{and} \,\,\,\, \lim_{n\to +\infty} b\alpha_s^{-n}(a)=0.$$
\end{lemma}

The next lemma plays a prominent role in the construction of the KPW-extension class. In Subsection \ref{sec:Lip_Ruelle_Smooth_Ext} we provide a refined version of it.

\begin{lemma}[{\cite[Lemma 6.3]{KPW}}]\label{lem:plusinftyKPW}
For any $a\in \mathcal{S}(Q)$ and $b\in \mathcal{U}(P)$, we have 
\begin{align*}
&\lim_{n\to +\infty}\|\alpha_s^n(a)b-b\alpha_s^n(a)\|=0,\\
&\lim_{n\to +\infty}\|\alpha_s^n(a)\alpha_u^{-n}(b)-\alpha_u^{-n}(b)\alpha_s^n(a)\|=0.
\end{align*}
\end{lemma}

Let us consider the inflated representation $\overline{\rho_{s}}:\mathcal{R}^s(Q)\to \mathcal{B}(\mathscr{H}\otimes \ell^2(\mathbb Z))$ given by 
\begin{equation}\label{eq:inflatedstableRuelle}
a\mapsto \bigoplus_{n\in \mathbb Z} \alpha_s^n(a)\,\,\,\, \text{and} \,\,\,\, u\mapsto 1\otimes B,
\end{equation}
where $a\in \mathcal{S}(Q)$, $u$ is the unitary given by $u\delta_x=\delta_{\varphi(x)}$ and $B$ is the left bilateral shift given by $B\delta_n=\delta_{n-1}$. Moreover, let $\overline{\rho_{u}}:\mathcal{R}^u(P)\to \mathcal{B}(\mathscr{H}\otimes \ell^2(\mathbb Z))$ be the representation given, for $b\in \mathcal{U}(P)$, by
\begin{equation}\label{eq:inflatedusntableRuelle}
b\mapsto b\otimes 1 \,\,\,\, \text{and} \,\,\,\, u\mapsto u\otimes B^*.
\end{equation}
Both representations are faithful and the commutators $[\overline{\rho_s}(a),\overline{\rho_u}(u)]$, $[\overline{\rho_u}(b),\overline{\rho_s}(u)]$ and $[\overline{\rho_s}(u), \overline{\rho_u}(u)]$ are zero. We then have the following result.

\begin{lemma}[{\cite[Lemma 6.4]{KPW}}]\label{lem:commutationRuelles}
It holds that $\overline{\rho_s}(\mathcal{R}^s(Q))$ and $\overline{\rho_u}(\mathcal{R}^u(P))$ commute modulo compact operators on $\mathscr{H}\otimes \ell^2(\mathbb Z).$
\end{lemma}

Due to nuclearity, one can then form the extension (or rather the Busby invariant) $\tau_{\Delta}: \mathcal{R}^s(Q)\otimes \mathcal{R}^u(P)\to \mathcal{Q}(\mathscr{H}\otimes \ell^2(\mathbb Z))$ that on elementary tensors and generators is
\begin{equation}\label{eq: tau_Delta}
\begin{split}
\tau_{\Delta}(au^j\otimes bu^{j'})&= (\overline{\rho_s}\cdot \overline{\rho_u})(a u^j\otimes b u^{j'})+ \mathcal{K}(\mathscr{H}\otimes \ell^2(\mathbb Z))\\
&=(\sum_{n\in \mathbb Z}\alpha_s^n(a)\otimes e_{n,n})(1\otimes B^j)(b\otimes 1) (u^{j'}\otimes B^{-j'})+ \mathcal{K}(\mathscr{H}\otimes \ell^2(\mathbb Z))\\
&= \sum_{n\in \mathbb Z}\alpha_s^n(a)bu^{j'}\otimes e_{n,n+j-j'} + \mathcal{K}(\mathscr{H}\otimes \ell^2(\mathbb Z))
\end{split}
\end{equation}
where $a\in \mathcal{S}(Q),\, b\in \mathcal{U}(P),\, j,j'\in \mathbb Z$, and $e_{n,m}$ are matrix units. This extension is non-trivial owing to $\mathcal{R}^s(Q)\otimes \mathcal{R}^u(P)$ being simple and the next lemma.

\begin{lemma}[{\cite[Lemma 4.4.13]{Whittaker_PhD}}]
There are $a\in \mathcal{S}(Q)$ and $b\in \mathcal{U}(P)$ such that the operator $(\overline{\rho_s}\cdot \overline{\rho_u})(a\otimes b)$ is not compact.
\end{lemma}

\begin{definition}[{\cite[Def. 6.6]{KPW}}]\label{def:KPWextclass}
The \textit{KPW-extension class} is represented by the \textit{KPW-extension} $\tau_{\Delta}$ (\ref{eq: tau_Delta}). The corresponding class $\Delta\in \KKt_1(\mathcal{R}^s(Q)\otimes \mathcal{R}^u(P),\mathbb C)$ will be called the $\Kt$\textit{-homology duality class}.
\end{definition}

\section{Slant products and K-homological summability}\label{sec:Slant_prod_Unif_Smooth}
\enlargethispage{\baselineskip}
Let $A,B$ be separable $C^*$-algebras and $x\in \KKt_1(A\otimes B,\mathbb C)$. First we focus on the Fredholm module picture of the slant product $\otimes_A x:\KKt_0(\mathbb C, A)\to \KKt_1(B,\mathbb C)$. In the literature we were able to find this computation only for unital $C^*$-algebras, see \cite[Prop. 2.1.3, Prop. 2.1.6]{GM}. However, if $A$ or $B$ are not unital, certain technical difficulties arise which are described below. Nevertheless, if $A$ is also simple and purely infinite, these can be circumvented. This suffices for the purpose of this paper. Further, if $x$ is represented by an extension formed by commuting (modulo compacts) representations of $A$ and $B$, we find a tractable description of the slant product $\otimes_A x:\KKt_*(\mathbb C, A)\to \KKt_{*+1}(B,\mathbb C)$ in terms of Fredholm module representatives. Finally, by developing tools from holomorphic functional calculus, if the later map is surjective, we deduce uniform summability conditions for the $\Kt$-homology of $B$.

\subsection{Slant products for simple, purely infinite C*-algebras}\label{sec:SlantprodDuality}
Suppose that $A,B$ are separable $C^*$-algebras and let $(H,\rho, F)$ be an odd Fredholm module over $A\otimes B$, and $e\in \mathcal{M}_n(A)=\mathcal{M}_n(\mathbb C)\otimes A$ be a projection. In addition, let $(u_n)_{n\in \mathbb N}$ be an approximate identity for $B$ such that 
\begin{equation}\label{eq:approxid}
u_{n+1}u_n=u_n.
\end{equation}
It holds that the sequence of positive elements $(\id_{\mathcal{M}_n(\mathbb C)}\otimes \rho)(e\otimes u_n)$ is bounded and increasing. Therefore, it converges strongly to a positive operator $P_e\in B(\mathbb C^n \otimes H)$. Using (\ref{eq:approxid}) it is straightforward to show that $P_e$ is a projection. 

Let $H_e$ be the Hilbert space $P_e(\mathbb C^n \otimes H)$ and $\rho_e:B\to \mathcal{B}(H_e)$ be the representation given by $\rho_e(b)=(\id_{\mathcal{M}_n(\mathbb C)}\otimes \rho)(e\otimes b)$. This representation is well-defined since for every $b\in B$ we have that 
\begin{equation}\label{eq:e_rep}
\rho_e(b)P_e=\rho_e(b)=P_e\rho_e(b).
\end{equation}
Finally, we consider the operator 
\begin{equation}\label{eq:e_oper}
F_e=P_e(\id_{\mathbb C^n}\otimes F)P_e.
\end{equation}

\begin{lemma}\label{lem:slantcalc1}
The triple $(H_e,\rho_e,F_e)$ is an odd Fredholm module over $B$.
\end{lemma}

\begin{proof}
Since the triple $(\mathbb C^n\otimes H,\id_{\mathcal{M}_n(\mathbb C)}\otimes \rho, \id_{\mathbb C^n}\otimes F)$ is an odd Fredholm module over $\mathcal{M}_n(\mathbb C)\otimes A\otimes B$, it suffices to prove the statement only for $n=1$. Also, for brevity, if $S,T\in \mathcal{B}(H)$ and $S-T\in \mathcal{K}(H)$ we will write $S\sim_{\mathcal{K}} T$. Further, all the operators in $\mathcal{B}(H_e)$ can be thought to be in $\mathcal{B}(H)$ by letting them to be zero on $H_e^{\perp}$. Let $b\in B$, then we have that

\begin{align*}
\rho_e(b)(F_e^*-F_e)&=\rho_e(b)(P_eF^*P_e-P_eFP_e)\\
&=\rho_e(b)F^*P_e-\rho_e(b)FP_e\\
&\sim_{\mathcal{K}} (F^*-F)\rho_e(b)\\
&\sim_{\mathcal{K}} 0.\\
\intertext{Since $\rho_e(b)(F_e^*-F_e)$ is zero on $H_e^{\perp}$ and maps in $H_e$, we have $\rho_e(b)(F_e^*-F_e)\in \mathcal{K}(H_e)$. Similarly, it holds that $\rho_e(b)(F_e^2-P_e)\in \mathcal{K}(H_e)$. Finally, }
[F_e,\rho_e(b)]&=P_eFP_e\rho_e(b)-\rho_e(b)P_eFP_e\\
&=P_eF\rho_e(b)P_e-P_e\rho_e(b)FP_e\\
&=P_e[F,\rho_e(b)]P_e\\
&\sim_{\mathcal{K}} 0.
\end{align*}
As a result, $[F_e,\rho_e(b)]\in \mathcal{K}(H_e)$.
\end{proof}

To compute the slant products we need a good description of the $\Kt$-theory classes over $A$. If $A$ is unital, it is well-known that $\Kt_0(A)$ can be alternatively given by formal differences of isomorphism classes of finitely generated projective modules over $A$. More precisely, if $e\in \mathcal{M}_n(A)$ is a projection, its Murray-von Neumann equivalence class $[e]$ corresponds to the isomorphism class $[eA^n]$. Following \cite[Section 6]{Kasparov5}, the class $[eA^n]$ corresponds to the equivalence class of the Kasparov $(\mathbb C, A)$-bimodule $(eA^n,j,0)$, where $j:\mathbb C\to \mathcal{B}_A(pA^n)$ is given by $j(\lambda)=\lambda \cdot \id$. If $A$ is not unital, the above description does not hold and one has to consider relative $\Kt$-groups for computing these slant products. This seems quite elaborate and is not needed here. 

Suppose now that $A$ is simple and purely infinite. Then, there is a non-zero full projection $p\in A$ \cite[Lemma 1.1]{Brown}. Also, the $C^*$-subalgebra $pAp$ is simple, purely infinite, unital, $C^*$-stably isomorphic to $A$ and the inclusion map $\psi: pAp\to A$ induces an isomorphism $\psi_*:\KKt_*(\mathbb C, pAp)\to \KKt_*(\mathbb C,A)$ \cite[Prop. 1.2]{Paschke}. Moreover, from \cite{Cuntz}, the $\Kt$-theory of $pAp$ is given by
\begin{equation}\label{eq:K0pureinf}
\begin{split}
\Kt_0(pAp)&=\{[e]:e\enspace \text{is a non-zero projection in}\enspace pAp\},\\
\Kt_1(pAp)&=U(pAp)/U_0(pAp),
\end{split}
\end{equation}
where $U(pAp)$ is the unitary group and $U_0(pAp)$ is the connected component of $p$. Further, the class $[e]$ of a projection $e\in pAp$ corresponds to the isomorphism class of finitely generated projective modules $[e(pAp)]$, since $pAp$ is unital. In turn, this corresponds to the equivalence class of Kasparov $(\mathbb C,pAp)$-bimodules $[e(pAp),j,0]$. Composing with $\psi_*$, we obtain the isomorphism $\Kt_0(pAp)\to \KKt_0(\mathbb C,A)$,
\begin{equation}\label{eq:descrofKK0}
[e]\mapsto [e(pAp)\otimes_{\psi} A, j\otimes_{\psi} \id, 0].
\end{equation}

\begin{lemma}\label{lem:slantbimodule}
Let $A$ be a $C^*$-algebra with a full projection $p\in A$ and $e\in pAp$ be any projection. The map $U:e(pAp)\otimes_{\psi} A\to eA$, $ea'\otimes_{\psi}a\mapsto ea'a$, for $a\in A$ and  $a'\in pAp$, extends to an isomorphism of Hilbert $A$-modules.
\end{lemma}

\begin{proof}
Clearly, the map $U$ extends by linearity to an $A$-module map on the algebraic tensor product $e(pAp)\otimes_{pAp} A$. Similarly, it holds that $U$ preserves the inner products and thus it extends to an isometry on the complete tensor product $e(pAp)\otimes_{\psi} A$. Moreover, we have that $epApA=eApA$. Since $p$ is full, for every $c\in A$ there is $(c_n)_{n\in \mathbb N} \subset ApA$ such that $\lim_nc_n=c$. Consequently, $\lim_n epc_n=\lim_n ec_n= ec\in eA$, meaning that $U$ has dense range.
\end{proof}

The slant products will be computed in the context of $\Kt$-theory and $\Kt$-homology. Also, the isomorphism $\psi_*$ is considered in the context of $\Kt$-theory, meaning that $\psi_*:\Kt_*(pAp)\to \Kt_*(A)$. 

\begin{prop}\label{prop:slantprodK0theory}
Let $A$ be a separable, simple, purely infinite $C^*$-algebra and $B$ be a separable $C^*$-algebra. Let $p\in A$ be a non-zero projection and $(H,\rho,F)$ be an odd Fredholm module over $A\otimes B$, where $\rho$ is non-degenerate. The slant product $\otimes_A [H,\rho,F]: \Kt_0(A)\to \Kt^1(B)$ is given by $$\psi_*([e])\mapsto [H_e,\rho_e,F_e], \enspace \text{where} \enspace 0\neq e\in pAp \enspace \text{is a projection}.$$ 
\end{prop}

\begin{proof}
Using the map in (\ref{eq:descrofKK0}) we see that the class $\psi_*([e])$ is represented by the Kasparov $(\mathbb C, A)$-bimodule $$(e(pAp)\otimes_{\psi} A, j\otimes_{\psi} \id, 0).$$ The unitary $U:e(pAp)\otimes_{\psi} A\to eA$ of Lemma \ref{lem:slantbimodule} makes this bimodule unitarily equivalent to $(eA,j',0)$ with $j':\mathbb C \to \mathcal{B}_A(eA)$ given by $j'(\lambda)(ea)=\lambda ea$, for $a\in A$. Hence, $\psi_*([e])\otimes_A[H,\rho,F]$ is equal to the product $\tau_B([eA,j',0])\otimes_{A\otimes B}[H,\rho,F]$ which, since the operator on the left $\KKt$-class is zero, is given by $$[(eA\otimes B)\otimes_{\rho}H,(j'\otimes \id)\otimes_{\rho} \id ,G],$$ where $G$ is any $F$-connection for $eA\otimes B$, see \cite[Def. 2.2.4]{JT} and  \cite[Def. 2.2.7]{JT}. Before constructing the operator $G$, we briefly mention the definition of a connection in our context. For every $x\in eA\otimes B$, let $T_x\in \mathcal{B}(H, (eA\otimes B)\otimes_{\rho} H)$ be given by $T_x(\xi)=x\otimes_{\rho} \xi$ and then $T_x^*(y\otimes_{\rho} \xi)=\rho (\langle x,y \rangle_{A\otimes B})\xi.$ For $G$ to be an $F$-connection for $eA\otimes B$ we should have that, for every $x\in eA\otimes B$, 
\begin{align*}
T_xF - GT_x &\in \mathcal{K}(H, (eA\otimes B)\otimes_{\rho} H),\\
FT_x^* - T_x^*G &\in \mathcal{K}((eA\otimes B)\otimes_{\rho} H,H).
\end{align*}
\enlargethispage{\baselineskip}
\enlargethispage{\baselineskip}
Let us first find a nicer description of the Hilbert space $(eA\otimes B)\otimes_{\rho} H)$. The map $W: (eA\otimes B)\otimes_{\rho} H\to H_e$ given on simple tensors by  $x\otimes_{\rho} \xi \mapsto \rho(x)\xi$ is unitary. Indeed, it is straightforward to see that $W$ preserves the inner products. Also, it has a dense range since $\rho(eA\otimes B)=P_e \rho(A\otimes B)$ and hence $\overline{\rho(eA\otimes B)H}= \overline{P_e\rho(A\otimes B)H}=P_e(\overline{\rho(A\otimes B)H})=H_e$. The claim now is that $W^{-1}F_eW$ is an $F$-connection for $eA\otimes B$. Before proving the claim we should note that for $x\in eA\otimes B$ it holds that $P_e\rho(x)=\rho(x)$ and  $\rho(x^*)P_e=\rho(x^*).$ We have that $T_xF-W^{-1}F_eW T_x \in \mathcal{K}(H, (eA\otimes B)\otimes_{\rho} H)$ if and only if $WT_xF-F_eWT_x\in \mathcal{K}(H,H_e).$ To prove the latter, let $\xi \in H$ and then,
\begin{align*}
(WT_xF-F_eWT_x)\xi&= W(x\otimes_{\rho} F\xi)-F_eW(x\otimes_{\rho}\xi)\\
&=\rho(x)F\xi-F_e\rho(x)\xi\\
&= P_e\rho(x)F\xi-P_eF\rho(x)\xi\\
&= P_e[\rho(x),F]\xi.
\end{align*}
Similarly, it holds that $FT_x^*-T_x^*W^{-1}F_eW \in \mathcal{K}((eA\otimes B)\otimes_{\rho} H,H)$ if and only if the operator $FT_x^*W^{-1}-T_x^*W^{-1}F_e\in \mathcal{K}(H_e,H).$ Recall that $P_e$ is the strong operator limit of $(\rho(e\otimes u_n))_{n\in \mathbb N}$, where $(u_n)_{n\in \mathbb N}$ is an approximate identity for $B$ satisfying $u_{n+1}u_n=u_n$. For $\xi \in H_e$ we have,
\begin{align*}
(FT_x^*W^{-1}-T_x^*W^{-1}F_e)\xi &= FT_x^*W^{-1}P_e\xi-T_x^*W^{-1}F_e\xi\\
&= \lim_{n\to \infty}(FT_x^*((e\otimes u_n)\otimes_{\rho} \xi)-T_x^*((e\otimes u_n)\otimes_{\rho} F\xi))\\
&= \lim_{n\to \infty}(F\rho(\langle x, e\otimes u_n\rangle_{A\otimes B})\xi -\rho(\langle x, e\otimes u_n\rangle_{A\otimes B})F\xi)\\
&= [F,\rho(x^*)]\xi.
\end{align*}
This proves the claim, that $W^{-1}F_eW$ is an $F$-connection for $eA\otimes B$. The proof of the proposition finishes by observing that the representation $j'':\mathbb C\otimes B\to \mathcal{B}(H_e)$ given by conjugation with $W$, that is, $$j''(\lambda \otimes b)= W(j'\otimes \id)\otimes_{\rho} \id)(\lambda \otimes b) W^{-1},$$ is the same with the representation $\rho_e:B\to \mathcal{B}(H_e)$ given by $\rho_e(b)\to \rho(e\otimes b)$. This can be proved in a similar fashion as above, by evaluating $j''(\lambda \otimes b)$ on some $\xi \in H_e$ which is written as $P_e\xi$. Then, the result follows from the strong limit description of $P_e$. Consequently, the desired $\KKt$-class $[(eA\otimes B)\otimes_{\rho}H,(j'\otimes \id)\otimes_{\rho} \id ,W^{-1}F_eW]$ is equal to $[H_e,\rho_e,F_e]$.
\end{proof}

Using the notation of Proposition \ref{prop:slantprodK0theory}, with a simple calculation, we obtain the following result.

\begin{cor}\label{cor:slantprodext}
Let $A$ be a separable, simple, purely infinite $C^*$-algebra and $B$ be a separable $C^*$-algebra. Let $p\in A$ be a non-zero projection and $\tau:A\otimes B\to \mathcal{Q}(H)$ be an invertible extension. The slant product $\otimes_A [\tau]: \Kt_0(A)\to \Ext^{-1}(B)$ is given by $$\psi_*([e])\mapsto [\tau_e],$$ where $0\neq e\in pAp$ is a projection and $\tau_e(b)=\tau(e\otimes b)$, for $b\in B$.
\end{cor}

Slant products in terms of extensions can be particularly nice. Specifically, we have the following.

\begin{prop}\label{prop:slantprodKtheory}
Let $A$ be a separable, simple, purely infinite $C^*$-algebra and $B$ be a separable $C^*$-algebra such that $A\otimes B= A\otimes_{\max} B$. Let $p\in A$ be a non-zero projection and $\tau:A\otimes B\to \mathcal{Q}(H)$ be an invertible extension that on elementary tensors is given by $\tau(a\otimes b)= \rho_A(a)\rho_B(b)+\mathcal{K}(H),$ where $\rho_A:A\to \mathcal{B}(H),\, \rho_B:B\to \mathcal{B}(H)$ are representations that commute modulo $\mathcal{K}(H)$. Then, the slant product
\begin{enumerate}[(i)]
\item $\otimes_A [\tau]: \Kt_0(A)\to \Kt^1(B)$ is given by $$\psi_*([e])\mapsto [H,\rho_B,2\rho_A(e)-1],$$   where $0\neq e\in pAp$ is a projection;
\item $\otimes_A [\tau]: \Kt_1(A)\to \Kt^0(B)$ is given by $$\psi_*([v])\mapsto [H\oplus H, \rho_B\oplus \rho_B, \begin{pmatrix}
0&F^*\\
F&0
\end{pmatrix}],$$ where  $v\in pAp$ is a unitary and the operator $F=\rho_A(v)+1-\rho_A(p)$, if $[v]\neq 0$. If $[v]=0$ we can simply choose $F=1$.
\end{enumerate}
\end{prop}

\begin{proof}
Let $e\in pAp$ be a non-zero projection. By Corollary \ref{cor:slantprodext}, the class $\psi_*([e]) \otimes_A [\tau]$ in $\Kt^1(B)$ is represented by the extension $\tau_e:B\to \mathcal{Q}(H)$ defined as $\tau_e(b)=\tau(e\otimes b)$. Here $\tau_e$ is, in particular, a Toeplitz extension since $$\tau_e(b)=\rho_A(e)\rho_B(b)+\mathcal{K}(H)= \rho_A(e)\rho_B(b)\rho_A(e)+\mathcal{K}(H).$$ Hence, the odd Fredholm module $(H,\rho_B,2\rho_A(e)-1)$ over $B$ represents $\psi_*([e]) \otimes_A [\tau]$. This completes the proof of part (i).

For part (ii) consider a unitary $v\in pAp$. The case where $[v]=0$ is straightforward. If $[v]\neq 0$, the spectrum $\sigma(v)=\mathbb T$. Consider the composition of maps $$\overline{v}:C_0(\mathbb R)\to C_0(\mathbb T\setminus \{1\})\to C(\mathbb T)\to pAp,$$ where the first is induced by the inverse of the Cayley transform $\mathbb R\to  \mathbb T\setminus \{1\}$ given by $t\mapsto (t-i)/(t+i)$, the second by inclusion as an ideal and the third by continuous functional calculus for $v$. Also, the induced map $\overline{v}_*:\Kt_1(C_0(\mathbb R))\to \Kt_1(pAp)$ maps the Bott element $\beta$ to $[v]$, see \cite[Section 6.2]{Emersonbook}. Therefore, we have that 
\begin{align*}
\psi_*([v])\otimes_A[\tau]&=(\psi\circ \overline{v})_*(\beta) \otimes_A [\tau]\\
&=\tau_B((\psi\circ \overline{v})_*(\beta))\otimes_{A\otimes B}[\tau]\\
&= \tau_B(\beta \otimes_{C_0(\mathbb R)} [\psi\circ \overline{v}])\otimes_{A\otimes B}[\tau]\\
&= \tau_B(\beta)\otimes_{C_0(\mathbb R)\otimes B}\tau_B([\psi\circ \overline{v}])\otimes_{A\otimes B}[\tau]\\
&= \tau_B(\beta)\otimes_{C_0(\mathbb R)\otimes B} ((\psi\circ \overline{v})\otimes \text{id}_B)^*([\tau])
\end{align*}
where $((\psi\circ \overline{v})\otimes \text{id}_B)^*([\tau])$ is represented by the extension $\tau_v:C_0(\mathbb R)\otimes B\to \mathcal{Q}(H)$ that on elementary tensors is given by $\tau_v(f\otimes b)= \tau((\psi\circ \overline{v})(f)\otimes b)=\tau(\overline{v}(f)\otimes b)$, since $\psi$ can now be dropped. In particular, we have that $$\tau_v(f\otimes b)= (\rho_A\circ  \overline{v})(f)\rho_B(b)+\mathcal{K}(H).$$ We also have that the function $g\in C_0(\mathbb R)$ defined as $g(t)=-2i/(t+i)$ satisfies $\overline{v}(g)=v-p$ and hence from \cite[Lemma 7]{EmersonDuality} we obtain that the class $\psi_*([v])\otimes_A[\tau]$ is represented by the (balanced) even Fredholm module $$(H\oplus H, \rho_B\oplus \rho_B, \begin{pmatrix}
0&F^*\\
F&0
\end{pmatrix}),$$ where $F=\rho_A(v)+1-\rho_A(p)$.
\end{proof}

\begin{remark}
The condition $A\otimes B= A\otimes_{\max} B$ in Proposition \ref{prop:slantprodKtheory} guarantees that the extension $\tau$ is well-defined.
\end{remark}

\subsection{Holomorphic functional calculus and summability}\label{sec:Hol_Fun_Cal} 
We extend a lemma of Connes, regarding the ideals $\mathcal{L}^p(H)$ for $p\geq1$, to general symmetrically normed ideals. Also, we extend it to include the case $p\in (0,1)$. However, the latter endeavour is more elaborate since the corresponding ideals are quasi-normed. After some lemmas about holomorphic functional calculus we obtain Proposition \ref{prop:Ext_smoothness_prop} which is about uniform $\Kt$-homological summability. We give the following definition which is a combination of Definitions 1 of \cite[p.92]{ConnesDifGeom} and \cite[p.285]{Connes}.

\begin{definition}\label{def:Hol_fun_cal}
Let $D$ be a Banach algebra and $\mathcal{D}$ a subalgebra with unitisation $\widetilde{\mathcal{D}}\subset \widetilde{D}$. We say that $\mathcal{D}$ is \textit{holomorphically closed} in $D$ if, for every $d\in \widetilde{\mathcal{D}}$ we have $$f(d)\in \widetilde{\mathcal{D}},$$ for any function $f$ that is holomorphic in an open neighbourhood of the spectrum of $d$ in $\widetilde{D}$. Moreover, we say that $\mathcal{D}$ is \textit{holomorphically stable} in $D$ if, for every $n\in \mathbb N$, the subalgebra $\mathcal{M}_n(\widetilde{\mathcal{D}})$ is holomorphically closed in $\mathcal{M}_n(\widetilde{D})$.
\end{definition}

The following theorem summarises a result of Schweitzer and the Density Theorem of Karoubi, see \cite{Schweitzer}, \cite[Prop. 2, p.285]{Connes} and \cite[p.92]{ConnesDifGeom}.

\begin{thm}[Density Theorem]\label{thm:Density_theorem}
Suppose that $\mathcal{D}$ is a holomorphically closed dense subalgebra of a Banach algebra $D$. Then, $\mathcal{D}$ is holomorphically stable in $D$. Moreover, the inclusion $i:\mathcal{D}\to D$ induces the isomorphism $$i_*:\Kt_*(\mathcal{D})\to \Kt_*(D).$$
\end{thm}

Let $H$ be a separable Hilbert space and $\mathcal{I}\subset \mathcal{B}(H)$ be a symmetrically normed ideal. The next lemma is basically Proposition $5a$ of \cite{ConnesDifGeom}, but stated more generally. Its proof follows from the symmetric property of $\mathcal{I}$, the fact that for $R\in \mathcal{I}$ one has $\|R\|_{\mathcal{B}(H)}\leq \|R\|_{\mathcal{I}}$ \cite[Chapter III]{GK}, and that every continuous function $g:\Gamma\to \mathcal{I}$, where $\Gamma$ is a simple closed $C^1$-curve, is Riemann integrable since $\mathcal{I}$ is a Banach space.

\begin{lemma}\label{lem:Banach_alg_hol_cal}
For every $S,T\in \mathcal{B}(H)$ such that $[S,T]\in \mathcal{I}$ and every function $f$ that is holomorphic in a neighbourhood of the spectrum $\sigma(S)$, we have $$[f(S),T]\in \mathcal{I}.$$
\end{lemma}

The situation for the symmetrically quasi-normed ideals $\mathcal{L}^{p}(H)$, for $p\in (0,1)$, is a bit different. Specifically, it is no longer true that every continuous function $g:\Gamma\to \mathcal{L}^p(H)$, for $\Gamma$ being a simple closed $C^1$-curve, is Riemann integrable. An example with values in $\ell^p(\mathbb N)$ can be found in \cite[Section 3]{Gramsch}. Even assuming that $g$ is holomorphic (differentiable) in a neighbourhood of $\Gamma$ does not suffice. One has to assume that $g$ is analytic, meaning that it has local power series expansions. We note that in the context of quasi-Banach spaces, analyticity is stronger than holomorphicity because tools like the Hahn-Banach Theorem do not necessarily hold. 

Quasi-Banach spaces can be studied in the context of $r$-Banach spaces, $r\in (0,1]$. Let us pause for a moment and give precise definitions. 
\enlargethispage{\baselineskip}

A \textit{quasi-normed space} is a complex vector space $V$ equipped with a \textit{quasi-norm} $\vertiii{\cdot}:V\to [0,\infty)$ which, for all $v,w\in V$, satisfies
\begin{enumerate}[(i)]
\item $\vertiii{v} =0$ only if $v=0$;
\item $\vertiii{zv}=|z|\vertiii{v}$, for $z\in \mathbb C$;
\item $\vertiii{v+w}\leq K(\vertiii{v}+\vertiii{w})$, for some $K\geq 1$.
\setcounter{nameOfYourChoice}{\value{enumi}}
\end{enumerate}

For $r\in (0,1],$ an $r$\textit{-normed space} is a complex vector space $V$ equipped with an $r$\textit{-norm} $\vertiii{\cdot}_r:V\to [0,\infty)$ that satisfies (i), (ii) and for every $v,w\in V$ one has

\begin{enumerate}[(i)]
\setcounter{enumi}{\value{nameOfYourChoice}}
\item $\vertiii{v+w}_r^r\leq \vertiii{v}_r^r+\vertiii{w}_r^r.$
\end{enumerate}

A \textit{quasi-Banach space} is a complete quasi-normed space, and an $r$\textit{-Banach space} is a complete $r$-normed space. If $V$ is $r$-normed with $\vertiii{\cdot}_r$, then $\vertiii{\cdot}_r$ is a quasi-norm with constant $K=2^{1/r-1}.$ Conversely, if $\vertiii{\cdot}$ is a quasi-norm on $V$ with constant $K$, the Aoki-Rolewicz Renorming Theorem \cite{Rolewicz} yields an equivalent $r$-norm $\vertiii{\cdot}_r$ on $V$, where also $K=2^{1/r-1}$, that for every $v\in V$ it holds 
\begin{equation}\label{eq:Renorming}
\vertiii{v}_r\leq \vertiii{v}\leq 2^{1/r} \vertiii{v}_r.
\end{equation}

Returning to our problem of integrating on simple closed $C^1$-curves, we mention that Riemann integration for quasi-Banach (or $r$-Banach) valued functions is defined in exactly the same way as for Banach-valued functions. Let $V$ be an $r$-normed space with $r$-norm $\vertiii{\cdot}_r$ and $\Gamma$ be a simple closed $C^1$-curve in $\mathbb C$. The linear space of Riemann integrable $V$-valued functions on $\Gamma$ is denoted by $\mathscr{R}(\Gamma,V)$. Moreover, a function $g:\Gamma\to V$ has $r$\textit{-expansion on} $\Gamma$ if, for every $k\geq 0$, there are $a_k\in V$ and Riemann integrable functions $f_k:\Gamma\to \mathbb C$ such that, 
$$g(z)=\sum_{k=0}^{\infty}a_kf_k(z)\enspace \text{and} \enspace \sum_{k=0}^{\infty}\vertiii{a_k}_r^r \|f_k\|_{\infty}^r<\infty.
$$ Let us denote by $\mathscr{E}(\Gamma, V)$ the linear space of functions $\Gamma\to V$ that have $r$-expansions.

\begin{prop}[{\cite[Prop. 3.7]{Gramsch}}]\label{prop:quasi_Banach_Integration}
The linear space $\mathscr{E}(\Gamma, V)$ is contained in $\mathscr{R}(\Gamma,V)$.
\end{prop}

For the next two results, let $p\in (0,1)$ and consider the ideal $\mathcal{L}^p(H)$ equipped with the Schatten quasi-norm $\|\cdot \|_p$. Its constant is $K=2^{1/p}$, and the renorming process gives an equivalent complete $r$-norm $|\cdot |_r $, where $r= p/(p+1)$. The first result is the key tool for extending Lemma \ref{lem:Banach_alg_hol_cal} to $\mathcal{L}^p(H)$.

\begin{prop}\label{prop:Schatten_p_zero_one}
Let $S,T\in \mathcal{B}(H)$ such that $[S,T]\in \mathcal{L}^p(H)$, and let $\Gamma$ be a simple closed $C^1$-curve in the resolvent $\mathbb C\setminus \sigma(S)$. Then, the map $F:\Gamma\to \mathcal{L}^p(H)$ given by $$F(z)=[(z-S)^{-1},T],$$ is Riemann integrable.
\end{prop}

\begin{proof} First note that $F$ is well-defined since for every $z\in \mathbb C\setminus \sigma(S)$ it holds
\begin{equation}\label{eq:hol_fun_commutator}
[(z-S)^{-1},T]= (z-S)^{-1}[S,T](z-S)^{-1}.
\end{equation}
Instead of showing directly that $F$ is Riemann integrable with respect to the quasi-norm $\|\cdot \|_p$, we will prove that it is Riemann integrable with respect to $|\cdot |_r $ and then use inequalities (\ref{eq:Renorming}). Specifically, we will show that $F$ is a finite sum of functions with $r$-expansions on $\Gamma$, and hence from Proposition \ref{prop:quasi_Banach_Integration} we obtain that $F$ is Riemann integrable.

To begin, let us also denote by $F$ the extension of the map on $\mathbb C\setminus \sigma(S)$. Observe that the map $z\mapsto (z-S)^{-1}$ defined on $\mathbb C\setminus \sigma(S)$ is analytic in $\mathcal{B}(H)$. Indeed, let $w\in \mathbb C\setminus \sigma(S)$ and for every $z\in \mathbb C$ with $|z-w|<1/ \|(w-S)^{-1}\|$ we have that $$(z-S)^{-1}=\sum_{k=0}^{\infty}s_k(w)(w-z)^k,$$ where $s_k(w)=((w-S)^{-1})^{k+1}.$ 

Moreover, using (\ref{eq:hol_fun_commutator}), for $w\in \mathbb C\setminus \sigma(S)$ and $z\in \mathbb C$ with $|z-w|<1/ \|(w-S)^{-1}\|$ we have, 
\begin{equation}\label{eq:Schatten_p_zero_one1}
F(z)=\sum_{n=0}^{\infty}c_n(w)(w-z)^n,
\end{equation}
where $$c_n(w)=\sum_{j=0}^n s_j(w)[S,T]s_{n-j}(w).$$ The expansion (\ref{eq:Schatten_p_zero_one1}) converges in $\mathcal{B}(H)$ and hence $F$ is analytic in $\mathcal{B}(H)$. Clearly, the expansion (\ref{eq:Schatten_p_zero_one1}) around $w\in \mathbb C\setminus \sigma(S)$ holds for all $z\in \mathbb C$ such that 
\begin{equation}\label{eq:Schatten_p_zero_one2}
|z-w|<\frac{1}{2K\|(w-S)^{-1}\|}.
\end{equation}
This open ball will be denoted by $B(w).$ 

Let now $w\in \Gamma$ and consider the open segment $\Gamma(w)=\Gamma \cap B(w)$. We claim that the expansion (\ref{eq:Schatten_p_zero_one1}), restricted on $\Gamma(w)$, converges (absolutely) also in $(\mathcal{L}^p(H),|\cdot |_r)$. In other words, it is a local $r$-expansion, in the sense that, if $f_{n,w}:\Gamma(w)\to \mathbb C$ are given by $f_{n,w}(z)=(w-z)^n$, then 
\begin{equation}\label{eq:Schatten_p_zero_one3}
\sum_{n=0}^{\infty}|c_n(w)|_r^r \|f_{n,w}\|_{\infty}^r<\infty.
\end{equation}
Then, since $|\cdot|_r$ is equivalent to $\|\cdot \|_p$ which majorises the operator norm $\|\cdot \|$ on $\mathcal{B}(H)$, we get that this local $r$-expansion agrees with $F$ on $\Gamma(w)$. Of course, each $c_n(w)\in \mathcal{L}^p(H)$ and the functions $f_{n,w}$ are Riemann integrable. We note that for (\ref{eq:Schatten_p_zero_one3}) to be true, it suffices to show that 
\begin{equation}\label{eq:Schatten_p_zero_one4}
\sum_{n=0}^{\infty}\|c_n(w)\|_p^r \|f_{n,w}\|_{\infty}^r<\infty.
\end{equation}
We have that
\begin{equation}\label{eq:Schatten_p_zero_one5}
\|f_{n,w}\|_{\infty}=\sup_{z\in \Gamma(w)}|w-z|^n\leq \frac{1}{(2K)^n\|(w-S)^{-1}\|^n}.
\end{equation}
Moreover, 
\begin{align*}
\|c_n(w)\|_p&\leq K^n\sum_{j=0}^n \|s_j(w)[S,T]s_{n-j}(w)\|_p\\
&\leq \|[S,T]\|_pK^n\sum_{j=0}^n \|s_j(w)\|\|s_{n-j}(w)\|,
\end{align*}
and since $r\in (0,1]$, we also have that 
\begin{equation}\label{eq:Schatten_p_zero_one6}
\|c_n(w)\|_p^r\leq (\|[S,T]\|_pK^{n})^r\sum_{j=0}^n \|s_j(w)\|^r\|s_{n-j}(w)\|^r.
\end{equation}

In addition, the series $$\sum_{n=0}^{\infty}K^{rn}\|s_n(w)\|^r\|f_{n,w}\|^r_{\infty}<\infty$$ because from (\ref{eq:Schatten_p_zero_one5}) we obtain that
\newpage
\begin{align*}
K^{rn}\|s_n(w)\|^r\|f_{n,w}\|^r_{\infty}&=K^{rn}\|((w-S)^{-1})^{n+1}\|^r\|f_{n,w}\|^r_{\infty}\\
&\leq \|(w-S)^{-1}\|^r \frac{K^{rn} \|(w-S)^{-1}\|^{rn}}{2^{rn}K^{rn}\|(w-S)^{-1}\|^{rn}}\\
&= \|(w-S)^{-1}\|^r (\frac{1}{2^r})^n.
\end{align*}
Therefore, by noticing that $\|f_{n,w}\|_{\infty}^r=\|f_{1,w}\|_{\infty}^{rn}$, we have
$$
\sum_{n=0}^{\infty} (\sum_{j=0}^n \|s_j(w)\|^r\|s_{n-j}(w)\|^r) K^{rn}\|f_{n,w}\|_{\infty}^r = (\sum_{n=0}^{\infty}K^{rn}\|s_n(w)\|^r\|f_{n,w}\|^r_{\infty})^2 <\infty,
$$
and then from (\ref{eq:Schatten_p_zero_one6}) the series (\ref{eq:Schatten_p_zero_one4}) converges. 

From the compactness of $\Gamma$, there are $w_1,\ldots ,w_m\in \Gamma$ that determine open segments $\Gamma(w_1),\ldots , \Gamma(w_m)$ that cover $\Gamma$, on which $F$ has local $r$-expansions. Let $\{h_{w_1},\ldots , h_{w_m}\}$ be a partition of unity subordinated to the cover $\{\Gamma(w_1),\ldots ,\Gamma(w_m)\}$. Then, the function $F:\Gamma\to \mathcal{L}^p(H)$ can be written as $$F=\sum_{j=1}^{m} Fh_{w_j}.$$ Now, each $Fh_{w_j}$ has an $r$-expansion on $\Gamma$ since, if we let $\widetilde{f}_{n,w_j}$ denote the extension of $f_{n,w_j}:\Gamma(w_j)\to \mathbb C$ to $\Gamma$, then for every $z\in \Gamma$ one has $$(Fh_{w_j})(z)=\sum_{n=0}^{\infty} c_n(w_j)(\widetilde{f}_{n,w_j}h_{w_j})(z).$$ We still have that $\widetilde{f}_{n,w_j}h_{w_j}$ is Riemann integrable and 
$$
\sum_{n=0}^{\infty}|c_n(w_j)|_r^r \|\widetilde{f}_{n,w_j}h_{w_j}\|_{\infty}^r \leq \sum_{n=0}^{\infty}|c_n(w_j)|_r^r \|f_{n,w_j}\|_{\infty}^r \|h_{w_j}\|_{\infty}^r<\infty. 
$$
Consequently, the function $F:\Gamma\to (\mathcal{L}^p(H),|\cdot |_r)$ lies in $\mathscr{E}(\Gamma, \mathcal{L}^p(H))$ and hence it is Riemann integrable.
\end{proof}

Working as in \cite[Prop. 5a]{ConnesDifGeom} but applying Proposition \ref{prop:Schatten_p_zero_one} we obtain the following result.

\begin{lemma}\label{lem:Banach_alg_hol_cal_zero_one} Let $p\in (0,1)$. 
For every $S,T\in \mathcal{B}(H)$ such that $[S,T]\in \mathcal{L}^p(H)$ and every function $f$ that is holomorphic in a neighbourhood of the spectrum $\sigma(S)$, we have $$[f(S),T]\in \mathcal{L}^p(H).$$
\end{lemma}

\begin{lemma}\label{lem:extend_almost_commuting_algebras} Let $A,B$ be $C^*$-algebras in $\mathcal{B}(H)$ and $\mathcal{I}\subset \mathcal{B}(H)$ be a symmetrically normed ideal, or a Schatten $p$-ideal for $p\in (0,1)$. Also, let $\mathcal{A}\subset A,\, \mathcal{B}\subset B$ be dense $*$-subalgebras that commute modulo $\mathcal{I}$; $[a,b]\in \mathcal{I}$, for every $a\in \mathcal{A}$ and $b\in \mathcal{B}$. Then, there are dense $*$-subalgebras $\mathcal{A}_{\mathcal{B}}\subset A$ and $\mathcal{B}_{\mathcal{A}}\subset B$ that
\begin{enumerate}[(1)]
\item satisfy $\mathcal{A}\subset \mathcal{A}_{\mathcal{B}}$ and $\mathcal{B}\subset \mathcal{B}_{\mathcal{A}}$;
\item are holomorphically stable;
\item commute modulo $\mathcal{I}$.
\end{enumerate}
\end{lemma}

\begin{proof}
Let $\mathcal{A}_{\mathcal{B}}=\{a\in A: [a,b]\in \mathcal{I},\,\, \text{for all}\,\, b\in \mathcal{B}\}$. It is straightforward to see that $\mathcal{A}_{\mathcal{B}}$ is a $*$-subalgebra of $A$, and since it contains $\mathcal{A}$, it is also dense in $A$. Moreover, one can see that $$\widetilde{\mathcal{A}_{\mathcal{B}}}=\{a\in \widetilde{A}: [a,b]\in \mathcal{I},\,\, \text{for all}\,\, b\in \mathcal{B}\}.$$ Let $a\in \widetilde{\mathcal{A}_{\mathcal{B}}}$ and consider an arbitrary $b\in \mathcal{B}$. From Lemmas \ref{lem:Banach_alg_hol_cal} and \ref{lem:Banach_alg_hol_cal_zero_one} we see that, if $f$ is holomorphic in an open neighbourhood of the spectrum of $a$ in $\widetilde{A}$ (note that it agrees with the spectrum in $\mathcal{B}(H)$), then it still holds that $[f(a),b]\in \mathcal{I}$. Therefore, the algebra $\mathcal{A}_{\mathcal{B}}$ is holomorphically closed in $A$, and from Theorem \ref{thm:Density_theorem} we obtain that it is actually holomorphically stable in $A$.

Consider now $\mathcal{B}_{\mathcal{A}}=\{b\in B: [a,b]\in \mathcal{I},\,\, \text{for all}\,\, a\in \mathcal{A}_{\mathcal{B}}\}.$ By definition, $\mathcal{B}_{\mathcal{A}}$ contains $\mathcal{B}$ and therefore is a dense $*$-subalgebra of $B$. Working as before we obtain that $\mathcal{B}_{\mathcal{A}}$ is holomorphically stable in $B$.
\end{proof}

Since we work with corner $C^*$-subalgebras we need the following fact.

\begin{lemma}\label{lem:hereditary_hol_cal}
Let $A$ be a $C^*$-algebra and $\mathcal{A}\subset A$ be a holomorphically stable $*$-subalgebra. Then, for every projection $p\in \mathcal{A}$, the $*$-algebra $p\mathcal{A}p$ is holomorphically stable in $pAp$.
\end{lemma}

\begin{proof}
Let $p\in \mathcal{A}$ be a projection. We claim that $p\mathcal{A}p$ is holomorphically closed in $pAp$. Note that $pAp$ is unital with unit $p$, and hence we do not need to consider its unitisation. In order to avoid trivialities assume that $p\neq 0$. If $A$ is not unital, we consider its unitisation $\widetilde{A}$ as well as the unitisation $\widetilde{\mathcal{A}}\subset \widetilde{A}$. This makes sense because $p\widetilde{A}p=pAp$ (similarly for $\mathcal{A}$), and therefore, we can simply assume that $A$ and $\mathcal{A}$ are unital, with unit $1$. Again, to avoid trivialities assume that $p\neq 1$.

Let $b\in p\mathcal{A}p$ and denote by $\sigma_p(b)$ the spectrum relative to $pAp$, and by $\sigma(b)$ the spectrum relative to $A$. It holds that $$\sigma(b)=\sigma_p(b)\cup \{0\}.$$ Indeed, the element $b$ cannot be invertible in $A$ and hence $0\in \sigma(b)$. Also, for $z \in \mathbb C$, if $z 1-b$ is invertible in $A$ with inverse $(z1-b)^{-1}$, then $z p-b$ is invertible in $pAp$ with inverse 
\begin{equation}\label{eq:hereditary_hol_cal}
(z p-b)^{-1}=p(z 1-b)^{-1}p.
\end{equation}
Consequently, $\sigma_p(b)\subset \sigma(b)$. Finally, by representing $A$ and $pAp$ faithfully on a Hilbert space (note that $*$-isomorphisms preserve the spectrum of elements), it is not hard to check that $\sigma(b)\setminus \{0\}\subset \sigma_p(b)$.

Consider now a function $f$ which is holomorphic in an open set $U\supset \sigma_p(b)$. By a slight abuse of notation, let us denote by $f_p(b)$ the holomorphic functional calculus relative to $pAp$, and in case it can be defined, lets us denote by $f(b)$ the holomorphic functional calculus relative to $A$. First, assume that $0\in U$ and hence $\sigma(b)\subset U$. This means $f(b)$ is defined. Let $\gamma$ be a simple closed $C^1$-curve in $U$ that encloses counter-clockwise the spectrum $\sigma(b)$ and then, using (\ref{eq:hereditary_hol_cal}), we have
\begin{align*}
f_p(b)&=\frac{1}{2\pi i}\int_{\gamma}f(z)(zp-b)^{-1}dz\\
&=p(\frac{1}{2\pi i}\int_{\gamma}f(z)(z1-b)^{-1}dz)p\\
&=pf(b)p.
\end{align*}
As a result, $f_p(b)\in p\mathcal{A}p$ because $f(b)\in \mathcal{A}$.

Assume now that $0\not \in U$. Then, since $\{0\}$ and $\sigma_p(b)$ are disjoint, we can find an open ball $W\subset \mathbb C$ around $0$ and an open set $V$ around $\sigma_p(b)$ such that $W\cap V=\varnothing$. Consider the holomorphic function 
$$g(z)=\begin{cases}
f(z),& \text{if}\enspace z\in U\cap V\\
0, & \text{if} \enspace z\in W,
\end{cases}$$
and note that its domain contains $\sigma(b)$. Let $\gamma_1,\gamma_2$ be simple closed $C^1$-curves such that, $\gamma_1$ is in $U\cap V$ encircling $\sigma_p(b)$, and $\gamma_2$ is in $W$ around $0$, both with counter-clockwise orientation. Then, we have that 
\begin{align*}
g(b)&=\frac{1}{2\pi i}\int_{\gamma_1}g(z)(z1-b)^{-1}dz+\frac{1}{2\pi i}\int_{\gamma_2}g(z)(z1-b)^{-1}dz\\
&=\frac{1}{2\pi i}\int_{\gamma_1}g(z)(z1-b)^{-1}dz.
\end{align*}
Therefore, using (\ref{eq:hereditary_hol_cal}) we obtain that 
$$
pg(b)p=\frac{1}{2\pi i}\int_{\gamma_1}f(z)(zp-b)^{-1}dz=f_p(b).
$$
Since $g(b)\in \mathcal{A}$ it follows that $f_p(b)\in p\mathcal{A}p$. 

Finally, let $n\in \mathbb N$ and we claim that $\mathcal{M}_n(p\mathcal{A}p)$ is holomorphically closed in $\mathcal{M}_n(pAp)$. First, observe that $\mathcal{M}_n(pAp)=p'\mathcal{M}_n(A)p'$ and $\mathcal{M}_n(p\mathcal{A}p)=p'\mathcal{M}_n(\mathcal{A})p'$, where $p'$ is $\diag(p,\ldots,p)\in \mathcal{M}_n(\mathcal{A})$. Also, recall that $A$ and $\mathcal{A}$ can be assumed to be unital, and from the hypothesis we have that $\mathcal{M}_n(\mathcal{A})$ is holomorphically closed in $\mathcal{M}_n(A)$. Then, from the proof so far, the same holds for $\mathcal{M}_n(p\mathcal{A}p)\subset \mathcal{M}_n(pAp).$ 
\end{proof}

Let $\mathcal{I}\subset \mathcal{B}(H)$ be a symmetrically normed ideal or a Schatten $p$-ideal for $p\in (0,1)$. In what follows we consider a separable, simple, purely infinite $C^*$-algebra $A$ and a separable $C^*$-algebra $B$ such that $A\otimes B=A\otimes_{\max}B$. 

\begin{prop}\label{prop:Ext_smoothness_prop}
Let $\tau:A\otimes B\to \mathcal{Q}(H)$ be an invertible extension defined by $\tau(a\otimes b)= \rho_A(a)\rho_B(b)+\mathcal{K}(H),$ where $\rho_A:A\to \mathcal{B}(H),\, \rho_B:B\to \mathcal{B}(H)$ are faithful representations commuting modulo $\mathcal{K}(H)$. Assume that there are dense $*$-subalgebras $\mathcal{A}\subset A, \, \mathcal{B}\subset B$ such that $\rho_A(\mathcal{A}),\, \rho_B(\mathcal{B})$ commute modulo $\mathcal{I}$. Then,
\begin{enumerate}[(i)]
\item $\mathcal{A},\mathcal{B}$ can be respectively enlarged to holomorphically stable dense $*$-subalgebras $\mathcal{A}_{\mathcal{B}}\subset A,\, \mathcal{B}_{\mathcal{A}}\subset B$ such that $\rho_A(\mathcal{A}_{\mathcal{B}}),\, \rho_B(\mathcal{B}_{\mathcal{A}})$ commute modulo $\mathcal{I}$. As a result, $\tau$ is $\mathcal{I}$-smooth on $\mathcal{A}_{\mathcal{B}}\otimes_{\text{alg}}\mathcal{B}_{\mathcal{A}}$;
\item the $\Kt$-homology classes in the range of the map $\otimes_A[\tau]:\Kt_*(A)\to \Kt^{*+1}(B)$ are represented by Fredholm modules that are $\mathcal{I}$-summable on $\mathcal{B}_{\mathcal{A}}$. Therefore, if the map is surjective, then $B$ has uniformly $\mathcal{I}$-summable $\Kt$-homology on $\mathcal{B}_{\mathcal{A}}$.
\end{enumerate}
\end{prop}

\begin{proof}
By Lemma (\ref{lem:extend_almost_commuting_algebras}), $\rho_A(\mathcal{A})$ and $\rho_B(\mathcal{B})$ can be enlarged to holomorphically stable dense $*$-subalgebras $\rho_A(\mathcal{A})_{\mathcal{B}}\subset \rho_A(A)$ and $\rho_B(\mathcal{B})_{\mathcal{A}}\subset \rho_B(B)$ commuting modulo $\mathcal{I}$. Let $\mathcal{A}_{\mathcal{B}}\subset A$ and $\mathcal{B}_{\mathcal{A}}\subset B$ be the inverse images of $\rho_A(\mathcal{A})_{\mathcal{B}}$ and $\rho_B(\mathcal{B})_{\mathcal{A}}$. Then, the extension $\tau$ is $\mathcal{I}$-smooth on $\mathcal{A}_{\mathcal{B}}\otimes_{\text{alg}}\mathcal{B}_{\mathcal{A}}$. This proves part (i).

For part (ii) consider a non-zero projection $p\in \mathcal{A}_{\mathcal{B}}$ and an arbitrary $x\in \Kt_0(A)$. Let $\psi:pAp\to A$ be the inclusion yielding the isomorphism $\psi_*:\Kt_*(pAp)\to \Kt_*(A)$. Since $p\mathcal{A}_{\mathcal{B}}p$ is dense, holomorphically stable in $pAp$ (see Lemma \ref{lem:hereditary_hol_cal}), by the Density Theorem \ref{thm:Density_theorem} there is a non-zero projection $e\in p\mathcal{A}_{\mathcal{B}}p$ such that $x\otimes_A[\tau]=\psi_*([e])\otimes_A[\tau]$. By Proposition \ref{prop:slantprodKtheory}, the class $\psi_*([e])\otimes_A[\tau]$ is represented by the odd Fredholm module $(H,\rho_B,2\rho_A(e)-1)$ which is $\mathcal{I}$-summable on $\mathcal{B}_{\mathcal{A}}$. Similarly if $x\in \Kt_1(A)$.
\end{proof}

\section{Lipschitz algebras on general {\'e}tale groupoids}\label{sec:Lip_etale_alg} We show that on {\'e}tale groupoids equipped with a nice metric, one can construct dense Lipschitz $*$-subalgebras of the corresponding groupoid $C^*$-algebras. The main result of this section indicates the properties that the Smale space groupoid metrics should have, see Section \ref{sec:Metrise_Groupoids}. For details about {\'e}tale groupoids we refer to \cite{Renault_Book,Sims}.

Let $G$ be a second countable, locally compact, {\'e}tale groupoid whose topology is induced by a metric $d_G$, and let $G^{(0)}$ be the \textit{space of units}. For simplicity assume that $G^{(0)}\subset G$. By $r,s:G\to G^{(0)}$ we denote the \textit{range} and \textit{source maps} defined as
\begin{equation}
r(\gamma)=\gamma \gamma^{-1}, \enspace s(\gamma)=\gamma^{-1}\gamma,
\end{equation}
and by $G^{(2)}$ the \textit{set of composable pairs}
\begin{equation}
G^{(2)}=\{(\alpha,\beta)\in G\times G: s(\alpha)=r(\beta)\},
\end{equation}
which is closed in $G\times G$ since $G^{(0)}$ is Hausdorff. On $G^{(2)}$ we define the \textit{multiplication map} $m:G^{(2)}\to G$, denoted by $(\alpha,\beta)\mapsto\alpha \beta$, and for $A,B\subset G$ we define $AB\subset G$ as
\begin{equation}
AB=m((A\times B)\cap G^{(2)}).
\end{equation}
One can observe that, if $A,B\subset G$ are (pre-)compact, then $AB$ is also (pre-)compact. Indeed, since $A\times B$ is (pre-)compact in $G\times G$ and $G^{(2)}$ is closed, then $(A\times B)\cap G^{(2)}$ is (pre-)compact in $G^{(2)}$ and hence $AB$ is also (pre-)compact in $G$.

The groupoid being {\'e}tale means that $r,s$ are local homeomorphisms.
In fact, there is an abundance of sets, called bisections, on which $r,s$ are simultaneously local homeomorphisms. More precisely, a subset $B\subset G$ is a \textit{bisection} if there is an open $V\subset G$ containing $B$ such that, the restrictions $r_V:V\to r(V)$ and $s_V:V\to s(V)$ are homeomorphisms, with respect to the topology of $G$. This means that $G^{(0)}$ is open in $G$ and hence the $r,s$-fibres are discrete. 

Since $G$ is second countable and Hausdorff, its topology has a countable base $\mathcal{B}$ of open bisections. In addition, since $G$ is locally compact, the base can be chosen to consist of open, pre-compact bisections. Moreover, since $G$ is {\'e}tale we get that the multiplication map $m:G^{(2)}\to G$ is open. In particular, if $A,B\subset G$ are open, then $AB$ is also open. An important fact to notice is that, if $A,B$ are open, pre-compact bisections, then $AB$ is an open, pre-compact bisection.

Let us now consider the complex vector space $C_c(G)$ which becomes a $*$-algebra with \textit{convolution product} given by 
\begin{equation}\label{eq:convolution_etale}
(f\cdot g)(\gamma)=\sum_{\alpha \beta=\gamma}f(\alpha)g(\beta),
\end{equation}
and \textit{involution} given by
\begin{equation}
f^*(\gamma)=\overline{f(\gamma^{-1})}.
\end{equation}
Moreover, let $\pi$ be a faithful representation of $C_c(G)$ on a Hilbert space $H$ and $C^*_{\pi}(G)$ to denote the completion of $\pi(C_c(G))$ in $\mathcal{B}(H)$. Since $\pi$ is faithful we will often consider $C_c(G)\subset C^*_{\pi}(G)$ by suppressing the notation of $\pi$. The main result in this subsection is the following.

\begin{prop}\label{prop:Lip_dense_alg_general}
Suppose that $r,s:G\to G^{(0)}$ are locally bi-Lipschitz with respect to the metric $d_G$. Then, the complex vector space of compactly supported Lipschitz functions $\Lip_c(G,d_G)$ forms a dense $*$-subalgebra of $C_{\pi}^*(G)$.
\end{prop}

The proof of Proposition \ref{prop:Lip_dense_alg_general} is achieved by establishing all the following lemmas. The proof of the first lemma is omitted as it is elementary. Specifically, it follows from the existence of locally Lipschitz partitions of unity \cite[Theorem 5.3]{LV} and the Stone-Weierstrass Theorem for locally compact Hausdorff spaces, for instance see \cite[Theorem A.10.1]{DE}.

\begin{lemma}\label{lem:Stone_Weir}
Let $(Z,d_Z)$ be a locally compact metric space. Then, the $*$-subalgebra of compactly supported complex valued Lipschitz functions $\Lip_c(Z,d_Z)$ is dense in $C_0(Z)$, with respect to the sup-norm.
\end{lemma}

The proof of next result is the partition of unity argument of \cite[Lemma 3.1.3]{Sims}, but in the setting of Lipschitz functions.

\begin{lemma}\label{lem:Lip_span}
It holds that $$\Lip_c(G,d_G)=\sp \{f\in \Lip_c(G,d_G):\text{there is}\,\, B\in \mathcal{B}\,\,\text{such that}\,\, \supp(f)\subset B\}.$$
\end{lemma}

We also have the following.

\begin{lemma}\label{lem:Lip_dense_subspace}
It holds that $\Lip_c(G,d_G)$ is a dense subspace of $C_{\pi}^*(G)$.
\end{lemma}

\begin{proof}
It suffices to prove that, for $f\in C_c(G)$ and $\varepsilon >0$, there is $g\in \Lip_c(G,d_G)$ such that $\|\pi(f)-\pi(g)\|< \varepsilon.$ From \cite[Prop. 3.2.1]{Sims}, for every $h\in C_c(G)$, there is a constant $K_h\geq 0$ such that, for every representation $\rho$ of $C_c(G)$ on a Hilbert space, it holds that $\|\rho(h)\|\leq K_h.$ 
Using a partition of unity argument like the one needed in Lemma \ref{lem:Lip_span} (where the partition functions are only required to be continuous), we can write $h$ as a finite sum $\sum_{i}h_i,$ with each $h_i$ being compactly supported on a bisection in $\mathcal{B}$. Then, according to \cite{Sims} we can choose $$K_h=\sum_{i} \|h_i\|_{\infty}.$$ In particular, for the representation $\pi$, we get that 
\begin{equation}\label{eq:Lips_dense_subspace}
\|\pi(h)\|\leq \sum_{i} \|h_i\|_{\infty}.
\end{equation}

In our case, we can write $$f=\sum_{i=1}^{n}ff_i,$$ where $f_i$ are Lipschitz partition functions, compactly supported on bisections in $\mathcal{B}$. Let us now consider $M=\max_{1\leq i\leq n}\|f_i\|_{\infty}$ and then from Lemma \ref{lem:Stone_Weir} there exists $\widetilde{g}\in \Lip_c(G,d_G)$ so that $\|f-\widetilde{g}\|_{\infty}<\varepsilon (nM)^{-1}.$ Let $$g=\sum_{i=1}^{n}\widetilde{g}f_i,$$ and note that $g\in \Lip_c(G,d_G)$ and that each $(f-\widetilde{g})f_i$ is compactly supported on a bisection in $\mathcal{B}$. Then, from (\ref{eq:Lips_dense_subspace}) we get that $\|\pi(f)-\pi(g)\|<\varepsilon$ and the proof is complete.
\end{proof}

\begin{lemma}\label{lem:convolution_Lip}
Supppose that the maps $r,s$ are locally bi-Lipschitz with respect to $d_G$. Then, for every $f,g\in \Lip_c(G,d_G)$, the convolution $f\cdot g\in \Lip_c(G,d_G).$ Consequently, $\Lip_c(G,d_G)$ is a $*$-subalgebra of $C_{\pi}^*(G)$.
\end{lemma}

\begin{proof}
Let $f,g\in \Lip_c(G,d_G)$ and by Lemma \ref{lem:Lip_span} we can assume that $\supp(f)\subset A$, $\supp(g)\subset B$, where $A,B\in \mathcal{B}$. First we note that, for all $\gamma=\alpha \beta \in AB$ it holds that, $r_{AB}(\gamma)=r_{A}(\alpha)$ and $s_{AB}(\gamma)=s_{B}(\beta)$ and as a result, 
\begin{equation}\label{eq:convolution_Lip}
\alpha=r_A^{-1}(r_{AB}(\gamma)),\enspace \beta= s_B^{-1}(s_{AB}(\gamma)).
\end{equation}
From the definition of the convolution (\ref{eq:convolution_etale}) we see that, if $(f\cdot g)(\gamma)\neq 0$, then there is a pair $(\alpha ,\beta) \in G^{(2)}$ such that $\gamma=\alpha \beta$ and $f(\alpha)g(\beta)\neq 0$. Therefore, $\alpha \in A,\, \beta \in B$ and hence $\alpha, \beta$ are uniquely determined by $\gamma$ as in (\ref{eq:convolution_Lip}). 

So, if $\gamma\in AB$, then 
\begin{equation}\label{eq:convolution_Lip2}
(f\cdot g)(\gamma)=(f\circ r_A^{-1}\circ r_{AB})(\gamma)(g\circ s_B^{-1}\circ s_{AB})(\gamma),
\end{equation}
and if $\gamma \not \in AB$, then $(f\cdot g)(\gamma)=0$. Moreover, $\supp(f)\supp(g)$ is a compact subset of $AB$ and since $$\supp(f\cdot g)\subset \supp(f)\supp(g),$$ we obtain that $\supp(f\cdot g)$ is also compact. 

Assume now that $f\cdot g\neq 0$ and hence $AB\neq \varnothing$. Then, since $r,s$ are locally bi-Lipschitz and $A,B,AB$ are open, the restricted maps $r_A,s_B,r_{AB},s_{AB}$ are also locally bi-Lipschitz with respect to $d_G$. In particular, the inverses $r_A^{-1},s_B^{-1}$ are locally bi-Lipschitz. Also, since $\supp(f\cdot g)$ compact in the open $AB$, we can find an $\varepsilon$-neighbourhood $N_{\varepsilon}$ of $\supp(f\cdot g)$ such that $\cl(N_{\varepsilon})\subset AB$. From the discussion so far we note that $r_A^{-1}\circ r_{AB}$ and $s_B^{-1}\circ s_{AB}$ are locally bi-Lipschitz homeomorphisms. Then, from Lemma \ref{lem:locally_bi_Lip} we obtain that $r_A^{-1}\circ r_{AB}$ and $s_B^{-1}\circ s_{AB}$ are bi-Lipschitz on the compact set $\cl(N_{\varepsilon})$, and since $f,g$ are Lipschitz with compact support we get that $f\cdot g$ is Lipschitz on $\cl(N_{\varepsilon})$. Since $f\cdot g$ vanishes outside of $\cl(N_{\varepsilon})$ it then follows that $f\cdot g$ is globally Lipschitz.
\end{proof}

\section{Metrisation of Smale space groupoids}\label{sec:Metrise_Groupoids}
We construct dynamical metrics on the stable and unstable groupoids of Smale spaces. These are compatible with the {\'e}tale topologies of the groupoids, making the range, source maps locally bi-Lipschitz and the automorphisms (induced from the Smale space homeomorphisms) bi-Lipschitz. In the case of TMC we obtain groupoid ultrametrics. We note that similar metrics can be constructed for the homoclinic groupoid and the semi-direct product groupoids found in \cite{PS}. However, what follows is notationally demanding, hence we avoid performing similar computations for the latter groupoids, whose metric structure investigation is not needed here. This section is an important step in studying summability for Ruelle algebras, since the groupoid metrics are designed to yield Lipschitz algebras on which the $\Kt$-homology of Ruelle algebras is uniformly finitely summable, see Section \ref{sec:SmoothRuelle}.

Our metrisation tool is the Alexandroff-Urysohn-Frink Theorem \ref{thm:Metrisation_thm}. Let $G$ denote the stable or unstable groupoid. The main utility of Theorem \ref{thm:Metrisation_thm} lies in building a tractable uniform structure (see \cite[Chapter 6]{Kelley}) on $G$ by considering countably many neighbourhoods of the diagonal of $G$ in $G \times G$. The corresponding uniform structure should also generate the {\'e}tale (Hausdorff) topology on $G$, thus giving a compatible metric on $G$. 

Let $(X,d,\varphi)$ be an irreducible Smale space and fix two periodic orbits $P,Q$. For simplicity we state and prove our basic results only for the stable groupoid $G^s(Q)$, as the statements, proofs and notation for the unstable groupoid $G^u(P)$ are similar. We begin by noting that $G^s(Q)$ is an invariant of topological conjugacy for $(X,d,\varphi)$, see \cite{Putnam_Func}. 

To make this specific to our case, let $\text{M}_d(X)$ be the collection of all metrics $d'$ on $X$ inducing the same topology with $d$. Also, consider the collection of \textit{hyperbolic metrics}
\begin{equation}
\text{hM}_d(X,\varphi)=\{d'\in \text{M}_d(X): (X,d',\varphi)\, \, \text{is a Smale space}\},
\end{equation}
and its non-empty sub-collection of \textit{self-similar (hyperbolic) metrics}
\begin{equation}
\text{sM}_d(X,\varphi)=\{d'\in \text{M}_d(X): (X,d',\varphi)\, \, \text{is a self-similar Smale space}\}.
\end{equation}
The algebraic structure of $G^s(Q)$ can equivalently be given by any metric in $\text{M}_d(X)$. Moreover, for every metric $d'\in \text{hM}_d(X,\varphi)$, Theorem \ref{thm: stable bisections} can be applied to the Smale space $(X,d',\varphi)$ and yield a topological base $\mathcal{B}^s(Q,d')$ for $G^s(Q)$. The gist is that all these bases produce the same topology on $G^s(Q)$. Thus, one can freely choose suitable hyperbolic metrics to work with and the self-similar are the best behaved. In the sequel, the construction of $G^s(Q)$ from the Smale space $(X,d',\varphi)$, where $d'\in \text{hM}_d(X,\varphi)$, will be called the $d'$\textit{-model} of $G^s(Q)$.

\subsection{Metrics and Lipschitz dynamics on Smale space groupoids}\label{sec:Lipschitz_metrics_on_groupoids}
We use the powerful metrisation tool of Alexandroff, Urysohn and Frink \cite{Frink} to metrise $G^s(Q)$. More precisely, for every $d'\in \text{hM}_d(X,\varphi)$, using the {\'e}tale topological base $\mathcal{B}^s(Q,d')$ one can build a compatible metric $D_{s,d'}$ for $G^s(Q)$. In particular, if $d'\in \text{sM}_d(X,\varphi)$, we obtain the following result (Theorem \ref{thm:GroupoidMetrisation}), whose proof is achieved as a combination of Lemmas \ref{lem:generalgroupoidmetric}, \ref{lem:generalgroupoidmetricLip}, \ref{lem:bi_Lip_inversion} and \ref{lem:bi-Lipschitz_groupoid_maps}. Before stating the theorem, we should highlight a few basic facts and introduce some notation. 

By $i:G^s(Q)\to G^s(Q)$ we denote the inversion map
\begin{equation}
i(x,y)=(y,x),
\end{equation}
and recall that by $r,s:G^s(Q)\to X^u(Q)$ we denote the range and source maps, defined as
\begin{equation}
r(x,y)=x, \enspace s(x,y)=y.
\end{equation}
In the sequel we will consider compatible metrics $D,\, \widetilde{D}$ on $G^s(Q),\, X^u(Q)$, respectively (see (\ref{eq:groupoidmetric}), (\ref{eq:restricted_metric}) and Lemma \ref{lem:bi-Lipschitz_groupoid_maps}), for which $r$ becomes locally bi-Lipschitz, that is, for every $a\in G^s(Q)$ there is $\ell_a >0$ and $\Lambda_a \geq 1$ such that, 
\begin{equation}\label{eq:locally_bi_Lip}
\Lambda_a^{-1}D(a,b)\leq \widetilde{D}(r(a),r(b))\leq \Lambda_a D(a,b),
\end{equation}
for every $b\in B_{D}(a,\ell_a)$. Similarly, for the map $s$. This should be no surprise, because for every $V\in \mathcal{B}(G^s(Q))$, where 
\begin{equation}
\mathcal{B}(G^s(Q))=\{W\subset G^s(Q): W\in \mathcal{B}^s(Q,d'),\, d'\in \text{hM}_{d'}(X,\varphi)\},
\end{equation}
it is not hard to show that the restrictions $r_V:V\to r(V)$ and $s_V:V\to s(V)$ are homeomorphisms, and therefore the maps $r,s$ are local homeomorphisms (note that $j(X^u(Q))$ is open is $G^s(Q)$ since $G^s(Q)$ is {\'e}tale). 

\begin{thm}\label{thm:GroupoidMetrisation}
For every metric $d'\in \text{sM}_d(X,\varphi)$ there are compatible metrics $D_{s,d'}$ on $G^s(Q)$ and $\widetilde{D_{s,d'}}$ on $X^u(Q)$ so that, with respect to these metrics,
\begin{enumerate}[(1)]
\item the groupoid automorphism $\Phi=\varphi \times \varphi :G^s(Q)\to G^s(Q)$ is bi-Lipschitz with $$4^{-1}D_{s,d'}(a,b)\leq D_{s,d'}(\Phi(a),\Phi(b))\leq 8D_{s,d'}(a,b),$$ for every $a,b\in G^s(Q)$;
\item the inversion map $i$ is bi-Lipschitz, and for every base set $V\in \mathcal{B}(G^s(Q))$, the restrictions $r_V$ and $s_V$ of the range and source maps are bi-Lipschitz. Specifically, the maps $r,s$ are locally bi-Lipschitz.
\end{enumerate}
Similarly, for every $d'\in \text{sM}_d(X,\varphi)$, there are compatible metrics $D_{u,d'}$ on $G^u(P)$ and $\widetilde{D_{u,d'}}$ on $X^s(P)$ for which, the automorphism $\Phi=\varphi \times \varphi :G^u(P)\to G^u(P)$ satisfies $$4^{-1}D_{u,d'}(a,b)\leq D_{u,d'}(\Phi^{-1}(a),\Phi^{-1}(b))\leq 8D_{u,d'}(a,b),$$ for every $a,b\in G^u(P)$, the inversion map is bi-Lipschitz, and the restrictions of the range and source maps on base sets are also bi-Lipschitz.
\end{thm}

\begin{remark}
Recall that for every $d_1,d_2\in \text{hM}_d(X,\varphi)$ the $d_1$-model of $G^s(Q)$ is exactly the same with the $d_2$-model of $G^s(Q)$. They are just different ways to construct $G^s(Q)$. However, for $d'\in \text{hM}_d(X,\varphi)$, the metric $D_{s,d'}$ is better adapted to the $d'$-model.
\end{remark}

One can find several formulations of the Alexandroff-Urysohn-Frink Metrisation Theorem \cite[Chapter 6]{Kelley}. Here it is best to use the one in \cite[Theorem 2.4.1]{Sakai_book}. Before stating it, if $Z$ is a set, $\mathcal{U}$ is a cover of $Z$ and $z\in Z$, let us denote the $\mathcal{U}$-\textit{star} of $z$ by
\begin{equation}
\st(z,\mathcal{U})=\bigcup \{U\in \mathcal{U}: z\in U\}.
\end{equation}

\begin{thm}[Alexandroff-Urysohn-Frink Metrisation Theorem]\label{thm:Metrisation_thm}
For a Hausdorff space $Z$ the following are equivalent.
\begin{enumerate}[(1)]
\item $Z$ is metrisable.
\item $Z$ has a sequence of open covers $(\mathcal{U}_n)_{n\in \mathbb N}$ such that, 
\begin{enumerate}[(a)]
\item for each $z\in Z$, the sequence $(\st(z,\mathcal{U}_n))_{n\in \mathbb N}$ is a neighbourhood base of $z$;
\item for every $U,U'\in \mathcal{U}_{n+1}$, if $U\cap U'\neq \varnothing$ then there is some $U''\in \mathcal{U}_n$ such that $U\cup U'\subset U''.$
\end{enumerate}
\item Each $z\in Z$ has a neighbourhood base $(V_n(z))_{n\in \mathbb N}$ of open sets satisfying the condition that, for every $z\in Z$ and $n \in \mathbb N$, there is some $j(z,n)\geq n$ so that, if $y\in Z$ with $V_{j(z,n)}(z)\cap V_{j(z,n)}(y)\neq \varnothing$, then $V_{j(z,n)}(y)\subset V_n(z)$.
\end{enumerate}
\end{thm}

We aim to construct metrics on $G^s(Q)$ by building neighbourhood bases that satisfy condition (3) of Theorem \ref{thm:Metrisation_thm}. To simplify the notation we will denote the elements of $G^s(Q)$ by $a=(a_1,a_2),\,b=(b_1,b_2),\, c=(c_1,c_2)$ or $e=(e_1,e_2)$. 

Let $d'\in \text{hM}_d(X,\varphi)$ be an arbitrary metric and consider the $d'$-model of $G^s(Q)$. This means all base sets are in $\mathcal{B}^s(Q,d')$ and all local stable, unstable sets, and the homeomorphism $\varphi$, are defined and behave with respect to $d'$. However, in order to keep the notation reasonably simple we do not highlight the dependence on the metric $d'$. In particular, the Smale space constants will be again denoted by $\varepsilon_X>0$ and $\lambda_X>1$, instead of $\varepsilon_{X,d'}$ and $\lambda_{X,d'}$.

First of all, since $\varphi$ is uniformly continuous on $X$, for every $n\geq 0$, we can consider $\eta_n\in (0,\varepsilon_X/2]$ to be the supremum amongst all $\eta \in (0,\varepsilon_X/2]$ such that $\varphi^n(X^u(z,2\eta))\subset \cl(X^u(\varphi^n(z),\varepsilon_X'/2))$, for every $z\in X$, where the closure is taken in $X^u(\varphi^n(z),\varepsilon_X)$. In particular, 

\begin{equation}\label{eq:widthofholonomies}
\varphi^n(X^u(z,2\eta_n))\subset \cl(X^u(\varphi^n(z),\varepsilon_X'/2))
\end{equation}
for every $z\in X$. The sequence $(\eta_n)_{n\geq 0}$ is decreasing since, for every $z\in X$, we have $\varphi^{n+1}(X^u(z,2\eta_{n+1}))\subset \cl(X^u(\varphi^{n+1}(z),\varepsilon_X'/2))$ and hence $$\varphi^{n}(X^u(z,2\eta_{n+1}))\subset \cl(X^u(\varphi^{n}(z),\lambda_X^{-1}\varepsilon_X'/2)),$$ meaning that $\eta_{n+1}\leq \eta_n$. Also, one can observe that $(\eta_n)_{n\geq 0}$ converges to zero.
\newpage
Recall that the base sets of the topology on $G^s(Q)$ are denoted by $V^s(a,h^s,\eta, N)$, where 
\begin{enumerate}[(i)]
\item $N\in \mathbb N$ is big enough so that $$\varphi^N(a_2)\in X^s(\varphi^N(a_1),\varepsilon_X'/2);$$
\item $\eta\in (0,\varepsilon_X/2]$ is small enough so that $$\varphi^{N}(X^u(a_2,\eta))\subset X^u(\varphi^{N}(a_2),\varepsilon_X'/2);$$
\item $h^s:X^u(a_2,\eta)\to X^u(a_1, \varepsilon_X/2)$ is given by $$h^s(z)=\varphi^{-N}[\varphi^N(z),\varphi^N(a_1)].$$ 
\end{enumerate}

For every $a\in G^s(Q)$ we define $N_a\geq 0$ to be the \textit{first time} that 
\begin{equation}\label{eq:first_time}
\varphi^{N_a}(a_2)\in X^s(\varphi^{N_a}(a_1),\varepsilon_X'/2),
\end{equation}
and consider the sequences $(N_{a,n})_{n\geq 0}$ and $(\eta_{a,n})_{n\geq 0}$ given by 
\begin{equation}
N_{a,n}= \max \{ N_a, n\} \enspace \text{and} \enspace \eta_{a,n}=\eta_{N_{a,n}}
\end{equation}
Then, define the holonomy map $h_{a,n}^s:X^u(a_2, \eta_{a,n})\to X^u(a_1,\varepsilon_X/2)$ given by 
\begin{equation}
h_{a,n}^s(z)=\varphi^{-N_{a,n}}[\varphi^{N_{a,n}}(z),\varphi^{N_{a,n}}(a_1)],
\end{equation}
and consider the base set 
\begin{equation}\label{eq:neigh_base_sets}
V_n(a)= V^s(a, h_{a,n}^s, \eta_{a,n}, N_{a,n}).
\end{equation}

The above condition (ii) can be relaxed by considering the closure on the right-hand side of the inclusion. Then, due to (\ref{eq:widthofholonomies}), each holonomy map $h_{a,n}^s$ can be extended on $X^u(a_2, 2\eta_{a,n})$.

\begin{lemma}\label{lem:groupoidNB}
For $a\in G^s(Q)$, the sequence of open sets $(V_n(a))_{n\geq 0}$ forms a decreasing neighbourhood base of $a$.
\end{lemma}

\begin{proof}
Let $a\in G^s(Q)$, we will show that if $b\in V_{n+1}(a)$, then $b\in V_n(a)$. Since $N_{a,n+1}\geq N_{a,n}$ it holds that $\eta_{a,n+1}\leq \eta_{a,n}$, hence $b_2\in X^u(a_2, \eta_{a,n+1})\subset X^u(a_2, \eta_{a,n})$, meaning that $b_2$ lies in the domain of the holonomy map $h^s_{a,n}$. The claim follows because $h_{a,n}^s(b_2)=b_1$, where $b_1=h_{a,n+1}^s(b_2)=\varphi^{-N_{a,n+1}}[\varphi^{N_{a,n+1}}(b_2),\varphi^{N_{a,n+1}}(a_1)].$ 

Let $V^s(c,h^s,\eta,N)$ be a base set containing $a\in G^s(Q)$. Choose $n$ big enough such that $N_{a,n}\geq N$ and $\eta_{a,n}$ is even smaller than $\eta$, so that $X^u(a_2, \eta_{a,n})\subset X^u(c_2,\eta)$. Then, it holds that $V_n(a)\subset V^s(c,h^s,\eta,N)$. Indeed, let $b\in V_n(a)$ and we have that $b_2\in X^u(a_2,\eta_{a,n})$, meaning that $b_2$ lies in the domain of $h^s$. It remains to show that $h^s(b_2)=b_1$, where $b_1=h_{a,n}^s(b_2)=\varphi^{-N_{a,n}}[\varphi^{N_{a,n}}(b_2),\varphi^{N_{a,n}}(a_1)].$ For this we use the fact that
\begin{enumerate}[(i)]
\item $\varphi^N(a_1)\in X^u(\varphi^N(c_1),\varepsilon_X/2)$;
\item $d'(\varphi^{N+k}(a_1),\varphi^{N+k}(b_2))<\varepsilon_X$, for all $k\in \{0,\ldots , N_{a,n}-N\}$.
\end{enumerate}

For part (i), observe that $\varphi^N(a_1)=[\varphi^N(a_2),\varphi^N(c_1)]$. Since $d'(\varphi^N(a_2),\varphi^N(c_2))<\varepsilon_X'/2$ and $d'(\varphi^N(c_2),\varphi^N(c_1))<\varepsilon_X'/2$, we have that $$[\varphi^N(a_2),\varphi^N(c_1)]\in X^u(\varphi^N(c_1),\varepsilon_X/2).$$ For part (ii), note that $b_2\in X^u(a_2, \eta_{a,n})$ and hence $\varphi^{N_{a,n}}(b_2)\in X^u(\varphi^{N_{a,n}}(a_2),\varepsilon_X'/2).$ Therefore, $\varphi^{N+k}(b_2)\in X^u(\varphi^{N+k}(a_2),\varepsilon_X'/2),$ for all $k\in \{0,\ldots , N_{a,n}-N\}$. In addition, $\varphi^N(a_1)\in X^s(\varphi^N(a_2),\varepsilon_X/2),$ and as a result $\varphi^{N+k}(a_1)\in X^s(\varphi^{N+k}(a_2),\varepsilon_X/2),$ for every $k\in \{0,\ldots , N_{a,n}-N\}$. 

With conditions (i) and (ii) the following computations with the bracket map are well-defined, that is,
\begin{align*}
h^s(b_2)&=\varphi^{-N}[\varphi^N(b_2),\varphi^N(c_1)]\\
&= \varphi^{-N}[\varphi^N(b_2),[\varphi^N(c_1),\varphi^N(a_1)]]\\
&= \varphi^{-N}[\varphi^N(b_2),\varphi^N(a_1)]\\
&= \varphi^{-N_{a,n}}[\varphi^{N_{a,n}}(b_2),\varphi^{N_{a,n}}(a_1)]\\
&= b_1. \qedhere
\end{align*}
\end{proof}

We now mimic the notation of part (3) of Theorem \ref{thm:Metrisation_thm}. For every $a\in G^s(Q)$ and $n\geq 0$ define 
\begin{equation}\label{eq:J(a,n)}
j(a,n)=N_{a,n}+\lceil \log_{\lambda_X}3 \rceil,
\end{equation}
and observe that $j(a,n)\geq \max \{ N_a, n\} +1$. 

\begin{lemma}\label{lem:groupoidNB2}
For every $a\in G^s(Q), n\geq 0$, if $b\in G^s(Q)$ with $V_{j(a,n)}(a)\cap V_{j(a,n)}(b)\neq \varnothing$, then $V_{j(a,n)}(b)\subset V_n(a)$.
\end{lemma}

\begin{proof}
Let $n\geq 0$ and $a,b\in G^s(Q)$ such that $V_{j(a,n)}(a)\cap V_{j(a,n)}(b)\neq \varnothing$. We have 
\begin{align*}
V_{j(a,n)}(a)&= V^s(a, h_{a,j(a,n)}^s, \eta_{a,j(a,n)}, N_{a,j(a,n)})\\
V_{j(a,n)}(b)&= V^s(b, h_{b,j(a,n)}^s, \eta_{b,j(a,n)}, N_{b,j(a,n)}).
\end{align*}
Let $c\in V_{j(a,n)}(a)\cap V_{j(a,n)}(b)$ and $e\in V_{j(a,n)}(b)$. We aim to show that $$e\in V_n(a)=V^s(a, h_{a,n}^s, \eta_{a,n}, N_{a,n}).$$

We have that 
\begin{equation}\label{eq:groupoidNB2_eq}
e_2\in X^u(b_2, \eta_{b,j(a,n)}), \, c_2\in X^u(b_2, \eta_{b,j(a,n)}), \, c_2\in X^u(a_2, \eta_{a,j(a,n)}). 
\end{equation}
In any case, it holds that 
\begin{equation}\label{eq:groupoidNB2_eq2}
\eta_{b,j(a,n)}\leq \eta_{a,j(a,n)}.
\end{equation}
To check that (\ref{eq:groupoidNB2_eq2}) is true, first recall the decreasing sequence $(\eta_n)_{n\geq 0}$ from (\ref{eq:widthofholonomies}). Then, note that $N_{a,j(a,n)}=j(a,n)$ because $j(a,n)>N_a$, hence $\eta_{a,j(a,n)}=\eta_{N_{a,j(a,n)}}
= \eta_{j(a,n)}$, where $\eta_{j(a,n)}\geq \eta_{N_{b,j(a,n)}}= \eta_{b,j(a,n)}$.

From (\ref{eq:groupoidNB2_eq}) and (\ref{eq:groupoidNB2_eq2}) we obtain that $e_2\in X^u(a_2,3 \eta_{a,j(a,n)})$. We aim to show that $3 \eta_{a,j(a,n)}\leq \eta_{a,n}$, hence $e_2\in X^u(a_2,\eta_{a,n})$, meaning that $e_2$ lies in the domain of $h_{a,n}^s$. Note that $\lambda_X^{-\lceil \log_{\lambda_X}3 \rceil} \leq 3^{-1}$ and by using (\ref{eq:widthofholonomies}) for $\eta_{a,j(a,n)}=\eta_{j(a,n)}$, for every $z\in X$,
\begin{align*}
\varphi^{N_{a,n}}(X^u(z, 6 \eta_{a,j(a,n)})) &= \varphi^{j(a,n)-\lceil \log_{\lambda_X}3 \rceil}(X^u(z, 6 \eta_{a,j(a,n)}))\\
&\subset \varphi^{j(a,n)}(X^u(\varphi^{-\lceil \log_{\lambda_X}3 \rceil}(z), \lambda_X^{-\lceil \log_{\lambda_X}3 \rceil}6 \eta_{a,j(a,n)}))\\
&\subset \varphi^{j(a,n)}(X^u(\varphi^{-\lceil \log_{\lambda_X}3 \rceil}(z), 2 \eta_{a,j(a,n)}))\\
&= \varphi^{j(a,n)}(X^u(\varphi^{-\lceil \log_{\lambda_X}3 \rceil}(z), 2 \eta_{j(a,n)}))\\
&\subset \cl (X^u (\varphi^{N_{a,n}}(z), \varepsilon_X'/2)).
\end{align*}
From the choice of $\eta_{a,n}$ we obtain that $3 \eta_{a,j(a,n)}\leq \eta_{a,n}$. It remains to show that $h_{a,n}^s(e_2)=e_1$, where 
$$e_1=h_{b,j(a,n)}^s(e_2)=
\begin{cases}
\varphi^{-N_b}[\varphi^{N_b}(e_2),\varphi^{N_b}(b_1)], &\text{if } j(a,n)\leq N_b\\
\varphi^{-j(a,n)}[\varphi^{j(a,n)}(e_2),\varphi^{j(a,n)}(b_1)], &\text{if } j(a,n)> N_b.
\end{cases}
$$
Assume that $j(a,n)\leq N_b$, and thus $N_{a,n}< N_b$. We will prove and use the following three facts:
\begin{enumerate}[(i)]
\item $\varphi^{N_{a,n}}(c_1)\in X^u(\varphi^{N_{a,n}}(a_1),\varepsilon_X/2)$;
\item $d'(\varphi^{N_{a,n}+k}(e_2), \varphi^{N_{a,n}+k}(c_1))< \varepsilon_X$, for all $k\in \{0,\ldots, N_b-N_{a,n}\}$;
\item $\varphi^{N_b}(c_1)\in X^u(\varphi^{N_b}(b_1),\varepsilon_X/2)$.
\end{enumerate}

Part (i) follows since $c\in V_{j(a,n)}(a)$, meaning that $\varphi^{j(a,n)}(c_1)\in X^u(\varphi^{j(a,n)}(a_1),\varepsilon_X/2)$, and hence $\varphi^{N_{a,n}}(c_1)\in X^u(\varphi^{N_{a,n}}(a_1),\varepsilon_X/2)$. For part (ii), observe that $c\in V_n(a)$. Then, we have $\varphi^{N_{a,n}}(c_1)\in X^s(\varphi^{N_{a,n}}(c_2), \varepsilon_X/2)$ and thus for all $k\in \{0,\ldots, N_b-N_{a,n}\}$, $$\varphi^{N_{a,n}+k}(c_1)\in X^s(\varphi^{N_{a,n}+k}(c_2), \varepsilon_X/2).$$ Also, since $e_2\in X^u(c_2, 2\eta_{b,j(a,n)})$ we have $\varphi^{N_b}(e_2)\in \cl(X^u(\varphi^{N_b}(c_2), \varepsilon_X'/2))$, hence $$\varphi^{N_{a,n}+k}(e_2)\in \cl(X^u(\varphi^{N_{a,n}+k}(c_2), \varepsilon_X'/2)),$$ for all $k\in \{0,\ldots, N_b-N_{a,n}\}$. The claim for (ii) follows, and part (iii) is clear. 

With conditions (i), (ii) and (iii) the following computations are well-defined, 
\begin{align*}
h_{a,n}^s(e_2)&=\varphi^{-N_{a,n}}[\varphi^{N_{a,n}}(e_2), \varphi^{N_{a,n}}(a_1)]\\
&= \varphi^{-N_{a,n}}[\varphi^{N_{a,n}}(e_2),[\varphi^{N_{a,n}}(a_1), \varphi^{N_{a,n}}(c_1)]]\\
&= \varphi^{-N_{a,n}}[\varphi^{N_{a,n}}(e_2), \varphi^{N_{a,n}}(c_1)]\\
&= \varphi^{-N_b}[\varphi^{N_b}(e_2), \varphi^{N_b}(c_1)]\\
&= \varphi^{-N_b}[\varphi^{N_b}(e_2), [\varphi^{N_b}(c_1), \varphi^{N_b}(b_1)]]\\
&= \varphi^{-N_b}[\varphi^{N_b}(e_2), \varphi^{N_b}(b_1)].
\end{align*}

The exactly same computations work for the case $j(a,n)>N_b$, where in the place of $N_b$ one has to put $j(a,n)$. In any case, we obtain that $e\in V_n(a)$ and this completes the proof of the lemma.
\end{proof}

Lemmas \ref{lem:groupoidNB} and \ref{lem:groupoidNB2} imply that $G^s(Q)$ satisfies condition (3) of Theorem \ref{thm:Metrisation_thm}. Now, following the proof of Theorem \ref{thm:Metrisation_thm} in \cite{Sakai_book}, we aim to construct a sequence of open covers $(\mathcal{U}^s_n)_{n\in \mathbb N}$ satisfying condition (2). First of all, in order to be consistent with Theorem \ref{thm:Metrisation_thm}, for every $a\in G^s(Q)$ we consider the neighbourhood base $(V_n(a))_{n\in \mathbb N}$ instead of $(V_n(a))_{n\geq 0}$. 
\enlargethispage{\baselineskip}

For every $a\in G^s(Q)$, set $k(a,1)=1,$ and inductively for $n\geq 2$ define 
\begin{equation}
k(a,n)=\max \{n, j(a,i): i=1,\ldots, k(a,n-1)\}.
\end{equation}
Then, for every $n\in \mathbb N$, let 
\begin{equation}
U_n(a)=\bigcap_{i=1}^{k(a,n)}V_i(a),
\end{equation}
and since $(V_n(a))_{n\in \mathbb N}$ is decreasing we have that $U_n(a)=V_{k(a,n)}(a).$
Then, the covers 
\begin{equation}
\mathcal{U}^s_n=\{U_n(a):a\in G^s(Q)\}
\end{equation}
form the desired sequence. The next lemma gives a precise definition of $\mathcal{U}^s_n$.

\begin{lemma}\label{lem:k(a,n)}
For every $a\in G^s(Q)$ and $n\in \mathbb N$ we have that $$k(a,n+1)=N_{a,1}+n \lceil \log_{\lambda_X}3 \rceil.$$  
\end{lemma}

\begin{proof}
By definition we have that $k(a,n+1)=\max \{n+1, j(a,i): i=1,\ldots, k(a,n)\}$. The sequence $(j(a,m))_{m\in \mathbb N}$ is increasing, hence $
k(a,n+1)=\max \{n+1, j(a, k(a,n))\}$. In fact, it holds that 
\begin{equation}\label{eq:lem_k(a,n)}
k(a,n+1)=j(a,k(a,n))
\end{equation}
because $j(a,k(a,n))\geq j(a,n)=N_{a,n}+\lceil \log_{\lambda_X}3 \rceil \geq n+1.$ To prove the lemma we will use (\ref{eq:lem_k(a,n)}) and induction on $n$. For $n=1$ we have $$k(a,2)=j(a,k(a,1))=j(a,1)=N_{a,1}+ \lceil \log_{\lambda_X}3 \rceil.$$ Assuming that $k(a,n+1)=N_{a,1}+n \lceil \log_{\lambda_X}3 \rceil$, one has that 
\begin{align*}
k(a,n+2)&= j(a,k(a,n+1))\\
&= N_{a,k(a,n+1)} + \lceil \log_{\lambda_X}3 \rceil \\
&= k(a,n+1)+ \lceil \log_{\lambda_X}3 \rceil \\ 
&= N_{a,1}+ (n+1) \lceil \log_{\lambda_X}3 \rceil,
\end{align*}
thus completing the induction.
\end{proof}

As we mentioned before, in order not to overload, we simplified our notation and did not highlight the dependence on the metric $d'\in \text{hM}_d(X,\varphi)$. For this reason, the sequence $(\mathcal{U}^s_n)_{n\in \mathbb N}$ should more appropriately be denoted as $(\mathcal{U}^{s,d'}_n)_{n\in \mathbb N}$. Now, let $\mathcal{U}^{s,d'}_0=\{G^s(Q)\}$ and according to the proof of Theorem \ref{thm:Metrisation_thm} in \cite{Sakai_book}, the sequence $(\mathcal{U}^{s,d'}_n)_{n\geq 0}$ yields a $2$-quasimetric on $G^s(Q)$ given by
\begin{equation}\label{eq:groupoidquasimetric}
\rho_{s,d'}(a,b)=\inf \{ 2^{-n}:\, \text{there is} \, \, U\in \mathcal{U}^{s,d'}_n \, \, \text{such that} \, \, a,b \in U\}.
\end{equation}
This in turn gives the chain-metric $D_{s,d'}:G^s(Q)\times G^s(Q)\to [0,1]$ defined as
\begin{equation}\label{eq:groupoidmetric}
D_{s,d'}(a,b)=\inf \{\sum_{i=1}^{n-1} \rho_{s,d'}(c_i,c_{i+1}): n\in \mathbb N, \, c_i \in G^s(Q), \, c_1=a, \, c_n=b\},
\end{equation}
which satisfies 
\begin{equation}\label{eq:metricandquasimetric}
4^{-1}\rho_{s,d'}(a,b)\leq D_{s,d'}(a,b) \leq \rho_{s,d'}(a,b)
\end{equation}
for every $a,b\in G^s(Q)$. In fact, for every $a\in G^s(Q)$ it holds that
\begin{equation}\label{eq:groupoidmetricstar}
\st(a,\mathcal{U}^{s,d'}_{n+2})\subset \overline{B}_{D_{s,d'}}(a,2^{-n-2})\subset \st (a,\mathcal{U}^{s,d'}_n),
\end{equation}
and since the sequence $(\st(a,\mathcal{U}^{s,d'}_n))_{n\geq 0}$ is a neighbourhood base, the metric $D_{s,d'}$ generates the {\'e}tale topology on $G^s(Q)$. 

Let now $j:X^u(Q)\to G^s(Q)$ be the \textit{units embedding} given by
\begin{equation}
j(x)=(x,x).
\end{equation}
The restriction of $D_{s,d'}$ on $j(X^u(Q))$ generates the relative topology on $j(X^u(Q))$, and hence the pull-back metric $\widetilde{D_{s,d'}}: X^u(Q)\times X^u(Q)\to [0,1]$ given by 
\begin{equation}\label{eq:restricted_metric}
\widetilde{D_{s,d'}}(x,y)=D_{s,d'}(j(x),j(y)),
\end{equation}
generates the topology on $X^u(Q)$. As a result, we have proved the following.

\begin{lemma}\label{lem:generalgroupoidmetric}
For $d'\in \text{hM}_d(X,\varphi)$, the metric $D_{s,d'}$ induces the topology of $G^s(Q)$. Also, the pull-back metric $\widetilde{D_{s,d'}}$ induces the topology of the units space $X^u(Q)$.
\end{lemma} 

We continue by showing that for self-similar hyperbolic metrics the groupoids admit bi-Lipschitz dynamics. Again, we keep the notation to the minimum and when we work with the $d'$-model of $G^s(Q)$, where $d'\in \text{hM}_d(X,\varphi)$, we do not indicate the dependence of the Smale space constants and the local stable and unstable sets on $d'$. Also, the homeomorphism $\varphi$ behaves with respect to $d'$ and the base sets of $G^s(Q)$ are in $\mathcal{B}^s(Q,d')$. Nevertheless, we do highlight the dependence of the covers $\mathcal{U}^{s,d'}_n$, the quasimetric $\rho_{s,d'}$ and the metric $D_{s,d'}$ on $d'$. Before proceeding to the next lemma we should make the following observation.

Let $d'\in \text{sM}_d(X,\varphi)$ and consider, for a moment, the $d'$-model of $G^s(Q)$. Then, for every $0<\varepsilon \leq \lambda_X^{-1}\varepsilon_X$ and $z\in X$, we have
\begin{equation}\label{eq:mult_of_isom}
\varphi (X^u(z,\varepsilon))=X^u(\varphi(z),\varepsilon \lambda_X).
\end{equation}
Consequently, for the sequence $(\eta_n)_{n\geq 0}$ in (\ref{eq:widthofholonomies}) it holds that $\eta_n=\lambda_X^{-n}\varepsilon_X'/4$. As a result, for every $a\in G^s(Q)$, the neighbourhood base sets (\ref{eq:neigh_base_sets}) corresponding to $d'$ are
\begin{equation}\label{eq:selfsimilarbaseset}
V_n(a)=V^s(a,h_{a,n}^s,\lambda_X^{-N_{a,n}}\varepsilon_X'/4,N_{a,n}).
\end{equation}

\begin{lemma}\label{lem:generalgroupoidmetricLip}
Let $d'\in \text{sM}_d(X,\varphi)$ and consider the metric $D_{s,d'}$ on $G^s(Q)$. Then, the automorphism $\Phi=\varphi \times \varphi :(G^s(Q),D_{s,d'})\to (G^s(Q),D_{s,d'})$ is bi-Lipschitz with $$4^{-1}D_{s,d'}(a,b)\leq D_{s,d'}(\Phi(a),\Phi(b))\leq 8D_{s,d'}(a,b),$$ for every $a,b\in G^s(Q)$. 
\end{lemma}

\begin{proof}
Let $d'\in \text{sM}_d(X,\varphi)$. Consider the $d'$-model of $G^s(Q)$ so that $\varphi$ satisfies (\ref{eq:mult_of_isom}). Recall the quasimetric $\rho_{s,d'}$ in (\ref{eq:groupoidquasimetric}) and we aim to show that, for every $a,b\in G^s(Q)$, it holds 
\begin{enumerate}[(i)]
\item $\rho_{s,d'}(\Phi^{-1}(a),\Phi^{-1}(b))\leq \rho_{s,d'}(a,b)$;
\item $\rho_{s,d'}(\Phi^{-1}(a),\Phi^{-1}(b))\geq 2^{-1}\rho_{s,d'}(a,b)$.
\end{enumerate}
Then, using the inequalities (\ref{eq:metricandquasimetric}) we obtain the result. First we prove part (i). If $a=b$, the statement is trivial. Also, since $\rho_{s,d'}$ is bounded by $1$, if $\rho_{s,d'}(a,b)=1$ then $\rho_{s,d'}(\Phi^{-1}(a),\Phi^{-1}(b))\leq \rho_{s,d'}(a,b)$. 

Assume now that $a\neq b$ and $\rho_{s,d'}(a,b)=2^{-n}$, for some $n\geq 2$ (the case $n=1$ is treated in the exactly same way, however we avoid to show it in an attempt to reduce notation). There is $c\in G^s(Q)$ such that $a,b\in U_n(c)\in \mathcal{U}^{s,d'}_n$. From (\ref{eq:selfsimilarbaseset}) and since $k(c,n)\geq N_c$ (note that $n\geq 2$), we have that $$U_n(c)=V_{k(c,n)}(c)=V^s(c,h_{c,k(c,n)}^s,\lambda_X^{-k(c,n)}\varepsilon_X'/4,k(c,n)).$$
\enlargethispage{\baselineskip}\\
\enlargethispage{\baselineskip}
For simplicity, let us drop the notation of $h_{c,k(c,n)}^s$, since it is completely determined by $c,\, k(c,n)$ and $\lambda_X^{-k(c,n)}\varepsilon_X'/4$. It is straightforward to see that $$\Phi^{-1}(a),\Phi^{-1}(b)\in V^s(\Phi^{-1}(c),\lambda_X^{-k(c,n)-1}\varepsilon_X'/4,k(c,n)+1).$$ One can observe that $N_{\Phi^{-1}(c),1}\leq N_{c,1}+1$ and hence, $$k(\Phi^{-1}(c),n)\leq k(c,n)+1.$$ Note that $k(\Phi^{-1}(c),n)\geq N_{\Phi^{-1}(c)}$, and we have 
\begin{align*}
V^s(\Phi^{-1}(c),\lambda_X^{-k(c,n)-1}\varepsilon_X'/4,k(c,n)+1)&= V^s(\Phi^{-1}(c),\lambda_X^{-k(c,n)-1}\varepsilon_X'/4,k(\Phi^{-1}(c),n))\\
&\subset V^s(\Phi^{-1}(c),\lambda_X^{-k(\Phi^{-1}(c),n)}\varepsilon_X'/4,k(\Phi^{-1}(c),n))\\
&= V_{k(\Phi^{-1}(c),n)}(\Phi^{-1}(c))\\
&= U_n(\Phi^{-1}(c)).
\end{align*}
Consequently, it holds that $\rho_{s,d'}(\Phi^{-1}(a),\Phi^{-1}(b))\leq 2^{-n}=\rho_{s,d'}(a,b)$, thus completing the proof of part (i).

To prove part (ii) suppose that $\rho_{s,d'}(a,b)=2^{-n}$, for some $n\geq 0$. The claim is that $\rho_{s,d'}(\Phi^{-1}(a),\Phi^{-1}(b))\geq 2^{-n-1}.$ Assume to the contrary that $$\rho_{s,d'}(\Phi^{-1}(a),\Phi^{-1}(b))\leq 2^{-n-2}.$$ Then, there is some $c\in G^s(Q)$ such that $\Phi^{-1}(a),\Phi^{-1}(b)\in U_{n+2}(c)=V_{k(c,n+2)}(c)$, where $$V_{k(c,n+2)}(c)=V^s(c, \lambda_X^{-k(c,n+2)}\varepsilon_X'/4,k(c,n+2)).$$ One can immediately see that $$a,b\in V^s(\Phi(c), \lambda_X^{-k(c,n+2)+1}\varepsilon_X'/4,k(c,n+2)-1).$$ Note that the latter base set is well-defined since $k(c,n+2)-1\geq 0$. Also, observe that $N_{\Phi(c),1}\leq N_{c,1}$, and hence
\begin{align*}
k(\Phi(c),n+1)&=N_{\Phi(c),1}+n\lceil \log_{\lambda_X} 3 \rceil \\
&\leq N_{c,1}+n\lceil \log_{\lambda_X} 3 \rceil \\
&\leq N_{c,1}+n\lceil \log_{\lambda_X} 3 \rceil+\lceil \log_{\lambda_X} 3 \rceil-1\\
&= k(c,n+2)-1.
\end{align*}

As a result, 
\begin{align*}
V^s(\Phi(c), \lambda_X^{-k(c,n+2)+1}\varepsilon_X'/4,k(c,n+2)-1)&= V^s(\Phi(c), \lambda_X^{-k(c,n+2)+1}\varepsilon_X'/4,k(\Phi(c),n+1))\\
&\subset V^s(\Phi(c), \lambda_X^{-k(\Phi(c),n+1)}\varepsilon_X'/4,k(\Phi(c),n+1))\\
&= V_{k(\Phi(c),n+1)}(\Phi(c))\\
&= U_{n+1}(\Phi(c)).
\end{align*}
This means that $\rho_{s,d'}(a,b)\leq 2^{-n-1}$, leading to a contradiction. This completes the proof of part (ii).
\end{proof}

We continue by taking a closer look at the Lipschitz structure of $G^s(Q)$. Again, given the groupoid metric that corresponds to $d'\in \text{sM}_d(X,\varphi)$, we prefer to work with the $d'$-model of the groupoid, and the only part where we highlight the dependence on $d'$ is in the cover $\mathcal{U}^{s,d'}_n$, the quasimetric $\rho_{s,d'}$ and the metric $D_{s,d'}$.

\begin{lemma}\label{lem:bi_Lip_inversion}
For every $d'\in \text{sM}_d(X,\varphi)$, the inversion map on $G^s(Q)$ is bi-Lipschitz with respect to the metric $D_{s,d'}$. 
\end{lemma}

\begin{proof}
Consider the $d'$-model of $G^s(Q)$. Also, note that since $i^2=\id$, we only have to show that the inversion map $i$ is $\Lambda$-Lipschitz, for some $\Lambda \geq 1$ which depends on $d'$, and then $$\Lambda^{-1}D_{s,d'}(a,b)\leq D_{s,d'}(i(a),i(b))\leq \Lambda D_{s,d'}(a,b),$$ for every $a,b\in G^s(Q)$. As usual, instead of working directly with $D_{s,d'}$, we work with the quasimetric $\rho_{s,d'}$, and it suffices to find $\Lambda' \geq 1$ such that, 
\begin{equation}\label{eq:bi_Lip_inversion}
\rho_{s,d'}(i(a),i(b))\leq \Lambda' \rho_{s,d'}(a,b),
\end{equation}
for every $a,b\in G^s(Q)$. Then, following the metric inequalities (\ref{eq:metricandquasimetric}) we can choose $\Lambda =4\Lambda '$. Let $M\in \mathbb N$ such that $\lambda_X^{-M}\varepsilon_X/2 \leq \varepsilon_X'/4$, and the claim is that we can choose $$\Lambda'=2^{M+1}.$$ 

The case where $a=b$ or $\rho_{s,d'}(a,b)=1$ is trivial. Also, since $\rho_{s,d'}$ is bounded by $1$, if $\rho_{s,d'}(a,b)=2^{-n}$ for some $n\in \{1,\ldots, M+1\}$, then $\Lambda'=2^{M+1}$ satisfies (\ref{eq:bi_Lip_inversion}). Suppose now that $\rho_{s,d'}(a,b)=2^{-n}$ for some $n> M+1$. Then, there is $c=(c_1,c_2)\in G^s(Q)$ such that $a,b\in U_n(c)\in \mathcal{U}^{s,d'}_n$. From (\ref{eq:selfsimilarbaseset}) and the fact that $n\geq 2$, and hence $k(c,n)\geq N_c$, it holds that $$U_n(c)=V_{k(c,n)}(c)=V^s(c,h_{c,k(c,n)}^s,\lambda_X^{-k(c,n)}\varepsilon_X'/4,k(c,n)).$$ It is not hard to see that $$h_{c,k(c,n)}^s(X^u(c_2,\lambda_X^{-k(c,n)}\varepsilon_X'/4))\subset X^u(c_1,\lambda_X^{-k(c,n)}\varepsilon_X/2),$$ and for simplicity, we henceforth drop the notation of $h_{c,k(c,n)}^s$. Since $n-M\geq 2$, we have $$k(c,n)\geq k(c,n)-M\geq k(c,n-M)\geq N_c=N_{i(c)},$$ and also that $$k(c,m)=k(i(c),m),$$ for all $m\in \mathbb N$. As a result, the following calculations are well-defined, 
\begin{align*}
i(V^s(c,\lambda_X^{-k(c,n)}\varepsilon_X'/4,k(c,n)))&=i(V^s(c,\lambda_X^{-k(c,n)}\varepsilon_X'/4,k(c,n)-M))\\
&\subset V^s(i(c),\lambda_X^{-k(c,n)}\varepsilon_X/2,k(c,n)-M)\\
&\subset V^s(i(c),\lambda_X^{-k(c,n)+M}\varepsilon_X'/4,k(c,n)-M)\\
&= V^s(i(c),\lambda_X^{-k(c,n)+M}\varepsilon_X'/4,k(c,n-M))\\
&\subset V^s(i(c),\lambda_X^{-k(c,n-M)}\varepsilon_X'/4,k(c,n-M))\\
&= V_{k(i(c),n-M)}(i(c))\\
&= U_{n-M}(i(c)).
\end{align*}
Consequently, $\rho_{s,d'}(i(a),i(b))\leq 2^{-n+M}=2^M \rho_{s,d'}(a,b)< 2^{M+1} \rho_{s,d'}(a,b),$ and this completes the proof of the claim.
\end{proof}

We now state an elementary fact about locally bi-Lipschitz maps whose proof is omitted.

\begin{lemma}\label{lem:locally_bi_Lip}
Suppose that $f:(Y,d_Y)\to (Z,d_Z)$ is locally bi-Lipschitz and $K\subset Y$ is a compact set on which the restriction $f_K:K\to f(K)$ is a homeomorphism. Then, $f_K$ is bi-Lipschitz.
\end{lemma}

For the next lemma recall that $\mathcal{B}(G^s(Q))$ is the collection of all base set in $\mathcal{B}^s(Q,d')$, for all $d'\in \text{hM}_{d'}(X,\varphi)$. Moreover, given a metric $D_{s,d'}$, recall the pull-back metric $\widetilde{D_{s,d'}}$ on the units space $X^u(Q)$ along the embedding $j:X^u(Q)\to G^s(Q)$, see (\ref{eq:restricted_metric}). 

\begin{lemma}\label{lem:bi-Lipschitz_groupoid_maps}
For every $d'\in \text{sM}_d(X,\varphi)$ and every base set $V\in \mathcal{B}(G^s(Q))$, the restricted range and source maps $r_V$ and $s_V$ are bi-Lipschitz, with respect to the metrics $D_{s,d'}$ on $G^s(Q)$ and $\widetilde{D_{s,d'}}$ on $X^u(Q)$. 
\end{lemma}

\begin{proof}
Let $d'\in \text{sM}_d(X,\varphi)$ and we aim to prove that, with respect to $D_{s,d'}$ and $\widetilde{D_{s,d'}}$,
\begin{enumerate}[(i)]
\item $s_V$ is bi-Lipschitz, for every $V\in \mathcal{B}^s(Q,d')$;
\item $s_V$ is bi-Lipschitz, for every $V\in \mathcal{B}(G^s(Q))$;
\item $r_V$ is bi-Lipschitz, for every $V\in \mathcal{B}(G^s(Q))$.
\end{enumerate}

We begin with part (i). Similarly as before, it is enough to consider the quasimetric $\rho_{s,d'}$ (see the metric inequalities (\ref{eq:metricandquasimetric})) and we claim that, for $V\in \mathcal{B}^s(Q,d')$, there is some $M_V\in \mathbb N$ such that, 
\begin{equation}\label{eq:bi-Lipschitz_groupoid_maps}
2^{-M_V}\rho_{s,d'}(a,b)\leq \rho_{s,d'}(j(s_V(a)),j(s_V(b)))\leq \rho_{s,d'}(a,b),
\end{equation}
for every $a,b \in V$.

First, we show that $\rho_{s,d'}(j(s_V(a)),j(s_V(b)))\leq \rho_{s,d'}(a,b)$. The case where $a=b$ or $\rho_{s,d'}(a,b)=1$ is trivial. Suppose that $\rho_{s,d'}(a,b)=2^{-n}$, for some $n\in \mathbb N$. Then there is $c=(c_1,c_2)\in G^s(Q)$ so that $a,b\in U_n(c)\in \mathcal{U}^{s,d'}_n$. From (\ref{eq:selfsimilarbaseset}) we have that $$U_n(c)=V_{k(c,n)}(c)=V^s(c,h^s_{c,k(c,n)},\lambda_X^{-N_{c,k(c,n)}}\varepsilon_X'/4,N_{c,k(c,n)}),$$ and, in particular, $$s_V(a),s_V(b)\in X^u(c_2, \lambda_X^{-N_{c,k(c,n)}}\varepsilon_X'/4).$$ As a result, 
\begin{align*}
j(s_V(a)),\, j(s_V(b))&\in j(X^u(c_2, \lambda_X^{-N_{c,k(c,n)}}\varepsilon_X'/4))\\
&= V^s(j(c_2), \id, \lambda_X^{-N_{c,k(c,n)}}\varepsilon_X'/4,0).
\end{align*}
Observe that $N_{c,k(c,n)}\geq k(j(c_2),n)=N_{j(c_2),k(j(c_2),n)} $ and hence
\begin{align*}
V^s(j(c_2), \id, \lambda_X^{-N_{c,k(c,n)}}\varepsilon_X'/4,0)&\subset V^s(j(c_2), \id, \lambda_X^{-k(j(c_2),n)}\varepsilon_X'/4,0)\\
&= V^s(j(c_2), h^s_{j(c_2),k(j(c_2),n)}, \lambda_X^{-k(j(c_2),n)}\varepsilon_X'/4,k(j(c_2),n))\\
&= V_{k(j(c_2),n)}(j(c_2))\\
&= U_n(j(c_2)).
\end{align*}
Therefore, it holds that $$\rho_{s,d'}(j(s_V(a)),j(s_V(b)))\leq 2^{-n}=\rho_{s,d'}(a,b).$$ 

We now aim to show that there is $M_V\in \mathbb N$ such that, $$2^{-M_V}\rho_{s,d'}(a,b)\leq \rho_{s,d'}(j(s_V(a)),j(s_V(b))),$$ for every $a,b\in V$. Let us now make the definition of $V$ more precise and assume that $V=V^s(e, h^s,\eta ,N)$, where $e\in G^s(Q)$. Also, let $N'\in \mathbb N$ such that $\lambda_X^{-N'}\varepsilon_X/2 \leq \varepsilon_X'/2$. Then, for $a,b\in V$ we have that 
\begin{equation}\label{eq:bi-Lipschitz_groupoid_maps2}
\varphi^N(a_2)\in X^s(\varphi^N(a_1),\varepsilon_X/2),
\end{equation}
and in particular, 
\begin{equation}\label{eq:bi-Lipschitz_groupoid_maps3}
\varphi^{N+N'}(a_2)\in X^s(\varphi^{N+N'}(a_1),\varepsilon_X'/2).
\end{equation}
Similarly for $b\in V$. The claim is that we can choose $$M_V= \Bigl\lceil \frac{ \log_{\lambda_X}2+(N+N')(1+\lceil \log_{\lambda_X}3 \rceil)}{\lceil \log_{\lambda_X}3 \rceil} \Bigl\rceil.$$ 

To this purpose, suppose that $\rho_{s,d'}(a,b)=2^{-n}$, for some $n\geq 0$, and we claim that $$2^{-M_V-n}\leq \rho_{s,d'}(j(s_V(a)),j(s_V(b))).$$ Our approach is proof by contradiction. So, assume that $$\rho_{s,d'}(j(s_V(a)),j(s_V(b)))\leq 2^{-M_V-n-1}.$$ Then, there is $c\in G^s(Q)$ such that $$j(s_V(a)),\, j(s_V(b))\in U_{M_V+n+1}(c)\in \mathcal{U}^{s,d'}_{M_V+n+1},$$ where $$U_{M_V+n+1}(c)=V^s(c,\lambda_X^{-k(c,  M_V+n+1)}\varepsilon_X'/4, k(c,  M_V+n+1)).$$ In particular, it holds that 
\begin{equation}\label{eq:bi-Lipschitz_groupoid_maps4}
s_V(b)\in X^u(s_V(a), 2\lambda_X^{-k(c,  M_V+n+1)}\varepsilon_X'/4).
\end{equation}
Moreover,
\begin{equation}\label{eq:bi-Lipschitz_groupoid_maps5}
2\lambda_X^{-k(c,  M_V+n+1)}\varepsilon_X'/4\leq \lambda_X^{-k(a,n+N+N'+1)}\varepsilon_X'/4.
\end{equation}
Indeed, by removing $\varepsilon_X'/4$ from both sides, taking the logarithm $\log_{\lambda_X}$ and considering (\ref{eq:bi-Lipschitz_groupoid_maps3}), the inequality (\ref{eq:bi-Lipschitz_groupoid_maps5}) is true because
\begin{align*}
\log_{\lambda_X}2+k(a,n+N+N'+1)&=\log_{\lambda_X}2+N_{a,1}+(n+N+N')\lceil \log_{\lambda_X}3 \rceil\\
&\leq \log_{\lambda_X}2+(N+N')+(n+N+N')\lceil \log_{\lambda_X}3 \rceil\\
&\leq (M_V+n)\lceil \log_{\lambda_X}3 \rceil\\
&\leq N_{c,1}+(M_V+n)\lceil \log_{\lambda_X}3 \rceil\\
&= k(c,M_V+n+1).
\end{align*} 

In addition, since $k(a,n+N+N'+1)\geq N+N'\geq N_a$, the base set $$V^s(a,\lambda_X^{-k(a,n+N+N'+1)}\varepsilon_X'/4, N+N')$$ is well-defined and equal to 
$$
V^s(a,\lambda_X^{-k(a,n+N+N'+1)}\varepsilon_X'/4, k(a,n+N+N'+1))
= U_{n+N+N'+1}(a).
$$
Using (\ref{eq:bi-Lipschitz_groupoid_maps2}), (\ref{eq:bi-Lipschitz_groupoid_maps4}) and (\ref{eq:bi-Lipschitz_groupoid_maps5}), one can see that the following computations for $a=(a_1,a_2),\, b=(b_1,b_2)$ in $V=V^s(e, h^s,\eta ,N)$ are well-defined, that is,
\begin{align*}
\varphi^{-N-N'}[\varphi^{N+N'}(b_2),\varphi^{N+N'}(a_1)]&= \varphi^{-N}[\varphi^{N}(b_2),\varphi^{N}(a_1)]\\
&= \varphi^{-N}[\varphi^{N}(b_2),[\varphi^{N}(a_1),\varphi^N(e_1)]]\\
&= \varphi^{-N}[\varphi^{N}(b_2),\varphi^N(e_1)]\\
&= b_1,
\end{align*}
and hence $$b\in V^s(a,\lambda_X^{-k(a,n+N+N'+1)}\varepsilon_X'/4, N+N').$$ Consequently, $$\rho_{s,d'}(a,b)\leq 2^{-n-N-N'-1},$$ which leads to a contradiction since $\rho_{s,d'}(a,b)=2^{-n}$. This completes the proof of part (i).

Part (ii) follows easily from part (i). Indeed, since $\mathcal{B}^s(Q,d')$ is a topological base for $G^s(Q)$, part (i) implies that the source map $s$ is locally bi-Lipschitz, with respect to $D_{s,d'}$ and $\widetilde{D_{s,d'}}$. Moreover, every $V\in \mathcal{B}(G^s(Q))$ is pre-compact and the restriction of $s$ on $\cl(V)$ is again a homeomorphism. Then, the proof of part (ii) follows from Lemma \ref{lem:locally_bi_Lip}. Finally, for part (iii) observe that the range map $r$ is also locally bi-Lipschitz, because $r=s\circ i$ and the inversion map $i$ is bi-Lipschitz, see Lemma \ref{lem:bi_Lip_inversion}. Similarly, the result follows from Lemma \ref{lem:locally_bi_Lip}.
\end{proof}

\begin{prop}\label{prop:groupoiddistances}
Consider the metric $d'\in \text{sM}_d(X,\varphi)$ and the Smale space $(X,d',\varphi)$ with contraction constant $\lambda_{X,d'}>1.$ Moreover, let $c\in G^s(Q)$ and consider the base set $V_{d'}^s(c,h^s,\lambda_{X,d'}^{-N-k}\gamma,N)$ $\in \mathcal{B}^s(Q,d'),$ where $k\in \mathbb N$, and $\gamma>0$ does not depend on $k$. Then, we can find $\gamma'=\gamma'(\lambda_{X,d'},\gamma)>0$ so that, if $a,b\in V^s(c,h^s,\lambda_{X,d'}^{-N-k}\gamma,N)$, then $$D_{s,d'}(a,b)\leq 2^{-k/\lceil \log_{\lambda_{X,d'}} 3 \rceil} \gamma'.$$ 
\end{prop}

\begin{proof}
Let $d'\in \text{sM}_d(X,\varphi)$ and again, for the sake of simplicity, we keep the notation of the metric $d'$ only for $\mathcal{U}^{s,d'}_n, \rho_{s,d'}$ and $D_{s,d'}$. Let $k'>0$ be such that $\lambda_X^{-k'}\gamma\leq \varepsilon_X'/4.$ If $k< k'+\lceil \log_{\lambda_X} 3 \rceil$, let $\gamma''>0$ be big enough so that $$2^{-k/\lceil \log_{\lambda_X} 3 \rceil} \gamma'' \geq 1\geq D_{s,d'}(a,b).$$ If $k\geq k'+\lceil \log_{\lambda_X} 3 \rceil$, then $$\lambda_X^{-N-k}\gamma \leq \lambda_X^{-N-k+k'}\varepsilon_X'/4< \lambda_X^{-N}\varepsilon_X'/4.$$ By dropping the notation of the holonomy maps, for brevity, we have that 
\begin{align*}
V^s(c,\lambda_X^{-N-k}\gamma,N)&\subset V^s(c,\lambda_X^{-N-k+k'}\varepsilon_X'/4,N)\\
&=V^s(c,\lambda_X^{-N-k+k'}\varepsilon_X'/4,N+k-k')\\
&= V_{N+k-k'}(c).
\end{align*}
Ideally, the goal is to find the largest $n\geq 2$ such that $k(c,n)\leq N+k-k'$, and then $V_{N+k-k'}(c)\subset V_{k(c,n)}(c)=U_n(c)\in \mathcal{U}^{s,d'}_n$, leading to $$D_{s,d'}(a,b)\leq \rho_{s,d'}(a,b)\leq 2^{-n}.$$
\enlargethispage{\baselineskip}
For $n\geq 2$, we have 
\begin{align*}
k(c,n)&=N_{c,1}+(n-1)\lceil \log_{\lambda_X} 3 \rceil \\
&\leq N+(n-1)\lceil \log_{\lambda_X} 3 \rceil.
\end{align*}
Hence we can find the largest $n$ so that $N+(n-1)\lceil \log_{\lambda_X} 3 \rceil\leq N+k-k'$, which is $$n=\lfloor (k-k')/\lceil \log_{\lambda_X} 3 \rceil \rfloor+1 \geq 2.$$
It holds that $n>(k-k')/\lceil \log_{\lambda_X} 3 \rceil$ and hence $$2^{-n}< 2^{-k/\lceil \log_{\lambda_X} 3 \rceil} 2^{k'/\lceil \log_{\lambda_X} 3 \rceil}.$$ By choosing $\gamma'=\max \{ \gamma'', 2^{k'/\lceil \log_{\lambda_X} 3 \rceil}\}$ we complete the proof.
\end{proof}

\subsection{Optimisation of uniform convergence rates}\label{sec:Optimisation} In Section \ref{sec:SmoothRuelle} we intend to build a family of Lipschitz algebras of compactly supported functions, namely
$$\{\Lip_{c}(G^s(Q),D_{s,d'}): d'\in \text{sM}_d(X,\varphi)\},$$
which are dense $*$-subalgebras of $\mathcal{S}(Q)$. All these Lipschitz algebras are $\alpha_s$-invariant and therefore, for every $d'\in \text{sM}(X,\varphi)$, one can form the dense $*$-subalgebra 
\begin{equation}\label{eq:smoothest_Lip_alg}
\Lambda_{s,d'}=\Lip_{c}(G^s(Q),D_{s,d'})\rtimes_{\alpha_s,\text{alg}} \mathbb Z
\end{equation}
of the Ruelle algebra $\mathcal{R}^s(Q)$. One of our results will be that, for any $d'\in \text{sM}(X,\varphi)$, the algebra $\mathcal{R}^s(Q)$ has uniformly $\mathcal{L}^p$-summable $\Kt$-homology on $\Lambda_{s,d'}$, for every $p>p(d')$, where the constant $p(d')$ depends on the topological entropy $\ent(\varphi)$ and the metric $d'$. 

It turns out that, if $d_1,d_2\in \text{sM}_d(X,\varphi)$ and the corresponding contraction constants satisfy $\lambda_{X,d_1}\geq \lambda_{X,d_2}$, we have  
\begin{equation}
p(d_1)\leq p(d_2),
\end{equation}
meaning that $\Lambda_{s,d_1}$ is more smooth than $\Lambda_{s,d_2}$. This fact is related to Proposition \ref{prop:groupoiddistances} which offers a variety of exponential uniform convergence rates, each one depending on a choice of metric in $\text{sM}_d(X,\varphi)$, that can be used to estimate the convergence rates of the limits in Lemma \ref{lem:plusinftyKPW}, over the aforementioned Lipschitz subalgebras. Therefore, the question is, \textit{to what extent can we find such \enquote{smooth} Lipschitz subalgebras?}

\begin{definition}\label{def:l_number}
For every $d'\in \text{sM}_d(X,\varphi)$ consider the Smale space $(X,d',\varphi)$ with contraction constant $\lambda_{X,d'}>1.$ The quantity $$\lambda(X,\varphi)=\sup \{\lambda_{X,d'}: d'\in \text{sM}_d(X,\varphi)\}$$ is called the $\lambda$\textit{-number} of the Smale space $(X,\varphi)$.
\end{definition}

It turns out that the $\lambda$-numbers (also considered in \cite{Artigue}) are topological invariants of Smale spaces. Before proving this, we present the following lemma whose proof involves a straightforward manipulation of the Smale space axioms, and therefore is omitted.
\enlargethispage{\baselineskip}

\begin{lemma}\label{lem:pullback_Smale_space}
Let $(Y,\psi)$ be a topological dynamical system and $(Z,d_Z,\zeta, \varepsilon_Z,\lambda_Z)$ be a Smale space with locally defined bracket map $[\cdot, \cdot ]_Z$. Suppose that these systems are topologically conjugate via the conjugating homeomorphism $f:(Y,\psi)\to (Z,\zeta)$. Then $(Y,\psi)$ admits a Smale space structure with
\begin{enumerate}[(1)]
\item $d_Y(y_1,y_2)=d_Z(f(y_1),f(y_2))$, for any $y_1,y_2\in Y$;
\item $\varepsilon_Y=\varepsilon_Z$ and $$[y_1,y_2]_Z=f^{-1}[f(y_1),f(y_2)]_Z,$$
whenever $d_Y(y_1,y_2)\leq \varepsilon_Y$;
\item $\lambda_Y=\lambda_Z$.
\end{enumerate}
Moreover, if $(Z,d_Z,\zeta, \varepsilon_Z,\lambda_Z)$ is self-similar, the Smale space $(Y,d_Y,\psi,\varepsilon_Y,\lambda_Y)$ is too.
\end{lemma}

\begin{prop}\label{prop:lambda_number_top.invar.}
If the Smale space $(Y,\psi)$ is topologically conjugate to the Smale space $(Z,\zeta)$, then $\lambda(Y,\psi)=\lambda(Z,\zeta)$.
\end{prop}

\begin{proof}
Let us denote the metrics on $Y$ and $Z$ by $d_Y$ and $d_Z$. The result follows immediately from Lemma \ref{lem:pullback_Smale_space}, which shows that $$\{\lambda_{X,d'}: d'\in \text{sM}_{d_Y}(X,\varphi)\}=\{\lambda_{X,d'}: d'\in \text{sM}_{d_Z}(X,\varphi)\}.\qedhere$$
\end{proof}

Interestingly enough, the $\lambda$-numbers seem to capture some topological information about the Smale space. The next result is a consequence of \cite[Theorem 7.6]{Gero}. Also, recall that the topological dimension of $X$ is always finite.

\begin{prop}\label{prop:lambda_number_dimension}
It holds that $\lambda(X,\varphi)<\infty$ if and only if $\dim X >0$. Specifically, if $\dim X>0$ then, $$\lambda(X,\varphi)\leq e^{2\ent(\varphi)/\dim X}.$$
\end{prop}

\begin{proof}
First assume that $\dim X>0$. From \cite[Theorem 7.6]{Gero} we obtain that, for every $d'\in \text{sM}_d(X,\varphi)$, the Hausdorff dimension satisfies $$\dim_{H}(X,d')=\frac{2\ent(\varphi)}{\log\lambda_{X,d'}}.$$ Since $\dim X\leq \dim_{H}(X,d')$ (see \cite{Falconer}), by taking the supremum over $d'\in \text{sM}_d(X,\varphi)$ we obtain that $\lambda(X,\varphi)\leq e^{2\ent(\varphi)/\dim X}.$

We now prove the converse by showing that, if $\dim X=0$, then $\lambda(X,\varphi)=\infty$. First, we have that $(X,\varphi)$ is topologically conjugate to a TMC $(Y,\psi)$, as we see in \cite[Theorem 2.2.8]{Putnam_Book}. As a result, from Proposition \ref{prop:lambda_number_top.invar.} we get that $\lambda(X,\varphi)=\lambda(Y,\psi)$. Now, it is straightforward to see that the ultrametric (\ref{eq:SFTmetric}), which induces the product topology on $Y$, can have any factor $\lambda >1$ instead of just $\lambda=2$. All of these ultrametrics are self-similar for $\psi$. As a result, $\lambda(Y,\psi)=\infty$.
\end{proof}

Consequently, our quest for finding the smoothest Lipschitz subalgebra (\ref{eq:smoothest_Lip_alg}) of $\mathcal{R}^s(Q)$ by varying the metric $d'\in \text{sM}_d(X,\varphi)$ has a topological limitation, see also Remark \ref{rem:totalinf}. Another way to optimise even further our approach of using the Alexandroff-Urysohn-Frink Metrisation Theorem, would be to find sharper estimates in Proposition \ref{prop:groupoiddistances}. More precisely, the diameter of the sets $V_{d'}^s(c,h^s,\lambda_{X,d'}^{-N-k}\gamma,N)$ goes exponentially fast to zero, as $k$ goes to infinity. However, the base of the exponent, namely $$2^{-1/\lceil \log_{\lambda_{X,d'}} 3 \rceil},$$ is by no means random. In fact, it is the best possible (in this generality of stable and unstable groupoids) achieved by this metrisation method. Specifically, the number $3$ is related to the following fact: in a metric space $(Z,d)$, given any two sufficiently small intersecting balls $B_d(z_1,r_1)$ and $B_d(z_2,r_2)$, it holds $B_d(z_1,r_1)\subset B_d(z_2,3r)$, where $r=\max\{r_1,r_2\}$. The number $2$ is related to the fact that $\rho_{s,d'}$ is a $2$-quasimetric, and somehow seems to be important. In \cite{Schroeder}, the author constructs a $K$-quasimetric space, with $K>2$, for which Frink's chain-metric approach does not work.

\subsection{Groupoid ultrametrics for topological Markov chains}\label{sec:Groupoid_metrics_SFT} If the irreducible Smale space $(X,d,\varphi)$ is zero-dimensional, it is possible to build tractable ultrametrics on $G^s(Q)$ without using the Alexandroff-Urysohn-Frink Metrisation Theorem \ref{thm:Metrisation_thm}. With these ultrametrics one also obtains arbitrarily fast uniform convergence rates, in the sense of Proposition \ref{prop:groupoiddistances}, reflecting the fact that $\lambda(X,\varphi)=\infty$.

Indeed, assume that $(X,d,\varphi)$ is an irreducible TMC and the metric $d$, denoted henceforth by $d_{\lambda}$, is the usual self-similar ultrametric with expanding factor $\lambda >1$. Observe that the algebraic and topological structures of $G^s(Q)$ do not depend on $\lambda$, since for every $\lambda_1>1$ and $\lambda_2>1$ the metrics $d_{\lambda_1}$ and $d_{\lambda_2}$ generate the same topology. Also, for simplicity the elements of $G^s(Q)$ will be denoted by $a=(a_1,a_2),\, b=(b_1,b_2)$ and $c=(c_1,c_2)$.

We now mention a few basic facts about $(X,d_{\lambda},\varphi)$. First, the expansivity constant $\varepsilon_X=\lambda^{-1}$. Also, in our calculations there is no need to consider $\varepsilon_X/2$ since $d_{\lambda}$ is an ultrametric. Similarly, the constant $\varepsilon_X'\leq \varepsilon_X/2$ (see (\ref{eq:uniquebracket})) is also not needed because, whenever $d_{\lambda}(x,y)\leq \lambda^{-1}$, then 
\begin{equation}\label{eq:SFT_groupoid_ineq_1}
d_{\lambda}(x,[x,y]),\, d_{\lambda}(y,[x,y])\leq d_{\lambda}(x,y).
\end{equation} 

As a result, for every $c\in G^s(Q)$ it is enough to consider only the holonomy maps $h^s:X^u(c_2,\eta)\to X^u(c_1,\lambda^{-N-1})$ given by $$h^s(z)=\varphi^{-N}[\varphi^N(z),\varphi^N(c_1)],$$ where $N\geq 0$ is such that 
\begin{equation}\label{eq:SFT_groupoid_eq_3}
\varphi^N(c_2)\in X^s(\varphi^N(c_1),\lambda^{-1}) 
\end{equation}
and $\eta \leq \lambda^{-N-1}$. With that said, it can be assumed that the holonomy maps of all base sets $V^s(c,h^s,\eta,N)\subset G^s(Q)$ are as above. Note that these base sets will depend on the metric $d_{\lambda}$. 

It is important to mention that, for every $b\in V^s(c,h^s,\eta,N)$, we also have that
$\varphi^N(b_2)\in X^s(\varphi^N(b_1),\lambda^{-1}).$ To see this, note that $\varphi^N(b_1)=[\varphi^N(b_2),\varphi^N(c_1)]$, and also that 
\begin{align*}
d_{\lambda}(\varphi^N(b_2), [\varphi^N(b_2),\varphi^N(c_1)])&\leq d_{\lambda}(\varphi^N(b_2), \varphi^N(c_1))\\
&\leq \max \{ d_{\lambda}(\varphi^N(b_2), \varphi^N(c_2)), d_{\lambda}(\varphi^N(c_2), \varphi^N(c_1))\}\\
&\leq \lambda^{-1}.
\end{align*}

Moreover, every such $h^s$ is actually an isometry, and hence $$h^s:X^u(c_2,\eta)\to X^u(c_1,\eta).$$ To see this, first observe that, for every $x,y\in X$ with $d_{\lambda}(x,y)\leq \lambda^{-1}$, we have \begin{equation}\label{eq:SFT_groupoid_ineq_2}
d_{\lambda}(x,[y,x])=d_{\lambda}(y,[x,y]).
\end{equation} 
Let now $a,b\in V^s(c,h^s,\eta,N)$, then $a_2,b_2\in X^u(c_2,\eta)$ and $a_1,b_1\in X^u(c_1,\lambda^{-N-1})$. From the previous discussion we obtain that $\varphi^N(b_2)\in X^s(\varphi^N(b_1),\lambda^{-1})$, and hence 
\begin{align*}
a_1&=h^s(a_2)\\
&= \varphi^{-N}[\varphi^N(a_2),\varphi^N(c_1)]\\
&= \varphi^{-N}[\varphi^N(a_2),[\varphi^N(c_1),\varphi^N(b_1)]]\\
&= \varphi^{-N}[\varphi^N(a_2),\varphi^N(b_1)].
\end{align*}
Then, since $d_{\lambda}(\varphi^N(b_2),\varphi^N(a_1))\leq \lambda^{-1}$, from (\ref{eq:SFT_groupoid_ineq_2}) it holds that $$d_{\lambda}(\varphi^N(b_2),\varphi^N(a_2))=d_{\lambda}(\varphi^N(b_1),\varphi^N(a_1)).$$ Finally, since $\varphi^{-1}$ is the $\lambda^{-1}$-multiple of an isometry on local unstable sets, it follows that $$d_{\lambda}(b_2,a_2)=d_{\lambda}(b_1,a_1)=d_{\lambda}(h^s(b_2),h^s(a_2)).$$

In an attempt to metrise $G^s(Q)$, define $d_{s,\lambda}:G^s(Q)\times G^s(Q)\to [0,1]$ by 
\begin{equation}
d_{s,\lambda}(a,b)=
\begin{cases}
\max \{d_{\lambda}(a_1,b_1),d_{\lambda}(a_2,b_2)\}, & \text{if} \enspace b_2\in X^u(a_2,\lambda^{-1}),\, b_1\in X^u(a_1,\lambda^{-1})\\
1, & \text{if else.}
\end{cases}
\end{equation}
Using the fact that $d_{\lambda}$ is an ultrametric, it is not hard to show that $d_{s,\lambda}$ is also an ultrametric. However, it generates a strictly weaker topology than that of $G^s(Q)$ because there is no open ball, with respect to $d_{s,\lambda}$, that fits inside a base set of the form $V^s(a,h^s,\eta,N)$. The reason is that every global stable set wraps densely around every local unstable set of the same mixing component. 

To remedy this, consider the map $c_s:G^s(Q)\to \mathbb N$ that takes each $a\in G^s(Q)$ to the first time $N_{a}\geq 0$; that is, the minimum positive integer for which (\ref{eq:SFT_groupoid_eq_3}) holds.  One should also compare this definition with the one in (\ref{eq:first_time}). The map $c_s$ depends on $d_{\lambda}$, and the next lemma shows that it is continuous. Therefore, $G^s(Q)$ decomposes into the clopen subsets 
\begin{equation}\label{eq:SFT_groupoid_decomposition}
G^s(Q)=\bigsqcup_{N\geq 0}c_s^{-1}(N).
\end{equation}

In what follows, for simplifying the notation, a base set of the form $V^s(c,h^s,\eta,N)$ will often be written as $V^s(c,\eta,N)$.

\begin{lemma}\label{lem:Groupoid_metrics_SFT_1}
If $a\in c_s^{-1}(N)$ and $n\geq N$, then $V^s(a,\lambda^{-n-1},n)\subset c_s^{-1}(N).$
\end{lemma}

\begin{proof}
Due to self-similarity it holds that $$V^s(a,\lambda^{-n-1},n)=V^s(a,\lambda^{-n-1},N).$$ Let $b\in V^s(a,\lambda^{-n-1},N)$ and we claim $c_s(b)=N$. We have $\varphi^N(b_2)\in X^s(\varphi^N(b_1),\lambda^{-1})$ and so $c_s(b)\leq N$. Therefore, the base set $V^s(b,\lambda^{-n-1},N)$ is well-defined and one can easily show that it contains $a$. Now, due to self-similarity it holds that $$V^s(b,\lambda^{-n-1},N)=V^s(b,\lambda^{-n-1},c_s(b)),$$ and similarly one can show that $N=c_s(a)\leq c_s(b).$ This completes the proof.
\end{proof}

Let now $D_{s,\lambda}:G^s(Q)\times G^s(Q)\to [0,1]$ be defined by
\begin{equation}\label{eq:SFTgroupoidmetric}
D_{s,\lambda}(a,b)=
\begin{cases}
d_{s,\lambda}(a,b), & \text{if} \enspace  c_s(a)=c_s(b)\\
1, & \text{otherwise}.
\end{cases}
\end{equation}
Similarly, it is straightforward to show that $D_{s,\lambda}$ is an ultrametric.

\begin{lemma}\label{lem:SFT_groupoid_topology}
The ultrametric $D_{s,\lambda}$ generates the topology of $G^s(Q)$.
\end{lemma}

\begin{proof}
Let $a\in G^s(Q)$ and consider the base set $V^s(a,\lambda^{-n-1},N),$ where $n\geq N$. Due to self-similarity the base set is equal to $V^s(a,\lambda^{-n-1},c_s(a))$. We claim that $$B_{D_{s,\lambda}}(a,\lambda^{-n-1})\subset V^s(a,\lambda^{-n-1},c_s(a)).$$ Indeed, pick $b$ in the open ball. We have $c_s(b)=c_s(a)=M$ and $b_2\in X^u(a_2,\lambda^{-n-1})$, $b_1\in X^u(a_1,\lambda^{-n-1}).$ Moreover, $b\in V^s(a,\lambda^{-n-1},c_s(a))$ because 
\begin{align*}
\varphi^{-M}[\varphi^M(b_2),\varphi^M(a_1)]&= \varphi^{-M}[\varphi^M(b_2),[\varphi^M(a_1),\varphi^M(b_1)]]\\
&=\varphi^{-M}[\varphi^M(b_2),\varphi^M(b_1)]\\
&= b_1.
\end{align*}

For the reverse inclusion, consider the open ball $B_{D_{s,\lambda}}(a,\eta)$, where $\eta \leq \lambda^{-1}$. Then, $$B_{D_{s,\lambda}}(a,\eta)=\{b\in c_s^{-1}(c_s(a)):b_2\in X^u(a_2,\eta),\, b_1\in X^u(a_1,\eta)\}.$$ Since $c_s$ is continuous, there is $\delta >0, N\geq 0$ such that $V^s(a,\delta,N)\subset c_s^{-1}(c_s(a)).$ Choosing $\delta$ small enough gives that $V^s(a,\delta,N)\subset B_{D_{s,\lambda}}(a,\eta).$ As a result, the topology generated by $D_{s,\lambda}$ agrees with the topology of $G^s(Q)$.
\end{proof}

The following statement is the zero-dimensional analogue of Proposition \ref{prop:groupoiddistances}.

\begin{prop}\label{prop:SFT_groupoid_distances}
Let $c\in G^s(Q)$ and $n\geq c_s(c)$. For every $a,b$ in $V^s(c,\lambda^{-n-1},n)$ it holds that $$D_{s,\lambda}(a,b)\leq \lambda^{-n-1}.$$
\end{prop}

\begin{proof}

From Lemma \ref{lem:Groupoid_metrics_SFT_1} we have that $c_s(a)=c_s(b)$. Moreover, $b_2\in X^u(a_2,\lambda^{-n-1})$ and since the holonomy map is isometric we also have that $b_1\in X^u(a_1,\lambda^{-n-1})$. The result follows.
\end{proof}

The ultrametric $D_{s,\lambda}$ is very dynamic in nature. Before proving this, we should note that the automorphism $\Phi=\varphi \times \varphi :G^s(Q)\to G^s(Q)$ shifts the decomposition (\ref{eq:SFT_groupoid_decomposition}). More precisely, for every $N\in \mathbb N$ it holds that 
\begin{equation}\label{eq:SFT_groupoid_decomposition_shift}
\Phi^{-1}(c_s^{-1}(N))=c_s^{-1}(N+1),
\end{equation}
while the case $N=0$ is different since $c_s^{-1}(1)$ is strictly contained in $\Phi^{-1}(c_s^{-1}(0))$. 

Moreover, let $\widetilde{D_{s,\lambda}}$ denote the induced ultrametric on the space of units $X^u(Q)$, just like in (\ref{eq:restricted_metric}). In this case, however, since $X^u(Q)$ embeds in $c_s^{-1}(0)$, it is straightforward to see that
\begin{equation}\label{eq:SFTgroupoidmetric_units_space}
\widetilde{D_{s,\lambda}}(x,y)=
\begin{cases}
d_{\lambda}(x,y), & \text{if} \enspace  y\in X^u(x,\lambda^{-1})\\
1, & \text{if else}.
\end{cases}
\end{equation}

The next proposition is the zero-dimensional counterpart of Theorem \ref{thm:GroupoidMetrisation}. We prefer to work with $\Phi^{-1}$. 

\begin{prop}\label{prop:SFT_groupoids_Lipschitz}
Let $(X,d_{\lambda},\varphi)$ be an irreducible TMC with stable groupoid $G^s(Q)$. Also, for each $\kappa >1$ consider the equivalent ultrametric $d_{\kappa }$. Then, there is a family of compatible ultrametrics $\{D_{s,\kappa }:\kappa >1\}$ on $G^s(Q)$ so that, with respect to each of these ultrametrics,
\begin{enumerate}[(1)]
\item the groupoid automorphism $\Phi^{-1}:G^s(Q)\to G^s(Q)$ is bi-Lipschitz with $$\kappa^{-1}D_{s,\kappa}(a,b)\leq D_{s,\kappa}(\Phi^{-1}(a),\Phi^{-1}(b))\leq D_{s,\kappa}(a,b),$$ for every $a,b\in G^s(Q)$;
\item the map $\Phi^{-1}$ is locally contracting so that, if $D_{s,\kappa}(a,b)\leq \kappa^{-1}$, then $$D_{s,\kappa}(\Phi^{-1}(a),\Phi^{-1}(b))=\kappa^{-1}D_{s,\kappa}(a,b);$$
\item the range and source maps onto $(X^u(Q), \widetilde{D_{s,\kappa}})$ are locally isometric. Moreover, the inversion map is isometric.
\end{enumerate}
\end{prop}

\begin{proof}

Let $\kappa >1$. Lemma \ref{lem:SFT_groupoid_topology} states that $D_{s,\kappa}$ generates the topology of $G^s(Q)$. Moreover, we can assume that $G^s(Q)$ is built from the Smale space $(X,d_{\kappa},\varphi)$ as the algebraic and topological structure of $G^s(Q)$ is given also by $d_{\kappa}$. This means that all the following notation is relative to $d_{\kappa}$. 

To prove parts (1) and (2), let $a,b\in G^s(Q)$ and first assume that $D_{s,\kappa}(a,b)=1$. Then, it clearly holds that $$D_{s,\kappa}(\Phi^{-1}(a),\Phi^{-1}(b))\leq D_{s,\kappa}(a,b).$$ We claim that $D_{s,\kappa}(\Phi^{-1}(a),\Phi^{-1}(b))\geq \kappa^{-1}.$ To prove this, assume to the contrary that 
\begin{equation}\label{eq:SFT_Groupoids_Contradiction}
D_{s,\kappa}(\Phi^{-1}(a),\Phi^{-1}(b))\leq \kappa^{-2}.
\end{equation}
Then, $c_s(\Phi^{-1}(a))=c_s(\Phi^{-1}(b))$ and $$\varphi^{-1}(b_2)\in X^u(\varphi^{-1}(a_2),\kappa^{-2}),\,\, \varphi^{-1}(b_1)\in X^u(\varphi^{-1}(a_1),\kappa^{-2}).$$ Due to self-similarity we get that $$b_2\in X^u(a_2,\kappa^{-1}),\,\, b_1\in X^u(a_1,\kappa^{-1}).$$ In addition, we have that $c_s(a)=c_s(b)$. Indeed, if $c_s(\Phi^{-1}(a))=c_s(\Phi^{-1}(b))\geq 1$, then $c_s(a)=c_s(\Phi^{-1}(a))-1$ and $c_s(b)=c_s(\Phi^{-1}(b))-1$. If $c_s(\Phi^{-1}(a))=c_s(\Phi^{-1}(b))=0$, we have that $\Phi^{-1}(b)\in V^s(\Phi^{-1}(a),\kappa^{-2},0),$ and hence $$b\in V^s(a,\kappa^{-1},0).$$ From Lemma \ref{lem:Groupoid_metrics_SFT_1} we get that $c_s(a)=c_s(b)$. To summarise, the assumption (\ref{eq:SFT_Groupoids_Contradiction}) leads to $D_{s,\kappa}(a,b)\leq \kappa^{-1}$, which is a contradiction.

Assume now that $D_{s,\kappa}(a,b)\leq \kappa^{-1}$. Then, $c_s(a)=c_s(b)$ and $$b_2\in X^u(a_2,\kappa^{-1}),\, b_1\in X^u(a_1,\kappa^{-1}).$$ As a result, $$\varphi^{-1}(b_2)\in X^u(\varphi^{-1}(a_2),\kappa^{-2}),\, \varphi^{-1}(b_1)\in X^u(\varphi^{-1}(a_1),\kappa^{-2}),$$ and $c_s(\Phi^{-1}(a))=c_s(\Phi^{-1}(b))$. For the latter equality note that, if $c_s(a)=c_s(b)\geq 1$, then $c_s(\Phi^{-1}(a))=c_s(a)+1$ and $c_s(\Phi^{-1}(b))=c_s(b)+1$. Now, if $c_s(a)=c_s(b)=0$, then $b\in V^s(a,\kappa^{-1},0)$ and hence $$\Phi^{-1}(b)\in V^s(\Phi^{-1}(a),\kappa^{-2},1),$$ meaning that $c_s(\Phi^{-1}(a))=c_s(\Phi^{-1}(b))$. Note that $c_s(\Phi^{-1}(a))\leq 1$. In general, 
\begin{align*}
D_{s,\kappa}(\Phi^{-1}(a),\Phi^{-1}(b))&=d_{s,\kappa}(\Phi^{-1}(a),\Phi^{-1}(b))\\
&= \max \{ d_{\kappa}(\varphi^{-1}(a_1),\varphi^{-1}(b_1)), d_{\kappa}(\varphi^{-1}(a_2),\varphi^{-1}(b_2))\}\\
&= \max \{ \kappa^{-1}d_{\kappa}(a_1,b_1), \kappa^{-1}d_{\kappa}(a_2,b_2)\}\\
&=\kappa^{-1}d_{s,\kappa}(a,b)\\
&=\kappa^{-1}D_{s,\kappa}(a,b).
\end{align*}

For part (3) it suffices to show that, for an arbitrary base set $V=V^s(c,\kappa^{-n-1},n)$, the restriction of the source map $s_V:V\to s(V)$ is an isometry. This is because every base set is open in the topology generated by $D_{s,\kappa}$. Also, recall that $s_V$ is bijective. Let $a,b \in V$, then $c_s(a)=c_s(b)$ and $$b_2\in X^u(a_2,\kappa^{-n-1}),\, b_1\in X^u(a_1,\kappa^{-n-1}),$$ since $d_{\kappa}(a_2,b_2)=d_{\kappa}(a_1,b_1)$ (recall that all holonomy maps are isometric). Therefore, 
\begin{align*}
D_{s,\kappa}(a,b)&=\max \{ d_{\kappa}(a_1,b_1), d_{\kappa}(a_2,b_2)\}\\
&= d_{\kappa}(a_2,b_2)\\
&= \widetilde{D_{s,\kappa}}(s(a),s(b)).
\end{align*}
Similarly, one can show that the range map is locally isometric. The fact that the inversion is an isometry is straightforward.
\end{proof}

At this stage, it is still unclear if a similar approach could work for general Smale spaces, and this is the main reason for considering Theorem \ref{thm:GroupoidMetrisation}. But it should be mentioned that, if such metrisation approach is possible, the metrics might be very different from the ones obtained in Theorem \ref{thm:GroupoidMetrisation}. In fact, for the metrics $D_{s,\kappa}$ and $D_{s,d_{\kappa}}$, where $\kappa>1$, even though they generate the same topology and have similar dynamical behaviour, one can see that they are not uniformly equivalent. 

Specifically, there exist sequences $(a_n)_{n\geq 0},(b_n)_{n\geq 0} \subset G^s(Q)$ and some $t>0$ such that $\lim_n D_{s,\kappa}(a_n,b_n)=0$ while $D_{s,d_\kappa}(a_n,b_n)>t$, for all $n\geq 0$. The only reason this is happening is because $\lceil \log_{\lambda_{X,d_{\kappa}}} 3 \rceil$ is used in the definition of the sequences $j(a,n)$ in (\ref{eq:J(a,n)}), which is related to a triangle inequality, see Lemma \ref{lem:groupoidNB2}. However, if we remove $\lceil \log_{\lambda_{X,d_{\kappa}}} 3 \rceil$ from the definition of $D_{s,d_{\kappa}}$ (in the case of TMC is no longer needed) the two metrics become equivalent. 

In any case, Proposition \ref{prop:SFT_groupoids_Lipschitz} (just like Theorem \ref{thm:GroupoidMetrisation}) is exactly what we need to construct dense $*$-subalgebras of $\mathcal{S}(Q)$ which are $\alpha_s$-invariant and hence induce dense $*$-subalgebras of $\mathcal{R}^s(Q)$, see Subsection \ref{sec:Lip_Ruelle_Smooth_Ext}.

\section{K-homological finiteness of Ruelle algebras}\label{sec:SmoothRuelle}
We prove that the groupoid metrics of Section \ref{sec:Metrise_Groupoids} yield dense Lipschitz $*$-subalgebras of the stable and unstable Ruelle algebras. Then, we study commutation relations between stable and unstable Lipschitz algebras. Eventually, we obtain the smoothness of the KPW-extension and the $\Kt$-homological finiteness of Ruelle algebras.

Let $(X,d,\varphi)$ be an irreducible Smale space with periodic orbits $P,Q$ such that $P\cap Q=\varnothing$. Theorem \ref{thm:GroupoidMetrisation} associates to every self-similar metric $d'\in \text{sM}_d(X,\varphi)$ the compatible groupoid metrics $D_{s,d'}$ and $D_{u,d'}$ on $G^s(Q)$ and $G^u(P)$. In the case where $(X,d,\varphi)$ is a TMC, a different construction (Proposition \ref{prop:SFT_groupoids_Lipschitz}) produces ultrametrics $D_{s,\kappa}$ and $D_{u,\kappa}$ on $G^s(Q)$ and $G^u(P)$, for every expanding factor $\kappa>1$.

\subsection{Lipschitz subalgebras of Ruelle algebras}\label{sec:Lip_Ruelle_Smooth_Ext}

All these aforementioned facts combined with the results of Section \ref{sec:Lip_etale_alg} allow us to deduce the following.

\begin{prop}\label{prop:Lip_alg_Smale_groupoids}
Let $d'\in \text{sM}_d(X,\varphi)$. The complex vector space of compactly supported Lipschitz functions $\Lip_c(G^s(Q),D_{s,d'})$ forms a dense $*$-subalgebra of $\mathcal{S}(Q)$. Moreover, it is $\alpha_s$-invariant, and therefore, the algebraic crossed product $$\Lambda_{s,d'}(Q,\alpha_s)=  \Lip_c(G^s(Q),D_{s,d'})\rtimes_{\alpha_s,\text{alg}} \mathbb Z$$ 
is a well-defined dense $*$-subalgebra of $\mathcal{R}^s(Q)$. Similarly, in the unstable case we obtain the dense $*$-subalgebra $$\Lambda_{u,d'}(P,\alpha_u)=  \Lip_c(G^u(P),D_{u,d'})\rtimes_{\alpha_u,\text{alg}} \mathbb Z$$ of the Ruelle algebra $\mathcal{R}^u(P)$.
\end{prop}

\begin{proof}
We only prove it for the stable case. From Theorem \ref{thm:GroupoidMetrisation} we have that the range and source maps of $G^s(Q)$ are locally bi-Lipschitz with respect to $D_{s,d'}$. Therefore, from Proposition \ref{prop:Lip_dense_alg_general} we obtain that $\Lip_c(G^s(Q),D_{s,d'})$ is a convolution $*$-algebra whose image in $\mathcal{S}(Q)$ is dense.

Moreover, the groupoid automorphism $\Phi=\varphi \times \varphi :G^s(Q)\to G^s(Q)$ is bi-Lipschitz with respect to $D_{s,d'}$ (Theorem \ref{thm:GroupoidMetrisation}), and since for $a\in C_c(G^s(Q))$ we have $\alpha_s(a)= a\circ \Phi^{-1},$ then $$\alpha_s(\Lip_c(G^s(Q),D_{s,d'}))=\Lip_c(G^s(Q),D_{s,d'}).$$ As a result, we can form the algebraic crossed product $\Lambda_{s,d'}(Q,\alpha_s)$ which is a $*$-subalgebra of $\mathcal{R}^s(Q)$. Finally, $\Lambda_{s,d'}(Q,\alpha_s)$ is dense in $\mathcal{R}^s(Q)$ since $\Lip_c(G^s(Q),D_{s,d'})$ is dense in $\mathcal{S}(Q)$, see \cite[Remark 2.30]{DWilliams}.
\end{proof}

In exactly the same way, the ultrametrics of Proposition \ref{prop:SFT_groupoids_Lipschitz} allow us to build (roughly speaking) dense Lipschitz $*$-subalgebras of stabilised Cuntz-Krieger algebras.

\begin{prop}\label{prop:Lip_alg_Smale_groupoids_SFT}
Suppose that $(X,d,\varphi)$ is an irreducible TMC. For all $\kappa >1$, the vector space $\Lip_c(G^s(Q),D_{s,\kappa})$ forms a dense $*$-subalgebra of $\mathcal{S}(Q)$. Moreover, it is $\alpha_s$-invariant, and therefore, the algebraic crossed product $$\Lambda_{s,\kappa}(Q,\alpha_s)=  \Lip_c(G^s(Q),D_{s,\kappa})\rtimes_{\alpha_s,\text{alg}} \mathbb Z$$ 
is a well-defined dense $*$-subalgebra of $\mathcal{R}^s(Q)$. Similarly, in the unstable case we obtain the dense $*$-subalgebra $$\Lambda_{u,\kappa}(P,\alpha_u)=  \Lip_c(G^u(P),D_{u,\kappa})\rtimes_{\alpha_u,\text{alg}} \mathbb Z$$ of the Ruelle algebra $\mathcal{R}^u(P)$.
\end{prop}

Recall now the KPW-extension $\tau_{\Delta}:\mathcal{R}^s(Q)\otimes \mathcal{R}^u(P)\to \mathcal{Q}(\mathscr{H}\otimes \ell^2(\mathbb Z))$ formed by the faithful representations $\overline{\rho_{s}}:\mathcal{R}^s(Q)\to \mathcal{B}(\mathscr{H}\otimes \ell^2(\mathbb Z))$ and $\overline{\rho_{u}}:\mathcal{R}^u(P)\to \mathcal{B}(\mathscr{H}\otimes \ell^2(\mathbb Z))$ which commute modulo $\mathcal{K}(\mathscr{H}\otimes \ell^2(\mathbb Z))$, where the Hilbert space $\mathscr{H}=\ell^2(X^h(P,Q)).$ For details see Subsection \ref{sec:K-duality_Ruelle}.

\begin{remark}\label{rem:reduction_algebraic_computations}
Let $d'\in \text{sM}_d(X,\varphi)$. A main goal is to study $\tau_{\Delta}$ on the dense $*$-subalgebra $\Lambda_{s,d'}(Q,\alpha_s)\otimes_{\text{alg}} \Lambda_{u,d'}(P,\alpha_u)$. Due to linearity of $\overline{\rho_s}\cdot \overline{\rho_u}$, our arguments can be simplified. First, since the tensor product is algebraic, we can simply work on elementary tensors $x\otimes y$, for $x\in \Lambda_{s,d'}(Q,\alpha_s),\, y\in \Lambda_{u,d'}(P,\alpha_u).$ Further, since both factors of the tensor algebra are algebraic crossed products, it suffices to work on generators of the form $au^j\otimes bu^{j'}$, where $a\in \Lip_c(G^s(Q),D_{s,d'}),\, b\in \Lip_c(G^u(P),D_{u,d'})$ and $j,j'\in \mathbb Z$. Finally, since the groupoids are {\'e}tale, we can also assume that $a$ and $b$ are supported on bisections. Of course, the same reductions hold for all $d'\in \text{sM}_d(X,\varphi)$, and the algebras $\Lambda_{s,\kappa}(Q,\alpha_s)\otimes_{\text{alg}} \Lambda_{u,\kappa}(P,\alpha_u)$, where $\kappa>1$. 
\end{remark}

At this point our intentions should be clear. All these Lipschitz algebras will be shown to extend to smooth subalgebras of Ruelle algebras on which the $\Kt$-homology of Ruelle algebras is uniformly finitely summable. This is also related to the study of smoothness of the KPW-extension. To this end, we aim to investigate how much the algebras $\overline{\rho_{s}}(\Lambda_{s,d'}(Q,\alpha_s))$ and $\overline{\rho_{u}}(\Lambda_{u,d'}(P,\alpha_u))$ commute, for $d'\in \text{sM}_d(X,\varphi)$. This will yield results for general Smale spaces. Now, if $(X,d,\varphi)$ is a TMC, we will be able to obtain sharp estimates by studying how much the algebras $\overline{\rho_{s}}(\Lambda_{s,\kappa}(Q,\alpha_s))$ and $\overline{\rho_{u}}(\Lambda_{u,\kappa}(P,\alpha_u))$ commute, for $\kappa >1$.

\begin{lemma}\label{lem:singular_values}
Consider a compact operator $T=\bigoplus_{n\in \mathbb N} T_n\in \mathcal{B}(\bigoplus_{n\in \mathbb N} H_n)$. Assume that there are constants $C_1,C_2>0$ so that, for every $\varepsilon >0$ there are $n_0\in \mathbb N$ and $\alpha_{\varepsilon}, \beta_{\varepsilon}>1$ such that, for all $n\geq n_0$, it holds
\begin{enumerate}[(1)]
\item $\rank(T_n)\leq C_1 \alpha_{\varepsilon}^n$;
\item $\|T_n\|\leq C_2\beta_{\varepsilon}^{-n}.$
\end{enumerate}
Then, for every $\varepsilon>0$ it holds $s_n(T)=O(n^{-\log_{\alpha_{\varepsilon}} \beta_{\varepsilon}})$. Consequently, if we also assume that $\alpha_{\varepsilon},\beta_{\varepsilon}$ converge to $\alpha,\beta >1$ as $\varepsilon$ approaches zero, then $(s_n(T))_{n\in \mathbb N}\in \ell^p(\mathbb N)$, for every $p>\log_{\beta}\alpha.$ Moreover, if (2) is replaced by $\|T_n\|\leq C_2n^{-\gamma_{\varepsilon}}$ for some $\gamma_{\varepsilon}>0$ then, for every $\varepsilon>0$ we have $s_n(T)=O((\log n)^{-\gamma_{\varepsilon}})$. 
\end{lemma}

\begin{proof}
Let $\varepsilon>0$ and consider $n_0\in \mathbb N$ and $\alpha_{\varepsilon}, \beta_{\varepsilon}>1$ as in the statement. For every $n\in \mathbb N$, define $R_n=\sum_{i=1}^n \rank(T_i)+1$. The claim is that, for every $n\geq n_0$, we have $$s_{R_n}(T)\leq C_2\beta_{\varepsilon}^{-n-1}.$$ Indeed, from \cite[Chapter II]{GK} we know that, for any two compact operators $W_1,W_2$ acting on the same Hilbert space, the singular values satisfy $$s_{n+m-1}(W_1+W_2)\leq s_n(W_1)+s_m(W_2),$$ for all $n,m\in \mathbb N$. So, if $W_1=\bigoplus_{i=1}^{n} T_i$ and $W_2=\bigoplus_{i=n+1}^{\infty} T_i$ are seen as operators in $\mathcal{B}(\bigoplus_{i\in \mathbb N} H_i)$, then $W_1+W_2=T$ and $$s_{R_n}(W_1)=0,\, s_1(W_2)\leq C_2\beta_{\varepsilon}^{-n-1}.$$ 

Also, for $n\geq n_0$, we have that $$R_n\leq R_{n_0-1}+C_1\sum_{j=n_0}^{n}\alpha_{\varepsilon}^j.$$ For simplicity, the right hand side can be written as $R'_n=Q_{n_0}+C_1\alpha_{\varepsilon}^{n+1}/(\alpha_{\varepsilon}-1)$, where $Q_{n_0}=R_{n_0-1}-C_1\alpha_{\varepsilon}^{n_0}/(\alpha_{\varepsilon}-1)$. Therefore, for every $n\geq n_0$ we have that 
\begin{equation}\label{eq:singular_values}
s_{R'_n}(T)\leq s_{R_n}(T)\leq C_2\beta_{\varepsilon}^{-n-1}.
\end{equation}
Since $(R_n')_{n\geq n_0}$ is increasing to infinity, for every $m\geq R_{n_0}'$ we can find $n\geq n_0$ such that $R'_n\leq m < R'_{n+1}$. Note that $m>Q_{n_0}$, and since $m < R'_{n+1}$ we get that $$n> \log_{\alpha_{\varepsilon}}(m-Q_{n_0})+\log_{\alpha_{\varepsilon}}((\alpha_{\varepsilon}-1)/C_1)-2.$$ Finally, from $R'_n\leq m$ and (\ref{eq:singular_values}) we obtain that
\begin{align*}
s_m(T) &\leq C_2\beta_{\varepsilon}^{-n-1}\\
&\leq C_2C_{\alpha_{\varepsilon},\beta_{\varepsilon}}\beta_{\varepsilon}^{-\log_{\alpha_{\varepsilon}}(m-Q_{n_0})}\\
&= C_2C_{\alpha_{\varepsilon},\beta_{\varepsilon}} (m-Q_{n_0})^{-\log_{\alpha_{\varepsilon}}\beta_{\varepsilon}},
\end{align*}
where $C_{\alpha_{\varepsilon},\beta_{\varepsilon}}=\beta_{\varepsilon}^{-\log_{\alpha_{\varepsilon}}((\alpha_{\varepsilon}-1)/C_1)+1}.$ 

To summarise, for every $\varepsilon>0$ there are $n_{\varepsilon}\in \mathbb N$ and $C_{\varepsilon}>0$ such that, for all $n\geq n_{\varepsilon}$, we have $$s_n(T)\leq C_{\varepsilon}n^{-\log_{\alpha_{\varepsilon}}\beta_{\varepsilon}}.$$ This means, for all $\varepsilon>0$ and $p>\log_{\beta_{\varepsilon}}\alpha_{\varepsilon}$ one has $(s_n(T))_{n\in \mathbb N}\in \ell^p(\mathbb N)$. As a result, if $\alpha_{\varepsilon},\beta_{\varepsilon}$ converge to $\alpha,\beta >1$ as $\varepsilon$ approaches zero, then for every $p>\log_{\beta}\alpha$ we have $(s_n(T))_{n\in \mathbb N}\in \ell^p(\mathbb N)$. Finally, the case where $(\|T_n\|)_{n\in \mathbb N}$ goes polynomially fast to zero is dealt in exactly the same way. Just note that $(\log_{a_{\varepsilon}} n)^{-\gamma_{\varepsilon}}=O((\log n)^{-\gamma_{\varepsilon}}).$
\end{proof}

We are now in position to study commutation relations between the aforementioned stable and unstable Lipschitz algebras. Lemma \ref{lem:singular_values} suggests that the Hilbert space $\mathscr{H}\otimes \ell^2(\mathbb Z)$ should also be seen as $\bigoplus_{n\in \mathbb Z} \mathscr{H}$, and that it suffices to estimate the ranks and norms of the off-diagonal entries of the operators $\overline{\rho_s}(a u^j)\overline{\rho_u}(b u^{j'})-\overline{\rho_u}(b u^{j'})\overline{\rho_s}(a u^j)$ in $\mathcal{K}(\mathscr{H}\otimes \ell^2(\mathbb Z))$ that appear in (\ref{eq: tau_Delta}). We begin with the following lemma that is derived from \cite[Theorem 2.3]{KilPut}.

\begin{lemma}\label{lem:stable_unstable_intersection_KP}
Let $(Y,\psi)$ be a mixing Smale space and consider $x,y\in Y$. Let $B\subset X^u(x)$ and $C\subset X^s(y)$ be open with compact closure. Then, for every $\varepsilon >0$ there is $k_0\in \mathbb N$ such that, for all $k\geq k_0$, we have $$\# \psi^k(B)\cap C < e^{(\ent(\psi)+\varepsilon)k}.$$
\end{lemma}

Recall Smale's Decomposition Theorem \ref{thm: Smale decomposition}. Here we will use a different notation. Let $X_1,\ldots , X_M$ be the decomposition of the irreducible $(X,\varphi)$ into disjoint clopen sets which are cyclically permuted by $\varphi$, and where $\varphi^M|_{X_i}$ is mixing, for all $1\leq i\leq M$. The next result does not dependent on any choice of hyperbolic metric in $\text{hM}_d(X,\varphi)$, and this is why we do not highlight any metric on $(X,\varphi)$. 

\begin{lemma}\label{lem:Rank_Ruelle_alg}
Let $a\in C_c(G^s(Q))$ and $b\in C_c(G^u(P))$. Then, for every $n\in \mathbb Z$ the operator $\alpha_s^n(a)b$ has finite rank. Moreover, 
\begin{enumerate}[(i)]
\item there exists $n_0\in \mathbb N$ such that $\rank(\alpha_s^n(a)b)=0$, for every $n\leq - n_0$;
\item there exists a constant $C_1>0$ so that, for every $\varepsilon >0$ there is $n_1\in \mathbb N$ such that, for all $n\geq n_1$, we have $$\rank(\alpha_s^n(a)b)<C_1e^{(\ent(\varphi)+\varepsilon)n}.$$
\end{enumerate}
Similarly for the operators  $b\alpha_s^n(a)$, where $n\in \mathbb Z$.
\end{lemma}

\begin{proof}
Using a partition of unity argument (see Lemma \ref{lem:Lip_span}) we can decompose $a,b$ into finite sums of compactly supported functions $$a=\sum_{m=1}^{m'} a_m,\, b= \sum_{\ell=1}^{\ell'} b_{\ell},$$ with $\supp(a_m)\subset V^s(v_m,w_m,h^s_m,\eta_m,N_m)$ and $\supp(b_{\ell})\subset V^u(v'_{\ell},w'_{\ell},h^u_{\ell},\eta_{\ell}',N'_{\ell})$. Also, for every $n\in \mathbb Z$ we have $\alpha_s^n(a)\in C_c(G^s(Q))$. Therefore, we can use the same partition of unity argument on $\alpha_s^n(a)$, and from Lemma \ref{lem:stableunstablerankone} we obtain that $\alpha_s^n(a)b$ has finite rank. Moreover, using the above decomposition of $a,b$ and Lemma \ref{lem:minusinftylimit} we obtain part (i).

Let $\varepsilon>0$ and to prove part (ii) it suffices to show that, for every $m,\ell$ there is some $n(a_m,b_\ell,\varepsilon)$ so that, for all $n\geq n(a_m,b_\ell,\varepsilon)$, we have $$\rank(\alpha_s^n(a_m)b_{\ell})<e^{(\ent(\varphi)+\varepsilon)n}.$$ The result follows with $C_1=m'\ell'$ and $n_1$ being the maximum among all $n(a_m,b_\ell,\varepsilon)$. Indeed, fix some $n\geq0$ and $m,\ell$, and if $x\in X^s(w'_{\ell},\eta'_{\ell}),\, \varphi^{-n}\circ h^u_{\ell} (x)\in X^u(w_m, \eta_m)$ we have that $$\alpha_s^n(a_m) b_{\ell} \delta_x= a_m(h^s_m\circ \varphi^{-n}\circ h^u_{\ell} (x), \varphi^{-n}\circ h^u_{\ell} (x))b_{\ell}(h^u_{\ell}(x),x)\delta_{\varphi^n \circ h^s_m\circ \varphi^{-n}\circ h^u_{\ell} (x)},$$ otherwise $\alpha_s^n(a_m) b_{\ell} \delta_x=0$. For brevity let $B_m=X^u(w_m,\eta_m), C_{\ell}=X^s(v'_{\ell},\varepsilon_X/2)$ and so, if $\alpha_s^n(a_m) b_{\ell} \delta_x\neq 0$, then $$h^u_{\ell}(x)\in \varphi^n(B_m)\cap C_{\ell}.$$ Since $\varphi^n(B_m)\cap C_{\ell}$ is finite and $h^u_{\ell}$ is injective it holds that $$\rank(\alpha_s^n(a_m)b_{\ell})\leq \#\varphi^n(B_m)\cap C_{\ell}.$$

Let $j_m\geq 0$ be the first time that $\varphi^{j_m}(w_m)$ lies in the same mixing component, say $X_j$, with $v'_{\ell}$. Then, for every $k\geq 0$ the points $$\varphi^{j_m+kM}(w_m)\in X_j.$$ If $n\not \equiv j_m \mod M$, then $\varphi^n(B_m)\cap C_{\ell}=\varnothing$ and hence $\alpha_s^n(a_m) b_{\ell}=0$. Now, if $n=j_m+kM$, for some $k\geq 0$, we can write $$\varphi^n(B_m)\cap C_{\ell}= (\varphi^M|_{X_j})^k(\varphi^{j_m}(B_m))\cap C_{\ell}.$$ Using Lemma \ref{lem:stable_unstable_intersection_KP} for $\psi= \varphi^M|_{X_j}, B= \varphi^{j_m}(B_m)$ and $C=C_{\ell}$ we obtain some $k_0\in \mathbb N$ so that, if $k\geq k_0$, then  
\begin{align*}
\#\varphi^n(B_m)\cap C_{\ell}&= \# (\varphi^M|_{X_j})^k(\varphi^{j_m}(B_m))\cap C_{\ell}\\
&< e^{(\ent(\varphi^M|_{X_j})+\varepsilon)k}\\
&\leq e^{(\ent(\varphi^M|_{X_j})+\varepsilon)(n/M)}.
\end{align*}
Since $ \ent(\varphi^M|_{X_j})=M\ent(\varphi)$, we can define $n(a_m,b_\ell,\varepsilon)=j_m+k_0M$ and the proof of part (ii) is complete.
\end{proof}

\begin{lemma}\label{lem:KPW_norms}
Consider a metric $d'\in \text{sM}_d(X,\varphi)$ and let $a\in \Lip_c(G^s(Q),D_{s,d'})$ and $b\in \Lip_c(G^u(P),D_{u,d'})$. Then, there exist $C_2>0$ and $n_2\in \mathbb N$ such that, for all $n\geq n_2$, it holds $$\|\alpha_s^n(a)\alpha_u^{-n}(b)- \alpha_u^{-n}(b)\alpha_s^n(a)\|\leq C_22^{-n/\lceil \log_{\lambda_{X,d'}}3 \rceil}.$$ As a result, the constants $C_2$ and $n_2$ can be chosen bigger so that for $n\geq n_2$ one has $$\|\alpha_s^n(a)b- b\alpha_s^n(a)\|\leq C_22^{-(n/2)/\lceil \log_{\lambda_{X,d'}}3 \rceil}.$$ 
\end{lemma}

\begin{proof}
All notation is relative to $d'$ and for brevity we omit to indicate it, except for the case of the contraction constant $\lambda_{X,d'}>1$. First assume that $\supp(a)\subset V^s(v,w,h^s,\eta,N)$ and $\supp(b)\subset V^u(v',w',h^u,\eta',N')$. As in Lemma \ref{lem:Rank_Ruelle_alg} we see that, if $\alpha_s^n(a)\alpha_u^{-n}(b)\neq 0$, then $X^s(\varphi^{-n}(v'))\cap X^u(\varphi^n(w))\neq \varnothing.$ In particular, $\varphi^{-n}(v')$ and $\varphi^n(w)$ are in the same mixing component, otherwise $\alpha_s^n(a)\alpha_u^{-n}(b)=0$. Therefore, the possibility that $\alpha_s^n(a)\alpha_u^{-n}(b)\neq 0$ occurs periodically. 

Similarly, if $\alpha_u^{-n}(b)\alpha_s^n(a)\neq 0$, then $X^s(\varphi^{-n}(w'))\cap X^u (\varphi^n(v))\neq \varnothing$, hence $\varphi^{-n}(w'),$ $\varphi^n(v)$ lie in the same mixing component. But this can also happen if $\alpha_s^n(a)\alpha_u^{-n}(b)\neq 0$, since the points $\varphi^{-n}(w'),\,\varphi^n(v)$ are always in the same mixing component with $\varphi^{-n}(v'),\,\varphi^n(w)$, respectively. This last argument works symmetrically, and to summarise, the only chance for the operator $\alpha_s^n(a)\alpha_u^{-n}(b)- \alpha_u^{-n}(b)\alpha_s^n(a)$ to be non-zero is when $n$ lies in a certain arithmetic-like strictly increasing sequence. Therefore, for the sake of brevity, we can simply assume that $(X,\varphi)$ is mixing.

We have that, if 
\begin{equation}\label{eq:KPW_norms_1}
\varphi^n(x)\in X^s(w',\eta'), \, \varphi^{-2n}\circ h^u \circ \varphi^n (x)\in X^u(w,\eta),
\end{equation}
then $\alpha_s^n(a)\alpha_u^{-n}(b)\delta_x$ equals
$$
a(h^s\circ \varphi^{-2n}\circ h^u \circ \varphi^n (x), \varphi^{-2n}\circ h^u \circ \varphi^n (x)) b(h^u \circ \varphi^n (x), \varphi^n(x))\delta_{\varphi^n\circ h^s\circ \varphi^{-2n}\circ h^u \circ \varphi^n (x)},
$$
otherwise $\alpha_s^n(a)\alpha_u^{-n}(b)\delta_x=0$. Similarly, if 
\begin{equation}\label{eq:KPW_norms_2}
\varphi^{-n}(x)\in X^u(w,\eta), \, \varphi^{2n}\circ h^s \circ \varphi^{-n}(x)\in X^s(w',\eta'),
\end{equation}
then $\alpha_u^{-n}(b)\alpha_s^n(a)\delta_x$ equals
$$b(h^u\circ \varphi^{2n}\circ h^s \circ \varphi^{-n}(x), \varphi^{2n}\circ h^s \circ \varphi^{-n}(x)) a( h^s \circ \varphi^{-n}(x), \varphi^{-n}(x)) \delta_{\varphi^{-n}\circ h^u\circ \varphi^{2n}\circ h^s \circ \varphi^{-n}(x)},$$
and $\alpha_u^{-n}(b)\alpha_s^n(a)\delta_x=0$ otherwise.

Let now $\text{source}(a)$ and $\text{source}(b)$ denote the images of $\supp(a),\, \supp(b)$ via the groupoid source map. Then, $\text{source}(a)$ is a compact subset of $X^u(w,\eta)$ and $\text{source}(b)$ is a compact subset of $X^s(w',\eta')$. From compactness, there is an $\varepsilon>0$ such that the $\varepsilon$-neighbourhoods $N_{\varepsilon}(\text{source}(a))\subset X^u(w,\eta)$ and $N_{\varepsilon}(\text{source}(b))\subset X^s(w',\eta')$. We claim that there is $n_2\in \mathbb N$ such that, for every $n\geq n_2$, if $\alpha_s^n(a)\alpha_u^{-n}(b)\delta_x\neq 0$, then (\ref{eq:KPW_norms_2}) holds and hence the above formula for $\alpha_u^{-n}(b)\alpha_s^n(a)\delta_x$ is valid. However, this does not necessarily mean that $\alpha_u^{-n}(b)\alpha_s^n(a)\delta_x\neq 0.$ The same can be said for $\alpha_s^n(a)\alpha_u^{-n}(b)\delta_x$, if $\alpha_u^{-n}(b)\alpha_s^n(a)\delta_x\neq 0$. Indeed, let $n_2\geq N,N'$ and such that $\lambda_{X,d'}^{-2n_2+N}\varepsilon_X/2,\, \lambda_{X,d'}^{-2n_2+N'}\varepsilon_X/2 <\varepsilon$, and for $n\geq n_2$ suppose that $\alpha_s^n(a)\alpha_u^{-n}(b)\delta_x\neq 0$. Then, $$\varphi^n(x)\in \text{source}(b),\, \varphi^{-2n}\circ h^u \circ \varphi^n (x)\in \text{source}(a),$$ and since 
\begin{equation}\label{eq:KPW_norms_3}
\varphi^{-2n}\circ h^u \circ \varphi^n (x)\in X^u(\varphi^{-n}(x),\lambda_{X,d'}^{-2n+N'}\varepsilon_X/2),
\end{equation}
we obtain $\varphi^{-n}(x)\in N_{\varepsilon}(\text{source}(a))\subset  X^u(w,\eta)$. Hence, the expression $\varphi^{2n}\circ h^s \circ \varphi^{-n}(x)$ is well-defined and also 
\begin{equation}\label{eq:KPW_norms_4}
\varphi^{2n}\circ h^s \circ \varphi^{-n}(x)\in X^s(\varphi^n(x), \lambda_{X,d'}^{-2n+N}\varepsilon_X/2),
\end{equation}
meaning that $\varphi^{2n}\circ h^s \circ \varphi^{-n}(x)\in N_{\varepsilon}(\text{source}(b))\subset X^s(w',\eta')$. This proves the claim. Moreover, from \cite[Lemma 2.2]{Putnam_algebras}, we can actually choose $n_2$ big enough so that, for every $n\geq n_2$, if $\alpha_s^n(a)\alpha_u^{-n}(b)\delta_x\neq 0$ or $\alpha_u^{-n}(b)\alpha_s^n(a)\delta_x\neq 0$ then, $$\varphi^n\circ h^s\circ \varphi^{-2n}\circ h^u \circ \varphi^n (x)= \varphi^{-n}\circ h^u\circ \varphi^{2n}\circ h^s \circ \varphi^{-n}(x).$$

For $n\geq n_2$ and $x\in X^h(P,Q)$ such that $\alpha_s^n(a)\alpha_u^{-n}(b)\delta_x\neq 0$ or $\alpha_u^{-n}(b)\alpha_s^n(a)\delta_x\neq 0$, define 
\begin{align*}
x_1&=x,\\
x_2&= \varphi^n\circ h^s \circ \varphi^{-n}(x),\\
x_3&= \varphi^n\circ h^s\circ \varphi^{-2n}\circ h^u \circ \varphi^n (x)= \varphi^{-n}\circ h^u\circ \varphi^{2n}\circ h^s \circ \varphi^{-n}(x),\\
x_4&= \varphi^{-n}\circ h^u\circ \varphi^n(x).
\end{align*}
For an even bigger $n_2$ we can guarantee that 
\begin{align*}
x_1&\in X^s(x_2, \varepsilon_X'/2),\\
x_3&\in X^s(x_4, \varepsilon_X'/2),\\
x_1&\in X^u(x_4,\varepsilon_X'/2),\\
x_3&\in X^u(x_2,\varepsilon_X'/2).
\end{align*}
From the discussion so far, and since $\alpha_s^n(a)\alpha_u^{-n}(b)- \alpha_u^{-n}(b)\alpha_s^n(a)$ takes base vectors to base vectors, it holds that, for $n\geq n_2$, $$\|\alpha_s^n(a)\alpha_u^{-n}(b)- \alpha_u^{-n}(b)\alpha_s^n(a)\|$$ $$=\sup_x|a(\varphi^{-n}(x_3),\varphi^{-n}(x_4))b(\varphi^{n}(x_4),\varphi^n(x_1))-b(\varphi^n(x_3),\varphi^n(x_2))a(\varphi^{-n}(x_2),\varphi^{-n}(x_1))|,$$ where the supremum is taken over all $x\in X^h(P,Q)$ such that $\alpha_s^n(a)\alpha_u^{-n}(b)\delta_x\neq 0$ or $\alpha_u^{-n}(b)\alpha_s^n(a)\delta_x\neq 0$. To estimate the norm, let us denote $e=(\varphi^{-n}(x_3),\varphi^{-n}(x_4)),$ $f=(\varphi^{-n}(x_2),\varphi^{-n}(x_1)),\, e'= (\varphi^{n}(x_4),\varphi^n(x_1))$ and $f'=(\varphi^n(x_3),\varphi^n(x_2))$. Then, since $a,b$ are Lipschitz, we have that 
\begin{align*}
|a(e)-a(f)|&\leq \Lip(a) D_{s,d'}(e,f),\\
|b(e')-b(f')|&\leq \Lip(b) D_{u,d'}(e',f').
\end{align*}
Again, if we consider a bigger $n_2$ so that $\lambda_{X,d'}^{-n_2+N}\varepsilon_X/2,\, \lambda_{X,d'}^{-n_2+N'}\varepsilon_X/2\leq \varepsilon_X'/2$ then, for every $n\geq n_2$, the base sets $V^u(f',\lambda_{X,d'}^{-2n+N}\varepsilon_X/2,n),\, V^s(f, \lambda_{X,d'}^{-2n+N'}\varepsilon_X/2,n)$ are well-defined, and from (\ref{eq:KPW_norms_3}), (\ref{eq:KPW_norms_4}) they contain $e'$ and $e$, respectively. Now, from Proposition \ref{prop:groupoiddistances} we can find $\gamma, \gamma'>0$ (independent of $e,f$ and $e',f'$) such that 
\begin{align*}
D_{s,d'}(e,f)&\leq  2^{-n/\lceil \log_{\lambda_{X,d'}}3 \rceil}\gamma,\\
D_{u,d'}(e',f')&\leq 2^{-n/\lceil \log_{\lambda_{X,d'}}3 \rceil}\gamma'.
\end{align*}
Using the fact that $a,b$ are compactly supported, we obtain $C_1>0$ such that, for every $n\geq n_2$, it holds 
\begin{equation}\label{eq:KPW_norms_5}
\|\alpha_s^n(a)\alpha_u^{-n}(b)- \alpha_u^{-n}(b)\alpha_s^n(a)\|\leq C_12^{-n/\lceil \log_{\lambda_{X,d'}}3 \rceil}.
\end{equation}
More generally, if $a,b$ are not necessarily supported on bisections, then we can use Lemma \ref{lem:Lip_span} and obtain a constant $C_2>0$ such that (\ref{eq:KPW_norms_5}) holds with $C_2$ instead of $C_1$.

Finally, it is not hard to see that $$\|\alpha_s^{2n}(a)b- b\alpha_s^{2n}(a)\|=\|\alpha_s^n(a)\alpha_u^{-n}(b)- \alpha_u^{-n}(b)\alpha_s^n(a)\|,$$ and hence 
\begin{equation}\label{eq:KPW_norms_6}
\|\alpha_s^{2n}(a)b- b\alpha_s^{2n}(a)\|\leq C_22^{-n/\lceil \log_{\lambda_{X,d'}}3 \rceil},
\end{equation}
for all $n\geq n_2$. To complete the proof, observe that (\ref{eq:KPW_norms_6}) can be also applied to $\alpha_s(a)$ in place of $a$.
\end{proof}

In exactly the same way, using Proposition \ref{prop:SFT_groupoid_distances}, one can prove the following.

\begin{lemma}\label{lem:KPW_norms_SFT}
Suppose that $(X,d,\varphi)$ is an irreducible TMC and let $\kappa >1$. Then, for every $a\in \Lip_c(G^s(Q),D_{s,\kappa})$ and $b\in \Lip_c(G^u(P),D_{u,\kappa})$, there exist $C_2>0$ and $n_2\in \mathbb N$ such that, for all $n\geq n_2$, it holds $$\|\alpha_s^n(a)\alpha_u^{-n}(b)- \alpha_u^{-n}(b)\alpha_s^n(a)\|\leq C_2\kappa^{-2n}.$$ As a result, the constants $C_2$ and $n_2$ can be chosen bigger so that $$\|\alpha_s^n(a)b- b\alpha_s^n(a)\|\leq C_2\kappa^{-n},$$ for all $n\geq n_2$.
\end{lemma}

We now state the main results of this subsection.

\begin{prop}\label{prop:com_Lip_alg_general}
Let $d'\in \text{sM}_d(X,\varphi)$. The algebras $\overline{\rho_{s}}(\Lambda_{s,d'}(Q,\alpha_s)),\, \overline{\rho_{u}}(\Lambda_{u,d'}(P,\alpha_u))$ commute modulo the Schatten ideal $\mathcal{L}^p(\mathscr{H}\otimes \ell^2(\mathbb Z))$, for every $$p> \frac{2\ent(\varphi)\lceil \log_{\lambda_{X,d'}}3 \rceil}{\log 2}.$$
\end{prop}

\begin{proof}
To simplify the notation, let $p(d')= 2\ent(\varphi)\lceil \log_{\lambda_{X,d'}}3\rceil/\log 2.$ It suffices to show that, for every $a\in \Lip_c(G^s(Q),D_{s,d'}),\, b\in \Lip_c(G^u(P),D_{u,d'})$ and $j,j'\in \mathbb Z$, it holds $$[\overline{\rho_s}(a u^j),\overline{\rho_u}(b u^{j'})]\in \mathcal{L}^p(\mathscr{H}\otimes \ell^2(\mathbb Z)),$$ for every $p>p(d')$. By the construction of $\overline{\rho_s}$ and $\overline{\rho_u}$ we have that $$[\overline{\rho_s}(a u^j),\overline{\rho_u}(b u^{j'})]=[\overline{\rho_s}(a),\overline{\rho_u}(b)]\overline{\rho_s}( u^j)\overline{\rho_u}(u^{j'}),$$ and therefore it suffices to show that the singular values of the compact operator $$R=[\overline{\rho_s}(a),\overline{\rho_u}(b)]$$ satisfy $(s_n(R))_{n\in \mathbb N}\in \ell^p(\mathbb N)$, for every $p>p(d')$.

We have that $R=\bigoplus_{n\in \mathbb Z}R_n$, where $R_n=\alpha_s^n(a)b-b\alpha_s^n(a).$ From Lemma \ref{lem:Rank_Ruelle_alg} we obtain that each $R_n$ has finite rank and that there is some $n_0\in \mathbb N$ such that, for all $n\leq -n_0$, it holds $R_n=0$. For this reason, consider the compact operator $T=\bigoplus_{n\in \mathbb N}T_n\in \mathcal{B}(\bigoplus_{n\in \mathbb N} \mathscr{H})$, where each $T_n=R_{n-n_0}$, and using Lemma \ref{lem:singular_values} it suffices to show that $(s_n(T))_{n\in \mathbb N}\in \ell^p(\mathbb N)$, for every $p>p(d')$. 

From Lemma \ref{lem:Rank_Ruelle_alg} we can find a constant $C_1>0$ so that, for every $\varepsilon>0$ there is $n_1\in \mathbb N$ such that, for all $n\geq n_1$, it holds 
\begin{equation}\label{eq:com_Lip_alg_general_1}
\rank(T_n)\leq C_1e^{(\ent(\varphi)+\varepsilon)n}.
\end{equation}
Also, from Lemma \ref{lem:KPW_norms} we can find $C_2>0$ and $n_2\in \mathbb N$ such that, for all $n\geq n_2$, one has 
\begin{equation}\label{eq:com_Lip_alg_general_2}
\|T_n\|\leq C_22^{-(n/2)/\lceil \log_{\lambda_{X,d'}}3 \rceil}.
\end{equation}
Both (\ref{eq:com_Lip_alg_general_1}) and (\ref{eq:com_Lip_alg_general_2}) have been simplified to not include the integer $n_0$. Following Lemma \ref{lem:singular_values}, we have that $\alpha_{\varepsilon}=e^{\ent(\varphi)+\varepsilon},\,\alpha=e^{\ent(\varphi)}>1,\,\beta_{\varepsilon}=\beta=2^{(1/2)/\lceil \log_{\lambda_{X,d'}}3 \rceil}$, and also that, for every $p>\log_{\beta}\alpha$, the sequence $(s_n(T))_{n\in \mathbb N}\in \ell^p(\mathbb N)$. In fact, we have that $$\log_{\beta}\alpha= \frac{2\ent(\varphi)\lceil \log_{\lambda_{X,d'}}3 \rceil}{\log 2},$$ and the proof is complete.
\end{proof}

Similarly, one can show the following.

\begin{prop}\label{prop:com_Lip_alg_SFT}
Suppose that $(X,d,\varphi)$ is an irreducible TMC and let $\kappa >1$. The algebras $\overline{\rho_{s}}(\Lambda_{s,\kappa}(Q,\alpha_s))$ and $\overline{\rho_{u}}(\Lambda_{u,\kappa}(P,\alpha_u))$ commute modulo the Schatten ideal $\mathcal{L}^p(\mathscr{H}\otimes \ell^2(\mathbb Z))$, for every $$p> \frac{\ent(\varphi)}{\log \kappa}.$$
\end{prop}

\subsection{Main results}\label{sec:KdualitysmoothRuelle} In what follows, we use the $\lambda$-number $\lambda(X,\varphi)$ of $(X,d,\varphi)$, which is a topological invariant related to the family of metrics $\text{sM}_d(X,\varphi)$, see Subsection \ref{sec:Optimisation}. Also, for every $d'\in \text{sM}_d(X,\varphi)$ and $\kappa>1$, recall (from Subsection \ref{sec:Lip_Ruelle_Smooth_Ext}) the dense (Lipschitz) $*$-subalgebras $\Lambda_{s,d'}(Q,\alpha_s),\, \Lambda_{u,d'}(P,\alpha_u)$ and $\Lambda_{s,\kappa}(Q,\alpha_s),\, \Lambda_{u,\kappa}(P,\alpha_u)$ (the zero-dimensional case) of the Ruelle algebras $\mathcal{R}^s(Q)$, $\mathcal{R}^u(P)$, respectively. The proof of the next result is omitted since it is a straightforward application of Propositions \ref{prop:Lip_alg_Smale_groupoids}, \ref{prop:com_Lip_alg_general} and \ref{prop:Ext_smoothness_prop}.

\begin{thm}\label{thm:KPW_smoothness}
For every $d'\in \text{sM}_d(X,\varphi)$, there exist holomorphically stable dense $*\text{-subalgebras}\,$ $\Eta_{s,u,d'}(Q,\alpha_s)\subset \mathcal{R}^s(Q)$ and $\Eta_{u,s,d'}(P,\alpha_u)\subset \mathcal{R}^u(P)$ that
\begin{enumerate}[(1)]
\item contain $\Lambda_{s,d'}(Q,\alpha_s)$ and  $\Lambda_{u,d'}(P,\alpha_u)$, respectively;
\item the KPW-extension $\tau_{\Delta}:\mathcal{R}^s(Q)\otimes \mathcal{R}^u(P)\to \mathcal{Q}(\mathscr{H}\otimes \ell^2(\mathbb Z))$ is $p$-smooth on the algebra $\Eta_{s,u,d'}(Q,\alpha_s)\otimes_{\text{alg}} \Eta_{u,s,d'}(P,\alpha_u)$, for every 
\begin{equation}\label{eq:KPW_smoothness1}
p> \frac{2\ent(\varphi)\lceil \log_{\lambda_{X,d'}}3 \rceil}{\log 2}.
\end{equation}
\end{enumerate}
As a result, from a topological perspective, for every 
\begin{equation}\label{eq:KPW_smoothness2}
p> \frac{2\ent(\varphi)}{\log 2}+\frac{2\ent(\varphi)\log_{\lambda (X,\varphi)}3}{\log 2},
\end{equation}
there is $d'\in \text{sM}_d(X,\varphi)$ such that $\tau_{\Delta}$ is $p$-smooth on $\Eta_{s,u,d'}(Q,\alpha_s)\otimes_{\text{alg}} \Eta_{u,s,d'}(P,\alpha_u)$.
\end{thm}

\begin{remark}\label{rem:KPW_smooth_ext}
One of the indices of $\Eta_{s,u,d'}(Q,\alpha_s)\subset \mathcal{R}^s(Q)$ is the letter \enquote{u} which stands for \enquote{unstable}. Even though this subalgebra lies only in the stable Ruelle algebra $\mathcal{R}^s(Q)$, its construction depends on the Lipschitz algebra $\overline{\rho_{u}}(\Lambda_{u,d'}(P,\alpha_u))$ in $\mathcal{B}(\mathscr{H}\otimes \ell^2(\mathbb Z))$. For more details see the proof of Lemma \ref{lem:extend_almost_commuting_algebras}. Similarly for the subalgebra $\Eta_{u,s,d'}(P,\alpha_u)\subset \mathcal{R}^u(P)$. 
\end{remark}

\begin{remark}\label{rem:KPW_smooth_ext2}
Following Subsection \ref{sec:Optimisation}, in this generality of Smale spaces the degree of summability (\ref{eq:KPW_smoothness1}), for given $d'\in \text{sM}_d(X,\varphi)$, is the best possible. For a discussion on the infimum degree of summability see Remark \ref{rem:totalinf}. However, there is a subtlety. The ceiling function is not continuous at the integers, and hence we cannot derive, in general, that $\tau_{\Delta}$ is $p$-smooth for every 
\begin{equation}\label{eq:KPW_smooth_ext2}
p>\frac{2\ent(\varphi)\lceil \log_{\lambda(X,\varphi)}3 \rceil}{\log 2},
\end{equation}
where $\lambda(X,\varphi)=\sup \{\lambda_{X,d'}: d'\in \text{sM}_d(X,\varphi)\}$. This can be easily seen if $(X,d,\varphi)$ is zero-dimensional. In this case $\lambda(X,\varphi)=\infty$ and the smallest degree of summability (\ref{eq:KPW_smoothness1}) is $2\ent(\varphi)/\log 2$, while the one of (\ref{eq:KPW_smooth_ext2}) is zero. Equation (\ref{eq:KPW_smoothness2}) comes from the fact that $\lceil x \rceil < x+1$, for all $x\in \mathbb R$.
\end{remark}

From the $\Kt$-duality of Ruelle algebras (see Theorem \ref{thm:Ruelleduality}) and Propositions \ref{prop:Lip_alg_Smale_groupoids}, \ref{prop:com_Lip_alg_general} and \ref{prop:Ext_smoothness_prop}, we obtain the following important result.

\begin{thm}\label{cor:KPW_smoothness}
For all $d'\in \text{sM}_d(X,\varphi)$, the $\Kt$-homology of the Ruelle algebras $\mathcal{R}^s(Q)$ and $\mathcal{R}^u(P)$ is uniformly $\mathcal{L}^p$-summable on the holomorphically stable dense $*$-subalgebras $\Eta_{s,u,d'}(Q,\alpha_s)$ and $\Eta_{u,s,d'}(P,\alpha_u),$ respectively, for every $$p> \frac{2\ent(\varphi)\lceil \log_{\lambda_{X,d'}}3 \rceil}{\log 2}.$$
In particular, for every $$p> \frac{2\ent(\varphi)}{\log 2}+\frac{2\ent(\varphi)\log_{\lambda (X,\varphi)}3}{\log 2},$$ there is a metric $d'\in \text{sM}_d(X,\varphi)$ such that the $\Kt$-homology of $\mathcal{R}^s(Q)$ and $\mathcal{R}^u(P)$ is uniformly $\mathcal{L}^p$-summable on $\Eta_{s,u,d'}(Q,\alpha_s)$ and $\Eta_{u,s,d'}(P,\alpha_u),$ respectively.
\end{thm}

We now move on to sharpen the above results in the case $(X,d,\varphi)$ is a TMC. This is done using the ultrametrics $D_{s,\kappa},\, D_{u,\kappa}$ of Proposition \ref{prop:SFT_groupoids_Lipschitz}, which are indexed over every expanding factor $\kappa>1$. The next two results are strongly related to the fact that, in this case, the $\lambda$-number $\lambda(X,\varphi)$ is infinite. The proof of Theorem \ref{thm:KPW_smoothness_SFT} follows from Propositions \ref{prop:Lip_alg_Smale_groupoids_SFT}, \ref{prop:com_Lip_alg_SFT}, \ref{prop:Ext_smoothness_prop}, which by $\Kt$-duality also yield Theorem \ref{cor:KPW_smoothness_SFT}.

\begin{thm}\label{thm:KPW_smoothness_SFT}
Suppose that $(X,d,\varphi)$ is an irreducible TMC and let $\kappa >1$. There exist holomorphically stable dense $*$-subalgebras $\Eta_{s,u,\kappa}(Q,\alpha_s)\subset \mathcal{R}^s(Q)$, $\Eta_{u,s,\kappa}(P,\alpha_u)$ $\subset \mathcal{R}^u(P)$ that
\begin{enumerate}[(1)]
\item contain $\Lambda_{s,\kappa}(Q,\alpha_s)$ and  $\Lambda_{u,\kappa}(P,\alpha_u)$, respectively;
\item the KPW-extension $\tau_{\Delta}:\mathcal{R}^s(Q)\otimes \mathcal{R}^u(P)\to \mathcal{Q}(\mathscr{H}\otimes \ell^2(\mathbb Z))$ is $p$-smooth on the algebra $\Eta_{s,u,\kappa}(Q,\alpha_s)\otimes_{\text{alg}} \Eta_{u,s,\kappa}(P,\alpha_u)$, for every 
\begin{equation}\label{eq:KPW_smoothness_SFT_1}
p> \frac{\ent(\varphi)}{\log \kappa}.
\end{equation}
\end{enumerate}
As a result, the extension $\tau_{\Delta}$ is optimally smooth. More precisely, for every $p>0$ there is $\kappa>1$ such that $\tau_{\Delta}$ is $p$-smooth on $\Eta_{s,u,\kappa}(Q,\alpha_s)\otimes_{\text{alg}} \Eta_{u,s,\kappa}(P,\alpha_u)$.
\end{thm}

\begin{thm}\label{cor:KPW_smoothness_SFT}
Suppose that $(X,d,\varphi)$ is an irreducible TMC. Then, for every $\kappa >1$, the $\Kt$-homology of the Ruelle algebras $\mathcal{R}^s(Q)$ and $\mathcal{R}^u(P)$ is uniformly $\mathcal{L}^p$-summable on the holomorphically stable dense $*$-subalgebras $\Eta_{s,u,\kappa}(Q,\alpha_s)$ and $\Eta_{u,s,\kappa}(P,\alpha_u),$ for every $$p> \frac{\ent(\varphi)}{\log \kappa}.$$ In particular, for every $p>0$ there is $\kappa>1$ such that the $\Kt$-homology of $\mathcal{R}^s(Q)$ and $\mathcal{R}^u(P)$ is uniformly $\mathcal{L}^p$-summable on $\Eta_{s,u,\kappa}(Q,\alpha_s)$ and $\Eta_{u,s,\kappa}(P,\alpha_u),$ respectively.
\end{thm}

\begin{Acknowledgements}
My sincere appreciation goes to my advisors Mike Whittaker and Joachim Zacharias for their constant support during my doctoral studies. Also, I am grateful to Magnus Goffeng for many useful discussions on finite summability and index theory. Further, I would like to thank Siegfried Echterhoff, Heath Emerson, Bram Mesland, Ian Putnam and Christian Voigt for a number of useful conversations. I am also indebted to the anonymous referee for carefully reading the manuscript and for a number of helpful comments that improved it. Finally, I would like to thank EPSRC (grants NS09668/1, M5086056/1) as well as the London Mathematical Society and Heilbronn Institute for Mathematical Research (Early Career Fellowship) for supporting this research.
\end{Acknowledgements}

\end{document}